\documentclass[english]{article}
\usepackage[T1]{fontenc}
\usepackage[latin9]{inputenc}
\usepackage{geometry}
\usepackage{babel}
\usepackage{array}
\usepackage{varioref}
\usepackage{refstyle}
\usepackage{wrapfig}
\usepackage{units}
\usepackage{amsthm}
\usepackage{amsmath}
\usepackage{amssymb}
\usepackage{graphicx}
\usepackage[all]{xy}
\usepackage[unicode=true,pdfusetitle,
 bookmarks=true,bookmarksnumbered=false,bookmarksopen=false,
 breaklinks=false,pdfborder={0 0 1},backref=false,colorlinks=false]
 {hyperref}

\makeatletter

\AtBeginDocument{\providecommand\subref[1]{\ref{sub:#1}}}
\AtBeginDocument{\providecommand\secref[1]{\ref{sec:#1}}}
\AtBeginDocument{\providecommand\tabref[1]{\ref{tab:#1}}}
\AtBeginDocument{\providecommand\thmref[1]{\ref{thm:#1}}}
\AtBeginDocument{\providecommand\eqref[1]{\ref{eq:#1}}}
\AtBeginDocument{\providecommand\figref[1]{\ref{fig:#1}}}
\AtBeginDocument{\providecommand\claref[1]{\ref{cla:#1}}}
\AtBeginDocument{\providecommand\claimref[1]{\ref{claim:#1}}}
\AtBeginDocument{\providecommand\Claimref[1]{\ref{Claim:#1}}}
\AtBeginDocument{\providecommand\propref[1]{\ref{prop:#1}}}
\AtBeginDocument{\providecommand\enuref[1]{\ref{enu:#1}}}
\AtBeginDocument{\providecommand\corref[1]{\ref{cor:#1}}}
\providecommand{\tabularnewline}{\\}
\RS@ifundefined{subref}
  {\def\RSsubtxt{section~}\newref{sub}{name = \RSsubtxt}}
  {}
\RS@ifundefined{thmref}
  {\def\RSthmtxt{theorem~}\newref{thm}{name = \RSthmtxt}}
  {}
\RS@ifundefined{lemref}
  {\def\RSlemtxt{lemma~}\newref{lem}{name = \RSlemtxt}}
  {}

\numberwithin{equation}{section}
\numberwithin{figure}{section}
\theoremstyle{plain}
\newtheorem{thm}{\protect\theoremname}[section]
  \theoremstyle{plain}
  \newtheorem{prop}[thm]{\protect\propositionname}
  \theoremstyle{plain}
  \newtheorem{conjecture}[thm]{\protect\conjecturename}
  \theoremstyle{plain}
  \newtheorem{cor}[thm]{\protect\corollaryname}
  \theoremstyle{remark}
  \newtheorem{claim}[thm]{\protect\claimname}
  \theoremstyle{definition}
  \newtheorem{defn}[thm]{\protect\definitionname}
  \theoremstyle{plain}
  \newtheorem*{prop*}{\protect\propositionname}
  \theoremstyle{plain}
  \newtheorem*{thm*}{\protect\theoremname}
  \theoremstyle{remark}
  \newtheorem{rem}[thm]{\protect\remarkname}
  \theoremstyle{plain}
  \newtheorem{lem}[thm]{\protect\lemmaname}

\usepackage[leftcaption]{sidecap}
\usepackage{colortbl}
\usepackage{graphics}
\usepackage[perpage,symbol]{footmisc}
\usepackage{appendix}

\sidecaptionvpos{figure}{c}

\newcommand{\filleddiamond}{\raisebox{0.17\height}{\scalebox{0.7}[0.4]{$\vdots$}}}

\newcommand{\FigBesBeg}[1][1.0]{%
 \let\MyFigure\figure
 \let\MyEndfigure\endfigure
 \renewenvironment{figure}[1]{\begin{SCfigure}[#1]##1}{\end{SCfigure}}}

\newcommand{\FigBesEnd}{%
 \let\figure\MyFigure
 \let\endfigure\MyEndfigure}

\DefineFNsymbols*{lamportnostar}[math]{\dagger\ddagger\S\P\|{\dagger\dagger}{\ddagger\ddagger}}
\setfnsymbol{lamportnostar}

\newcommand{\F}{\mathbf{F}}

\renewcommand{\G}{\Gamma}

\newcommand{\ff}{\le_{ff}}

\newcommand{\fg}{\le_{fg}}

\newcommand{\alg}{\le_{alg}}
\newcommand{\neqalg}{\lneq_{alg}}

\newcommand{\rk}{\mathrm{rk}}
\newcommand{\cpt}{{\cal CW}_t}
\newcommand{\cptm}{{\cal CW}_t^m}
\newcommand{\pf}{{\mathfrak{ pf}}}

\newcommand{\XC}[2]{\left[#1,#2\right]_{\Xcov}}
\newcommand{\XCO}[2]{\left[#1,#2\right)_{\Xcov}}

\DeclareMathOperator{\Sym}{Sym}
\DeclareMathOperator{\crit}{Crit}

\usepackage{bbm}

\usepackage[all]{xy}

\newref{eq}{refcmd={(\ref{#1})}}
\newref{enu}{refcmd={(\ref{#1})}}
\newref{thm}{refcmd={\ref{#1}}}
\newref{lem}{refcmd={\ref{#1}}}
\newref{sec}{refcmd={\ref{#1}}}
\newref{sub}{refcmd={\ref{#1}}}
\newref{tab}{refcmd={\ref{#1}}}
\newref{fig}{refcmd={\ref{#1}}}

\newcommand{\Xcov}{\scriptscriptstyle{\stackrel{\twoheadrightarrow}{X}}}

\newcommand{\covers}{\mathrel{\leq_{\smash{\scalebox{0.9}[0.8]{\ensuremath{\Xcov}}}}}}

\theoremstyle{plain}
\newtheorem{claim}[thm]{\protect\claimname}

\usepackage[left,modulo]{lineno}

\makeatother

  \providecommand{\claimname}{Claim}
  \providecommand{\conjecturename}{Conjecture}
  \providecommand{\corollaryname}{Corollary}
  \providecommand{\definitionname}{Definition}
  \providecommand{\lemmaname}{Lemma}
  \providecommand{\propositionname}{Proposition}
  \providecommand{\remarkname}{Remark}
  \providecommand{\theoremname}{Theorem}
\providecommand{\theoremname}{Theorem}

\begin{document}

\title{Expansion of Random Graphs:\\
New Proofs, New Results}

\author{Doron Puder%
\thanks{Supported by Adams Fellowship Program of the Israel Academy of Sciences
and Humanities, and the ERC.%
} 
\\
{\small }\\
{\small Einstein Institute of Mathematics }\\
{\small{} Hebrew University, Jerusalem}\\
{\small{} }\texttt{\small doronpuder@gmail.com}}
\maketitle
\begin{abstract}
We present a new approach to showing that random graphs are nearly
optimal expanders. This approach is based on recent deep results in
combinatorial group theory. It applies to both regular and irregular
random graphs.

Let $\Gamma$ be a random d-regular graph on $n$ vertices, and let
$\lambda$ be the largest absolute value of a non-trivial eigenvalue
of its adjacency matrix. It was conjectured by Alon \cite{Alo86}
that a random $d$-regular graph is {}``almost Ramanujan'', in the
following sense: for every $\varepsilon>0$, $\lambda<2\sqrt{d-1}+\varepsilon$
asymptotically almost surely. Friedman famously presented a proof
of this conjecture in \cite{Fri08}. Here we suggest a new, substantially
simpler proof of a nearly-optimal result: we show that a random $d$-regular
graph satisfies $\lambda<2\sqrt{d-1}+1$ a.a.s.

A main advantage of our approach is that it is applicable to a generalized
conjecture: For $d$ even, a $d$-regular graph on $n$ vertices is
an $n$-covering space of a bouquet of $d/2$ loops. More generally,
fixing an arbitrary base graph $\Omega$, we study the spectrum of
$\Gamma$, a random $n$-covering of $\Omega$. Let $\lambda$ be
the largest absolute value of a non-trivial eigenvalue of $\Gamma$.
Extending Alon's conjecture to this more general model, Friedman \cite{Fri03}
conjectured that for every $\varepsilon>0,$ a.a.s.~$\lambda<\rho+\varepsilon$,
where $\rho$ is the spectral radius of the universal cover of $\Omega$.
When $\Omega$ is regular we get a bound of $\rho+0.84$, and for
an arbitrary $\Omega$, we prove a nearly optimal upper bound of $\sqrt{3}\rho$.
This is a substantial improvement upon all known results (by Friedman,
Linial-Puder, Lubetzky-Sudakov-Vu and Addario-Berry-Griffiths).
\end{abstract}
\tableofcontents{}

\section{Introduction\label{sec:Introduction}}

\subsection*{Random $d$-regular graphs\label{sub:intro-Random--Regular-Graphs}}

Let $\Gamma$ be a finite $d$-regular graph%
\footnote{Unless otherwise specified, a graph in this paper is undirected and
may contain loops and multiple edges. A graph without loops and without
multiple edges is called here \emph{simple}.%
} on $n$ vertices ($d\geq3$) and let $A_{\Gamma}$ be its adjacency
matrix. The \emph{spectrum} of $\Gamma$ is the spectrum of $A_{\Gamma}$
and consists of $n$ real eigenvalues, 
\[
d=\lambda_{1}\geq\lambda_{2}\geq\ldots\geq\lambda_{n}\geq-d.
\]
The eigenvalue $\lambda_{1}=d$ corresponds to constant functions
and is considered the \emph{trivial }eigenvalue of $\Gamma$. Let
\marginpar{$\lambda\left(\Gamma\right)$}$\lambda\left(\Gamma\right)$
be the largest absolute value of a non-trivial eigenvalue of $\Gamma$,
i.e.~$\lambda\left(\Gamma\right)=\max\left\{ \lambda_{2},-\lambda_{n}\right\} $.
This value measures the spectral expansion of the graph: the smaller
$\lambda\left(\Gamma\right)$ is, the better expander $\Gamma$ is
(see Appendix \ref{sec:Operators-on-Non-Regular} for details). 

The well-known Alon-Boppana bound states that $\lambda\left(\Gamma\right)\geq2\sqrt{d-1}-o_{n}\left(1\right)$
(\cite{Nil91}), bounding the spectral expansion of an infinite family
of $d$-regular graphs. There is no equivalent deterministic non-trivial
upper bound: for example, if $\Gamma$ is disconnected or bipartite
then $\lambda\left(\Gamma\right)=d$. However, Alon conjectured \cite[Conj.\ 5.1]{Alo86}
that if $\Gamma$ is a \emph{random} $d$-regular graph, then $\lambda\left(\Gamma\right)\leq2\sqrt{d-1}+o_{n}\left(1\right)$
a.a.s.\,(asymptotically almost surely, i.e.\,with probability tending
to 1 as $n\to\infty$) %
\footnote{In fact, Alon's original conjecture referred only to $\lambda_{2}\left(\Gamma\right)$,
the second largest eigenvalue.%
}. 

Since then, a series of papers have dealt with this conjecture. One
approach, due to Kahn and Szemerédi, studies the Rayleigh quotient
of the adjacency matrix $A_{\Gamma}$ and shows that it is likely
to be small on all points of an appropriate $\varepsilon$-net on
the unit sphere. This approach yielded an asymptotic bound of $\lambda\left(\Gamma\right)<c\sqrt{d}$
for some unspecified constant $c$ \cite{FKS89}. In the recent work
\cite[Thm.\ 26]{DJPP13}, it is shown that this bound can be taken
to be $10^{4}$. Other works, as well as the current paper, are based
on the idea of the \emph{trace method}, which amounts to bounding
$\lambda\left(\Gamma\right)$ by means of counting closed walks in
$\Gamma$. These works include \cite{BS87}, in which Broder and Shamir
show that a.a.s.~$\lambda\left(\Gamma\right)\leq$$\sqrt{2}d^{3/4}+\varepsilon$
($\forall\varepsilon>0$); \cite{Fri91} where Friedman obtains $\lambda\left(\Gamma\right)\leq2\sqrt{d-1}+2\log d+c$
a.a.s.; and, most famously, Friedman's 100-page-long proof of Alon's
conjecture \cite{Fri08}. Friedman shows that for every $\varepsilon>0$,
$\lambda\left(\Gamma\right)\leq2\sqrt{d-1}+\varepsilon$ a.a.s.

In the current paper we prove a result which is slightly weaker than
Friedman's. However, the proof we present is substantially shorter
and simpler than the sophisticated proof in \cite{Fri08}. Our proof
technique relies on recent deep results in combinatorial group theory
\cite{PP15}. We show the following:
\begin{thm}
\label{thm:2sqrt(d-1)+1}Fix $d\geq3$ and let $\Gamma$ be a random
$d$-regular simple graph on $n$ vertices chosen at uniform distribution.
Then 
\[
\lambda\left(\Gamma\right)<2\sqrt{d-1}+1
\]
asymptotically almost surely%
\footnote{For small $d$'s better bounds are attainable - see the table in Section
\subref{From-even-to-odd}.%
}. 
\end{thm}
For $d$ even, or $d$ odd large enough, we obtain a better bound
of $2\sqrt{d-1}+0.84$. The same result, for $d$ even, holds also
for random $d$-regular graphs in the \emph{permutation model} (see
below). In fact, we first prove the result stated in Theorem \ref{thm:2sqrt(d-1)+1}
for random graphs in this model (with $d$ even). The derivation of
Theorem \ref{thm:2sqrt(d-1)+1} for the uniform model and $d$ even
is then immediate by results of Wormald \cite{Wor99} and Greenhill
et al.~\cite{GJKW02} showing the \emph{contiguity}%
\footnote{Two models of random graphs are contiguous if the following holds:
$\left(i\right)$ for every (relevant) $n$ they define distributions
on the same set of graphs on $n$ vertices, and $\left(ii\right)$
whenever a sequence of events has a probability of $1-o_{n}\left(1\right)$
in one distribution, it has a probability of $1-o_{n}\left(1\right)$
in the other distribution as well.\label{fn:contiguity}%
}\emph{ }of different models of random regular graphs (see Appendix
\ref{sec:contiguity}). Finally, we derive the case of odd $d$ relying
on the even case and a contiguity argument in which we loose some
in the constant and get $1$ instead of $0.84$ (Section \subref{From-even-to-odd}).

The permutation model, which we denote by ${\cal P}_{n,d}$\marginpar{${\cal P}_{n,d}$},
applies only to even values of $d$. In this model, a random $d$-regular
graph $\Gamma$ on the set of vertices $\left[n\right]$ is obtained
by choosing independently and uniformly at random $\frac{d}{2}$ permutations
$\sigma_{1},\ldots,\sigma_{\frac{d}{2}}$ in the symmetric group $S_{n}$,
and introducing an edge $\left(v,\sigma_{j}\left(v\right)\right)$
for every $v\in\left[n\right]$ and $j\in\left\{ 1,\ldots,\frac{d}{2}\right\} $.
Of course, $\Gamma$ may be disconnected and can have loops or multiple
edges.

We stress that even after Alon's conjecture is established, many open
questions remain concerning $\lambda\left(\Gamma\right)$. In fact,
very little is known about the distribution of $\lambda\left(\Gamma\right)$.
A major open question is the following: what is the probability that
a random $d$-regular graph is \emph{Ramanujan}, i.e.~that $\lambda\left(\Gamma\right)\leq2\sqrt{d-1}$?
There are contradicting experimental pieces of evidence (in \cite{MNS08}
it is conjectured that this probability tends to $27\%$ as $n$ grows;
simulations depicted in \cite[Section 7]{HLW06} suggest it may be
larger than $50\%$) . However, even the following, much weaker question
is not known: are there infinitely many Ramanujan $d$-regular graphs
for every $d\geq3$? The only positive results here are by explicit
constructions of Ramanujan graphs when $d-1$ is a prime power by
\cite{LPS88,Mar88,Mor94}. In a recent major breakthrough, Marcus,
Spielman and Srivastava \cite{marcus2013interlacing} show the existence
of infinitely many $d$-regular \emph{bipartite-Ramanujan} graphs
for every $d\geq3$ (namely, these graphs have two {}``trivial''
eigenvalues, $d$ and $-d$, while all others lie inside $\left[-2\sqrt{d-1},2\sqrt{d-1}\right]$).
Still, the original problem remains open. We hope our new approach
may eventually contribute to answering these open questions.

\subsection*{Random coverings of a fixed base graph\label{sub:intro-Random-Coverings-of}}

The hidden reason for the number $2\sqrt{d-1}$ in Alon's conjecture
and Alon-Boppana Theorem is the following: All finite $d$-regular
graphs are covered by the $d$-regular (infinite) tree $T=T_{d}$.
Let $A_{T}:\ell^{2}\left(V\left(T\right)\right)\to\ell^{2}\left(V\left(T\right)\right)$
be the adjacency operator of the tree, defined by
\[
\left(A_{T}f\right)\left(u\right)=\sum_{v\sim u}f\left(v\right).
\]
Then $A_{T}$ is a self-adjoint operator and, as firstly proven by
Kesten \cite{Kes59}, the spectrum of $A_{T}$ is $\left[-2\sqrt{d-1},2\sqrt{d-1}\right]$.
Namely, $2\sqrt{d-1}$ is the \emph{spectral radius}%
\footnote{The spectral radius of an operator $M$ is defined as $\sup\left\{ \left|\lambda\right|\,\middle|\,\lambda\in\mathrm{Spec}\, M\right\} $.%
} of $A_{T}$. In this respect, among all possible (finite) quotients
of the tree, Ramanujan graphs are {}``ideal'', having their non-trivial
spectrum as good as the {}``ideal object'' $T$.

It is therefore natural to measure the spectrum of any graph $\Gamma$
against the spectral radius of its covering tree. Several authors
call graphs whose non-trivial spectrum is bounded by this value \emph{Ramanujan},
generalizing the regular case. Many of the results and questions regarding
the spectrum of $d$-regular graphs extend to this general case. For
example, an analogue of Alon-Boppana's Theorem is given in Proposition
\ref{prop:lower-bound}.

Ideally, one would like to extend Alon\textquoteright{}s conjecture
on nearly-Ramanujan graphs to every infinite tree $T$ with finite
quotients, and show that most of its quotients are nearly Ramanujan.
However, as shown in \cite{LN98}, there exist trees $T$ with some
minimal finite quotient $\Omega$ which is not Ramanujan. All other
finite quotients of $T$ are then coverings of $\Omega$, and inherit
the {}``bad'' eigenvalues of this quotient (we elaborate a bit more
in Appendix \secref{contiguity}) . Such examples invalidate the obvious
analogue of Alon's conjecture. 

But what if we ignore this few, fixed, {}``bad'' eigenvalues originated
in the minimal quotient $\Omega$ and focus only on the remaining,
{}``new'' eigenvalues of each larger quotient? In this sense, a
generalized version of Alon's conjecture is indeed plausible. Instead
of studying the spectrum of a random finite quotient of $T$, one
may consider the spectrum of a random finite covering of a fixed finite
graph. This is the content of the generalized conjecture of Friedman
appearing here as Conjecture \ref{conj:friedman}. 

\begin{wrapfigure}{R}{0.4\columnwidth}%
\noindent \begin{centering}
\includegraphics[bb=0bp 100bp 400bp 450bp,scale=0.4]{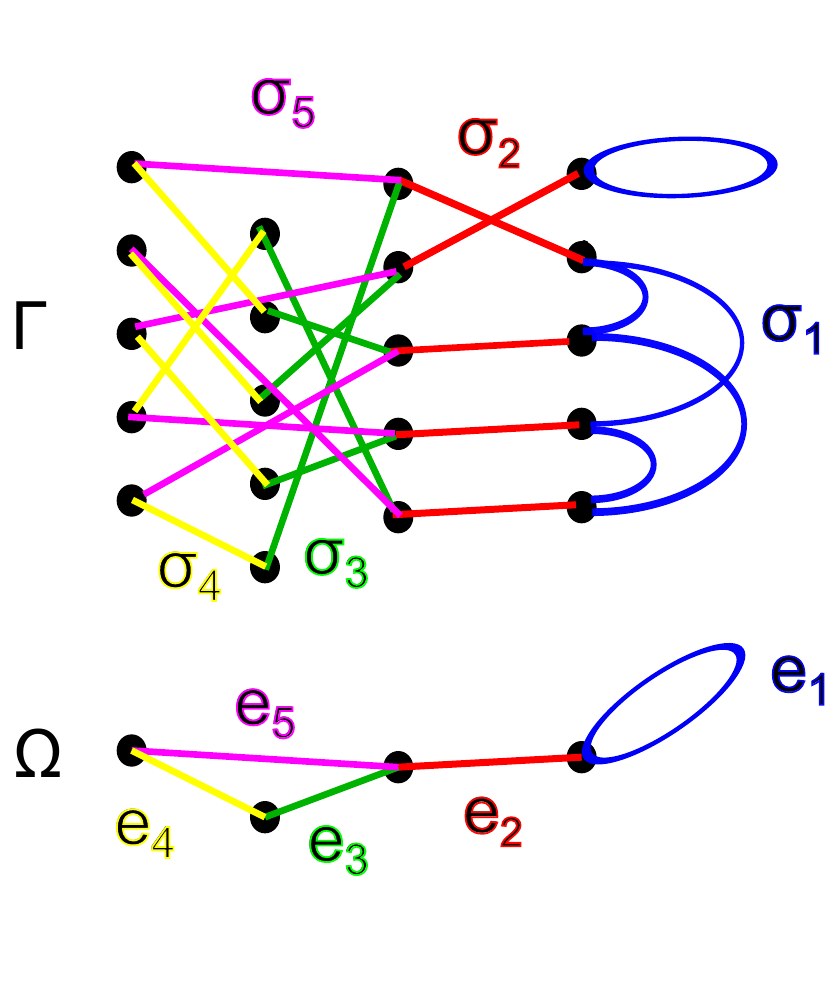}\begin{center}

\par\end{center}
\par\end{centering}

\caption{\label{fig:covering} A $5$-covering of a base graph using permutations.}
\end{wrapfigure}%
In order to describe this conjecture precisely, let us first describe
the random model we consider. This is a generalization of the permutation
model for random regular graphs, which generates families of graphs
with a common universal covering tree. A random graph $\Gamma$ in
the permutation model ${\cal P}_{n,d}$ can be equivalently thought
of as a random $n$-sheeted covering space of the bouquet with $\frac{d}{2}$
loops. In a similar fashion, fix a finite, connected base graph $\Omega$,
and let $\Gamma$ be a random $n$-covering space of $\Omega$. More
specifically, $\Gamma$ is sampled as follows: its set of vertices
is $V\left(\Omega\right)\times\left[n\right]$. A permutation $\sigma_{e}\in S_{n}$
is then chosen uniformly and independently at random for every edge
$e=\left(u,v\right)$ of $\Omega$, and for every $i\in\left[n\right]$
the edge $\left(\left(u,i\right),\left(v,\sigma_{e}\left(i\right)\right)\right)$
is introduced in $\Gamma$%
\footnote{We stress that we consider undirected edges. Although one should first
choose an arbitrary orientation for each edge in order to construct
the random covering, the orientation does not impact the resulting
probability space.%
}. We denote this model by ${\cal C}_{n,\Omega}$\marginpar{${\cal C}_{n,\Omega}$}
(so that ${\cal C}_{n,B_{\frac{d}{2}}}={\cal P}_{n,d}$, where $B_{\frac{d}{2}}$
is the bouquet with $\frac{d}{2}$ loops). For example, all bipartite
$d$-regular graphs on $2n$ vertices cover the graph $\vcenter{\xymatrix@1@R=2pt@C=25pt{
\bullet \ar@{-}@/^0.4pc/[r]_{\filleddiamond} \ar@{-}@/^.2pc/[r] \ar@{-}@/_0.41pc/[r] \ar@{-}[r]  & \bullet
}}$ with two vertices and $d$ edges connecting them. Various properties
of random graphs in the ${\cal C}_{n,\Omega}$ model were thoroughly
examined over the last decade (e.g. \cite{AL02,ALM02,Fri03,LR05,AL06,BL06,LP10}).
From now on, by a {}``random $n$-covering of $\Omega$'' we shall
mean a random graph in the model ${\cal C}_{n,\Omega}$.

A word about the spectrum of a non-regular graph is due. In the case
of $d$-regular graphs we have considered the spectrum of the adjacency
operator. In the general case, it is not apriori clear which operator
best describes in spectral terms the properties of the graph. In this
paper we consider two operators: the \emph{adjacency operator} $A_{\Gamma}$\marginpar{$A_{\Gamma}$}
defined as above, and the \emph{Markov operator} $M_{\Gamma}$\marginpar{$M_{\Gamma}$}
defined by
\[
\left(M_{\Gamma}f\right)\left(u\right)=\frac{1}{\deg\left(u\right)}\sum_{v\sim u}f\left(v\right).
\]
(A third possible operator is the \emph{Laplacian} - see Appendix
\ref{sec:Operators-on-Non-Regular}.) With a suitable inner product,
each of these operators is self-adjoint and therefore admits a real
spectrum (and see Appendix \ref{sec:Operators-on-Non-Regular} for
the relations of these spectra to expansion properties of $\Gamma$).

For a finite graph $\Omega$ on $m$ vertices, the spectrum of the
adjacency matrix $A_{\Omega}$ is 
\[
\pf\left(\Omega\right)=\lambda_{1}\geq\ldots\geq\lambda_{m}\geq-\pf\left(\Omega\right),
\]
$\pf\left(\Omega\right)$ \marginpar{$\pf\left(\Omega\right)$}being
the Perron-Frobenius eigenvalue of $A_{\Gamma}$. The spectrum of
$M_{\Omega}$ is 
\[
1=\mu_{1}\geq\ldots\geq\mu_{m}\geq-1,
\]
the eigenvalue 1 corresponding to the constant function. Every finite
covering $\Gamma$ of $\Omega$ shares the same Perron-Frobenius eigenvalue,
and moreover, inherits the entire spectrum of $\Omega$ (with multiplicity):
Let $\pi:\Gamma\to\Omega$ be the covering map, sending the vertex
$\left(v,i\right)$ to $v$ and the edge $\left(\left(u,i\right),\left(v,j\right)\right)$
to $\left(u,v\right)$. Every eigenfunction $f:V\left(\Omega\right)\to\mathbb{C}$
of any operator on $l^{2}\left(V\left(\Omega\right)\right)$ as above,
can be pulled back to an eigenfunction of $\Gamma$, $f\circ\pi$,
with the same eigenvalue. Thus, every eigenvalue of $\Omega$ (with
multiplicity) is trivially an eigenvalue of $\Gamma$ as well. We
denote by $\lambda_{A}\left(\Gamma\right)$\marginpar{$\lambda_{A}\left(\Gamma\right)$}
the largest absolute value of a \emph{new} eigenvalue of $A_{\Gamma}$,
namely the largest one not inherited from $\Omega$. Equivalently,
this is the largest absolute eigenvalue of an eigenfunction of $\Gamma$
which sums to zero on every fiber of $\pi$. In a similar fashion
we define $\lambda_{M}\left(\Gamma\right)$\marginpar{$\lambda_{M}\left(\Gamma\right)$},
the largest absolute value of a new eigenvalue of $M_{\Gamma}$. Note
that in the regular case (i.e.~when $\Omega$ is $d$-regular), $A_{\Gamma}=d\cdot M_{\Gamma}$,
and so $\lambda_{A}\left(\Gamma\right)=d\cdot\lambda_{M}\left(\Gamma\right)$.
Moreover, when $\Omega=B_{\frac{d}{2}}$ is the bouquet, $\lambda_{A}\left(\Gamma\right)=\lambda\left(\Gamma\right)$.

As in the regular case, the largest non-trivial eigenvalue is closely
related to the spectral radius of $T$, the universal covering tree
of $\Omega$ (which is also the universal covering of every connected
covering $\Gamma$ of $\Omega$). We denote by $\rho_{A}\left(\Omega\right)$
and $\rho_{M}\left(\Omega\right)$\marginpar{$\rho_{A}\left(\Omega\right),\rho_{M}\left(\Omega\right)$}
the spectral radii of $A_{T}$ and $M_{T}$, resp. (So when $\Omega$
is $d$-regular, $\rho_{A}\left(\Omega\right)=d\cdot\rho_{M}\left(\Omega\right)=2\sqrt{d-1}$.)
First, there are parallels of Alon-Boppana's bound in this more general
scenario. The first part of the following proposition is due to Greenberg,
while the second one is due to Burger:
\begin{prop}
\label{prop:lower-bound}Let $\Gamma$ be an $n$-covering of $\Omega$.
Then 
\begin{enumerate}
\item $\lambda_{A}\left(\Gamma\right)\geq\rho_{A}\left(\Omega\right)-o_{n}\left(1\right)$
\cite[Thm 2.11]{Gre95}.
\item $\lambda_{M}\left(\Gamma\right)\geq\rho_{M}\left(\Omega\right)-o_{n}\left(1\right)$
\cite[Prop. 6]{Bur87,GZ99}.
\end{enumerate}
\end{prop}
When $\Omega$ is $d$-regular (but not necessarily a bouquet), this
proposition was also observed by Serre \cite{Ser90}.

As in the $d$-regular case, the only deterministic upper bounds are
trivial: $\lambda_{A}\left(\Gamma\right)\leq\pf\left(\Omega\right)$
and $\lambda_{M}\left(\Gamma\right)\leq1$. But there are interesting
probabilistic phenomena. The following conjecture is the natural extension
of Alon's conjecture. The adjacency-operator version is due to Friedman
\cite{Fri03}. We extend it to the Markov operator $M$ as well:
\begin{conjecture}
[Friedman, \cite{Fri03}]\label{conj:friedman}Let $\Omega$ be a
finite connected graph. If $\Gamma$ is a random $n$-covering of
$\Omega$, then for every $\varepsilon>0$, 
\[
\lambda_{A}\left(\Gamma\right)<\rho_{A}\left(\Omega\right)+\varepsilon
\]
asymptotically almost surely, and likewise 
\[
\lambda_{M}\left(\Gamma\right)<\rho_{M}\left(\Omega\right)+\varepsilon
\]
asymptotically almost surely.
\end{conjecture}
Since $\lambda_{A}\left(\Gamma\right)$ and $\lambda_{M}\left(\Gamma\right)$
provide an indication for the quality of expansion of $\Gamma$ (see
Appendix \ref{sec:Operators-on-Non-Regular}), Conjecture \ref{conj:friedman}
asserts that if the base graph $\Omega$ is a good (nearly optimal)
expander then with high probability so is its random covering $\Gamma$.

In the same paper (\cite{Fri03}), Friedman generalizes the method
of Broder-Shamir mentioned above and shows that $\lambda_{A}\left(\Gamma\right)<\pf\left(\Omega\right)^{1/2}\rho_{A}\left(\Omega\right)^{1/2}+\varepsilon$
a.a.s. An easy variation on his proof gives $\lambda_{M}\left(\Gamma\right)<\rho_{M}\left(\Omega\right)^{1/2}+\varepsilon$
a.a.s. In \cite{LP10}, Linial and the author improve this to $\lambda_{A}\left(\Gamma\right)<3\pf\left(\Omega\right)^{1/3}\rho_{A}\left(\Omega\right)^{2/3}+\varepsilon$
(and with the same technique one can show $\lambda_{M}\left(\Gamma\right)<3\rho_{M}\left(\Omega\right)^{2/3}+\varepsilon$).
This is the best known result for the general case prior to the current
work.

Several works studied the special case where the base-graph $\Omega$
is $d$-regular (recall that in this case $\lambda_{A}\left(\Gamma\right)=d\cdot\lambda_{M}\left(\Gamma\right)$
and $\rho_{A}\left(\Omega\right)=2\sqrt{d-1}$). Lubetzky, Sudakov
and Vu \cite{LSV11} find a sophisticated improvement of the Kahn-Szemerédi
approach and prove that a.a.s.~$\lambda_{A}\left(\Gamma\right)\leq C\cdot\max\left(\lambda\left(\Omega\right),\rho_{A}\left(\Omega\right)\right)\cdot\log\rho_{A}\left(\Omega\right)$
for some unspecified constant $C$. An asymptotically better bound
of $430{,}656\sqrt{d}$ is given by Addario-Berry and Griffiths \cite{AbG10},
by further ameliorating the same basic technique (note that this bound
becomes meaningful only for $d\geq430{,}656^{2}$).

The following theorems differ from Conjecture \ref{conj:friedman}
only by a small additive or multiplicative factor, and are nearly
optimal by Proposition \ref{prop:lower-bound}. They pose a substantial
improvement upon all former results, both in the special case of a
$d$-regular base-graph $\Omega$ and, to a larger extent, in the
general case of any finite base-graph.
\begin{thm}
\label{thm:sqrt3-times-rho}Let $\Omega$ be an arbitrary finite connected
graph, and let $\Gamma$ be a random $n$-covering of $\Omega$. Then
for every $\varepsilon>0$, 
\[
\lambda_{A}\left(\Gamma\right)<\sqrt{3}\cdot\rho_{A}\left(\Omega\right)+\varepsilon
\]
asymptotically almost surely, and similarly

\[
\lambda_{M}\left(\Gamma\right)<\sqrt{3}\cdot\rho_{M}\left(\Omega\right)+\varepsilon
\]
asymptotically almost surely.
\end{thm}
For the special case where $\Omega$ is regular, we obtain the same
bound as in the case of the bouquet (Theorem \ref{thm:2sqrt(d-1)+1}
for $d$ even):
\begin{thm}
\label{thm:base-d-regular}Let $\Omega$ be a finite connected $d$-regular
graph ($d\geq3$) and let $\Gamma$ be a random $n$-covering of $\Omega$.
Then 
\[
\lambda_{A}\left(\Gamma\right)<\rho_{A}\left(\Omega\right)+0.84=2\sqrt{d-1}+0.84
\]
asymptotically almost surely.
\end{thm}
We stress the following special case concerning random \emph{bipartite}
$d$-regular graphs. It follows as all bipartite regular graphs cover
the graph $\Omega$ consisting of two vertices and $d$ edges connecting
them.
\begin{cor}
\label{cor:bipartite}Let $\Gamma$ be a random bipartite $d$-regular
graph on $n$ vertices ($d\geq3$). Then 
\[
\lambda_{A}\left(\Gamma\right)<2\sqrt{d-1}+0.84
\]
asymptotically almost surely (as $n\to\infty$)%
\footnote{Again, for small values of $d$ a better bound is reachable - see
Sections \subref{From-even-to-odd} and \subref{Proof-of-Theorem-base-regular}.%
}.
\end{cor}
This means that alongside the two trivial eigenvalues $\pm d$, all
other eigenvalues of the bipartite graph $\Gamma$ are a.a.s.~within
$\left[-2\sqrt{d-1}-0.84,2\sqrt{d-1}+0.84\right]$. The result applies
also to random simple bipartite regular graphs: see appendix \ref{sec:contiguity}.

To put Theorems \ref{thm:2sqrt(d-1)+1}, \ref{thm:sqrt3-times-rho}
and \ref{thm:base-d-regular} in context, Table \tabref{results}
summarizes the results mentioned above for the different cases in
question, with respect to the adjacency operator $A_{\Gamma}$. (See
also the late remark on Page \pageref{sec:Late-Remark} referring
to the very recent work \cite{friedman2014relativized}.)

\begin{table}[t]
\begin{tabular}{|>{\centering}m{2.2cm}||>{\centering}m{3cm}|>{\centering}m{3.4cm}|>{\centering}m{3cm}|}
\hline 
The base-graph $\Omega$ & Any graph & $d$-regular & $B_{\frac{d}{2}}=$ a bouquet of $\frac{d}{2}$ loops\tabularnewline
\hline 
 &  & $\rho=2\sqrt{d-1}$ & $\rho=2\sqrt{d-1}$\tabularnewline
\hline 
Deterministic lower bound for $\lambda_{A}\left(\Gamma\right)$ & $\rho-o_{n}\left(1\right)$\\
 \cite{Gre95} & $\rho-o_{n}\left(1\right)$\\
 \cite{Ser90} & $\rho-o_{n}\left(1\right)$ (Alon-Boppana) \cite{Nil91}\tabularnewline
\hline 
Conjectured probabilistic upper bound  & \multicolumn{2}{c|}{$\rho+\varepsilon$ \cite{Fri03}} & $\rho+\varepsilon$\\
\cite{Alo86}\tabularnewline
\hline 
Probabilistic upper bounds,  & $\sqrt{\pf\left(\Omega\right)\rho}+\varepsilon$ \cite{Fri03} & $\negthickspace\negthickspace\negthickspace\negthickspace\negthickspace\negthickspace\negthickspace\negthickspace\negthickspace\negthickspace\negthickspace\negthickspace\negthickspace\negthickspace\negthickspace\negthickspace\Longrightarrow\qquad\sqrt{d\rho}+\varepsilon$ & $\sqrt{d\rho}+\varepsilon$ \cite{BS87}\tabularnewline
\cline{2-4} 
ordered by asymptotic & $3\cdot\pf\left(\Omega\right)^{1/3}\rho{}^{2/3}+\varepsilon\;$ \cite{LP10} & $\negthickspace\negthickspace\negthickspace\negthickspace\negthickspace\negthickspace\Longrightarrow\qquad3\cdot d^{1/3}\rho^{2/3}+\varepsilon$ & \tabularnewline
\cline{2-4} 
strength for growing $\rho$ &  & $C\cdot\max\left(\lambda\left(\Omega\right),\rho\right)\log\rho$
\cite{LSV11} & \tabularnewline
\cline{2-4} 
 &  & $265{,}000\cdot\rho$\\
\cite{AbG10} & $6{,}200\cdot\rho$\\
\cite{FKS89,DJPP13}\tabularnewline
\cline{2-4} 
 & \cellcolor{yellow}$\sqrt{3}\cdot\boldsymbol{\rho+\varepsilon}$ \\
\textbf{(Thm \ref{thm:sqrt3-times-rho})} &  & \tabularnewline
\cline{2-4} 
 &  &  & $\rho+2\log d+c$ \cite{Fri91}\tabularnewline
\cline{2-4} 
 &  & \cellcolor{yellow}$\boldsymbol{\rho+0.84}$

\textbf{(Thm \ref{thm:base-d-regular})} & \cellcolor{yellow}$\boldsymbol{\rho+0.84}$ ~\\
\textbf{(Thm \ref{thm:2sqrt(d-1)+1})}\tabularnewline
\cline{2-4} 
 &  &  & $\rho+\varepsilon$ \cite{Fri08}\tabularnewline
\hline 
\end{tabular}\caption{Our results compared with former ones. As above, $\Omega$ is the
connected base-graph and $\rho=\rho_{A}\left(\Omega\right)$ is the
spectral radius of its universal covering tree. The results are ordered
by their asymptotic strength.}
\label{tab:results}
\end{table}

Finally, let us stress that alongside the different models for random
$d$-regular graphs, random coverings of a fixed, good expander, are
probably the most natural other source for random, good expanders
({}``good'' expanders are \emph{sparse} graphs\emph{ }with high
quality of expansion). Other known models for random graphs do not
necessarily have this property. For example, the Erdös-Rényi model
$G\left(n,p\right)$, fails to produce good expander graphs: when
$p$ is small ($O\left(\frac{1}{n}\right)$) the generic graph is
not an expander (due, e.g., to lack of connectivity), whereas for
larger values of $p$, the average degree grows unboundedly.

\section{Overview of the Proof\label{sec:Overview-of-the-proof}}

In this section we present the outline of the proof of Theorems \ref{thm:2sqrt(d-1)+1},
\ref{thm:sqrt3-times-rho} and \ref{thm:base-d-regular}. For simplicity,
only the spectrum of the adjacency operator is considered in this
section. We assume the reader has some familiarity with free groups,
although we recall the basic definitions and classical relevant results
throughout the text. For a good exposition of free groups and combinatorial
group theory we refer the reader to \cite{Bog08}.

\subsection*{Step I: The trace method}

Let $\Omega$ be a fixed base graph with $k$ edges and $\Gamma$
a random $n$-covering in the model ${\cal C}_{n,\Omega}$. In the
spirit of the trace method, the spectrum of $\Gamma$ is analyzed
by counting closed walks. More concretely, denote by ${\cal CW}_{t}\left(\Gamma\right)$\marginpar{${\cal CW}_{t}\left(\Gamma\right)$}
the set of closed walks of edge-length $t$ in $\Gamma$. If $\mathrm{Spec\left(A_{\Gamma}\right)}$\marginpar{$\mathrm{Spec\left(A_{\Gamma}\right)}$}
denotes the multiset of eigenvalues of $A_{\Gamma}$, then for every
$t\in\mathbb{N}$, 
\[
\sum_{\mu\in\mathrm{Spec}\left(A_{\Gamma}\right)}\mu^{\, t}=\mathrm{tr}\left(A_{\Gamma}^{\, t}\right)=\left|\cpt\left(\Gamma\right)\right|.
\]
Orient each of the $k$ edges of $\Omega$ arbitrarily, label them
by $x_{1},\ldots,x_{k}$ and let $X=\left\{ x_{1},\ldots,x_{k}\right\} $.
Let $\sigma_{1},\ldots,\sigma_{k}\in S_{n}$ denote the random permutations
by which $\Gamma$ is defined: for each edge $x_{j}=\left(u,v\right)$
of $\Omega$ and each $i\in\left[n\right]$, $\Gamma$ has an edge
$\left(\left(u,i\right),\left(v,\sigma_{j}\left(i\right)\right)\right)$.
Note that every closed walk in $\Gamma$ projects to a closed walk
in $\Omega$. Thus, instead of counting directly closed walks in $\Gamma$,
one can count, for every closed walk in $\Omega$, the number of closed
walks in $\Gamma$ projecting onto it. 

Let $w=x_{j_{1}}^{\varepsilon_{1}}\ldots x_{j_{t}}^{\varepsilon_{t}}\in\cpt\left(\Omega\right)\subseteq\left(X\cup X^{-1}\right)^{t}$
be a closed walk in the base graph $\Omega$, beginning (and terminating)
at some vertex $v\in V\left(\Omega\right)$. (Here $\varepsilon_{i}=\pm1$
and $x_{j}^{-1}$ means the walk traverses the edge $x_{j}$ in the
opposite orientation.) For every $i\in\left[n\right]$ there is a
unique lift of $w$ to some walk in $\Gamma$, not necessarily closed,
which begins at the vertex $\left(v,i\right)$. This lifted walk terminates
at the vertex $\left(v,j\right)$, where $j$ is obtained as follows:
let $w\left(\sigma_{1},\ldots,\sigma_{k}\right)$ denote the permutation
obtained by composing $\sigma_{1},\ldots,\sigma_{k}$ according to
$w$, namely, $w\left(\sigma_{1},\ldots,\sigma_{k}\right)=\sigma_{j_{1}}^{\varepsilon_{1}}\ldots\sigma_{j_{t}}^{\varepsilon_{t}}\in S_{n}$.
Then $j$ is the image of $i$ under this permutation: $j=w\left(\sigma_{1},\ldots,\sigma_{k}\right)\left(i\right)=\sigma_{j_{1}}^{\varepsilon_{1}}\ldots\sigma_{j_{t}}^{\varepsilon_{t}}\left(i\right)$%
\footnote{For convenience, we use in this paper the non-standard convention
that permutations are composed from left to right. %
}. Thus, the $i$-th lift of $w$ is a closed walk if and only if $i$
is a fixed point of the permutation $w\left(\sigma_{1},\ldots,\sigma_{k}\right)$,
and the number of closed walks in $\Gamma$ projecting onto $w$ is
equal to the number of fixed points of $w\left(\sigma_{1},\ldots,\sigma_{k}\right)$.
We denote this number by ${\cal F}_{w,n}={\cal F}_{w,n}\left(\sigma_{1},\ldots,\sigma_{k}\right)$\marginpar{${\cal F}_{w,n}$}.
\begin{claim}
For every even $t\in\mathbb{N}$,
\begin{equation}
\mathbb{E}\left[\lambda_{A}\left(\Gamma\right)^{t}\right]\leq\sum_{w\in\cpt\left(\Omega\right)}\left[\mathbb{E}\left[{\cal F}_{w,n}\right]-1\right].\label{eq:bounding-lambda-with-Fwn}
\end{equation}

\end{claim}
\noindent (The expectation on the l.h.s.~is over ${\cal C}_{n,\Omega}$,
which amounts to the i.i.d.~uniform permutations $\sigma_{1},\ldots,\sigma_{k}\in S_{n}$.
The expectation on the r.h.s.~is over the same $k$-tuple of permutations.)
\begin{proof}
Since $t$ is even,
\begin{eqnarray*}
\lambda_{A}\left(\Gamma\right)^{t} & = & \left(\max_{\mu\in\mathrm{Spec}\left(A_{\Gamma}\right)\setminus\mathrm{Spec}\left(A_{\Omega}\right)}\left|\mu\right|\right)^{t}\leq\sum_{\mu\in\mathrm{Spec}\left(A_{\Gamma}\right)\setminus\mathrm{Spec}\left(A_{\Omega}\right)}\mu^{t}=\sum_{\mu\in\mathrm{Spec}\left(A_{\Gamma}\right)}\mu^{t}-\sum_{\mu\in\mathrm{Spec}\left(A_{\Omega}\right)}\mu^{t}=\\
 & = & \left|{\cal CW}_{t}\left(\Gamma\right)\right|-\left|{\cal CW}_{t}\left(\Omega\right)\right|=\sum_{w\in\cpt\left(\Omega\right)}\left[{\cal F}_{w,n}\left(\sigma_{1},\ldots,\sigma_{k}\right)-1\right].
\end{eqnarray*}
(Recall that we regard the spectrum of an operator as a multiset.)
The claim is established by taking expectations.
\end{proof}
We shall assume henceforth that $t$ is an even integer. Note that
in the special case where $\Omega=B_{\frac{d}{2}}$ is a bouquet of
$\frac{d}{2}$ loops, $\mathrm{Spec}\left(A_{\Omega}\right)=\left\{ d\right\} $,
and $\cpt\left(\Omega\right)=\left(X\cup X^{-1}\right)^{t}$, i.e.~it
consists of all words of length $t$ in the letters $X\cup X^{-1}$
(not necessarily reduced), so that $\left|\cpt\left(B_{\frac{d}{2}}\right)\right|=d^{t}$.

\subsection*{Step II: The expected number of fixed points in $w\left(\sigma_{1},\ldots,\sigma_{k}\right)$}

The next stage in the proof of the main results is an analysis of
$\mathbb{E}\left[{\cal F}_{w,n}\right]$. This is where the results
from \cite{PP15} come to bear. Let $\F_{k}=\F\left(X\right)$ be
the free group on $k$ generators. Every word $w\in\cpt\left(\Omega\right)\subseteq\left(X\cup X^{-1}\right)^{t}$
corresponds to an element of $\F_{k}$ (by abuse of notation we let
$w$ denote an element of $\left(X\cup X^{-1}\right)^{t}$ and of
$\F_{k}=F\left(X\right)$ at the same time; it is important to stress
that reduction%
\footnote{By reduction of a word we mean the (repeated) deletion of subwords
of the form $x_{i}x_{i}^{-1}$ or $x_{i}^{-1}x_{i}$ for some $x_{i}\in X$.%
} of $w$ does not affect the associated permutation $w\left(\sigma_{1},\ldots,\sigma_{k}\right)$.)
The main theorem in \cite{PP15} estimates the expected number of
fixed points of the random permutation $w\left(\sigma_{1},\ldots,\sigma_{k}\right)\in S_{n}$,
where $\sigma_{1},\ldots,\sigma_{k}\in S_{n}$ are random permutations
chosen independently with uniform distribution. This theorem shows
that this expectation is related to an algebraic invariant of $w$
called its\emph{ primitivity rank}, which we now describe.

A word $w\in\F_{k}$ is \emph{primitive }if it belongs to a \emph{basis}%
\footnote{A basis of a free group is a free generating set. Namely, this is
a generating set such that every element of the group can be expressed
in a unique way as a reduced word in the elements of the set and their
inverses. For $\F_{k}$ this is equivalent to a generating set of
size $k$ \cite[Chap. 2.29]{Bog08}.\label{fn:basis}%
}\marginpar{\emph{primitive, basis}} of $\F_{k}$. For a given $w$,
one can also ask whether $w$ is primitive as an element of different
subgroups of $\F_{k}$ (which are free as well by a classical theorem
of Nielsen and Schreier \cite[Chap. 2.8]{Bog08}). If $w$ is primitive
in $\F_{k}$, it is also primitive in every subgroup $J\leq\F_{k}$
(e.g. \cite[Claim 2.5]{Pud14a}). However, if $w$ is \emph{not }primitive
in $\F_{k}$, it is sometimes primitive and sometimes not so in subgroups
containing it. Theoretically, one can go over all subgroups of $\F_{k}$
containing $w$, ordered by their rank%
\footnote{The \emph{rank }of a free group $\F$, denoted $\mathrm{rk}\left(\F\right)$,
is the size of (every) basis of $\F$.%
}, and look for the first time at which $ $$w$ is not primitive.
First introduced in \cite{Pud14a}, the primitivity rank of $w\in\F_{k}$
captures this notion:
\begin{defn}
\label{def:prim_rank} The \emph{primitivity rank} of $w\in\F_{k}$,
denoted $\pi\left(w\right)$, is \marginpar{$\pi\left(w\right)$}
\[
\pi(w)=\min\left\{ \rk\left(J\right)\,\middle|\,\begin{gathered}w\in J\le\F_{k}~s.t.\\
w\textrm{ is \textbf{not} primitive in \ensuremath{J}}
\end{gathered}
\right\} .
\]
If no such $J$ exists, i.e.\ if $w$ is primitive in $\F_{k}$,
then $\pi\left(w\right)=\infty$. \\
A subgroup $J$ for which the minimum is obtained is called \textbf{$w$}\emph{-critical},
and the set of $w$-critical subgroups is denoted\marginpar{$\crit\left(w\right)$}
$\crit\left(w\right)$. 
\end{defn}
For instance, $\pi\left(w\right)=1$ if and only if $w$ is a proper
power ($w=v^{d}$ for some $v\in\mathbf{F}_{k}$ and $d\ge2$). By
Corollary 4.2 and Lemma 6.8 in \cite{Pud14a}, the set of possible
primitivity ranks in $\F_{k}$ is $\left\{ 0,1,2,\ldots,k\right\} \cup\left\{ \infty\right\} $
(the only word $w$ with $\pi\left(w\right)=0$ being $w=1$). Moreover,
$\pi\left(w\right)=\infty$ iff $w$ is primitive. The same paper
also describes an algorithm to compute $\pi\left(w\right)$. 

The following theorem estimates $\mathbb{E}\left[{\cal F}_{w,n}\right]$,
the expected number of fixed points of $w\left(\sigma_{1},\ldots,\sigma_{k}\right)$,
where $\sigma_{1},\ldots,\sigma_{k}\in S_{n}$ are chosen independently
at random with uniform distribution:
\begin{thm}
\label{thm:avg_fixed_pts}\cite[Thm 1.8]{PP15} For every $w\in\F_{k}$,
the expected number of fixed points in $w\left(\sigma_{1},\ldots,\sigma_{k}\right)$
is
\[
\mathbb{E}\left[{\cal F}_{w,n}\right]=1+\frac{|\crit\left(w\right)|}{n^{\pi\left(w\right)-1}}+O\left(\frac{1}{n^{\pi\left(w\right)}}\right).
\]

\end{thm}
In particular, it is also shown that $\crit\left(w\right)$ is always
finite. The three leftmost columns in Table \tabref{Primitivity-Rank-and-fixed-points}
summarize the connection implied by Theorem \thmref{avg_fixed_pts}
between the primitivity rank of $w$ and the average number of fixed
points in the random permutation $w\left(\sigma_{1},\ldots,\sigma_{k}\right)$.

With Theorem \thmref{avg_fixed_pts} at hand, we can use the primitivity
rank to split the summation in \eqref{bounding-lambda-with-Fwn}.
We shall use the notation $\cptm\left(\Omega\right)=\left\{ w\in\cpt\left(\Omega\right)\,\middle|\,\pi\left(w\right)=m\right\} $\marginpar{$\cptm\left(\Omega\right)$}
for the subsets we obtain by splitting $\cpt\left(\Omega\right)$
according to primitivity ranks:
\begin{eqnarray}
\mathbb{E}\left[\lambda_{A}\left(\Gamma\right)^{t}\right] & \leq & \sum_{w\in\cpt\left(\Omega\right)}\left(\mathbb{E}\left[{\cal F}_{w,n}\right]-1\right)=\nonumber \\
 & = & \sum_{m=0}^{k}\sum_{w\in\cptm\left(\Omega\right)}\left(\frac{\left|\crit\left(w\right)\right|}{n^{m-1}}+O\left(\frac{1}{n^{m}}\right)\right)\label{eq:bounding-lambda-with-thm}
\end{eqnarray}
(note that for primitive words, i.e.~words with $\pi\left(w\right)=\infty$,
the expected number of fixed points is exactly 1, so their contribution
to the summation vanishes.)

\subsection*{Step III: A uniform bound for $\mathbb{E}\left[{\cal F}_{w,n}\right]$}

The error term $O\left(\frac{1}{n^{m}}\right)$ in \eqref{bounding-lambda-with-thm}
depends on $w$. For a given $w\in\cpt\left(\Omega\right)$, this
error term becomes negligible as $n\to\infty$. However, in order
to bound the r.h.s.~of \eqref{bounding-lambda-with-thm}, one needs
a\emph{ }uniform bound for all closed walks of length $t$ with a
given primitivity rank in $\Omega$. Namely, for every $m$ one needs
to control the $O\left(\cdot\right)$ term for all $w\in\cptm\left(\Omega\right)$
simultaneously. The third stage is therefore the following proposition:
\begin{prop*}[Follows from Prop.\ \ref{prop:controling-the-O} and Claim \ref{claim:controlling-the-O-for-m=00003D0}]
 Let $t=t\left(n\right)$ and $w\in\left(X\cup X^{-1}\right)^{t}$.
If $t^{2k+2}=o\left(n\right)$ then 
\[
\mathbb{E}\left[{\cal F}_{w,n}\right]\leq1+\frac{\left|\crit\left(w\right)\right|}{n^{\pi\left(w\right)-1}}\left(1+o_{n}\left(1\right)\right),
\]
where the $o_{n}\left(1\right)$ \emph{does not depend on $w$.}
\end{prop*}
Hence, as long as we keep $t^{2k+2}=o\left(n\right)$, we obtain:

\emph{
\begin{equation}
\mathbb{E}\left[\lambda_{A}\left(\Gamma\right)^{t}\right]\leq\left(1+o_{n}\left(1\right)\right)\sum_{m=0}^{k}\frac{1}{n^{m-1}}\sum_{w\in\cptm\left(\Omega\right)}\left|\crit\left(w\right)\right|.\label{eq:upper-bound-with-thm-2}
\end{equation}
}

\subsection*{Step IV: Counting words and critical subgroups}

The fourth step of the proof consists of estimating the exponential
growth rate (as $t\to\infty$) of the summation $\sum_{w\in\cptm\left(\Omega\right)}\left|\crit\left(w\right)\right|$
for every $m\in\left\{ 0,1,\ldots,k\right\} $. For $m=0$, the only
reduced word with $\pi\left(w\right)=0$ is $w=1$, and its sole critical
subgroup is the trivial subgroup $\left\{ 1\right\} $, so $\sum_{w\in\cpt^{0}\left(\Omega\right)}\left|\crit\left(w\right)\right|=\left|\cpt^{0}\left(\Omega\right)\right|$.
Moreover, words reducing to $1$ are precisely the completely back-tracking
closed walks, i.e.~the walks lifting to closed walks in the covering
tree. It follows that the exponential growth rate of $\left|\cpt^{0}\left(\Omega\right)\right|$
is exactly $\rho=\rho_{A}\left(\Omega\right)$, the spectral radius
of the covering tree (see Claim \ref{claim:bound-for-m=00003D0-general-Omega}).
For larger $m$ we obtain the following upper bound:
\begin{thm*}[Theorem \ref{thm:bound-for-m-in-general-Omega}]
Let $\Omega$ be a finite, connected graph with $k\geq2$ edges,
and let $m\in\left\{ 1,\ldots,k\right\} $. Then 
\[
\limsup_{t\to\infty}\left[\sum_{w\in\cptm\left(\Omega\right)}\left|\crit\left(w\right)\right|\right]^{1/t}\leq\left(2m-1\right)\cdot\rho.
\]

\end{thm*}
This upper bound is not tight in general. However, in the special
case where $\Omega$ is $d$-regular, we give better bounds:
\begin{thm*}[Follows from Corollaries \ref{cor:non-reduced-words_upper_bound}
and \ref{cor:bounds-for-non-red-words-and-reg-base-graph} and from
Theorem \ref{thm:prim-rank-distr}]
 Let $\Omega$ be a finite, connected $d$-regular graph ($d\geq3$)
with $k$ edges, and let $m\in\left\{ 0,1,\ldots,k\right\} $. Then
\[
\limsup_{t\to\infty}\left[\sum_{w\in\cptm\left(\Omega\right)}\left|\crit\left(w\right)\right|\right]^{1/t}\leq\begin{cases}
2\sqrt{2k-1} & 2m-1\leq\sqrt{2k-1}\\
2m-1+\frac{2k-1}{2m-1} & 2m-1\geq\sqrt{2k-1}
\end{cases}.
\]
Moreover, for $\Omega=B_{\frac{d}{2}}$ the bouquet, there is equality:
\[
\limsup_{t\to\infty}\left[\sum_{w\in\cptm\left(B_{\frac{d}{2}}\right)}\left|\crit\left(w\right)\right|\right]^{1/t}=\begin{cases}
2\sqrt{2k-1} & 2m-1\leq\sqrt{2k-1}\\
2m-1+\frac{2k-1}{2m-1} & 2m-1\geq\sqrt{2k-1}
\end{cases}.
\]
\end{thm*}
\begin{rem}
In fact, in the case of the bouquet, the growth rates in the statement
of the last theorem remain the same if we assume every word has only
a single critical subgroup. That is, the r.h.s.~gives also the growth
rate of the \emph{number} of words in $\left(X\cup X^{-1}\right)^{t}$
with primitivity rank $m$ - see Theorem \ref{thm:prim-rank-distr}. 
\end{rem}
Table \ref{tab:Primitivity-Rank-and-fixed-points} summarizes the
content of Theorems \ref{thm:avg_fixed_pts}, \ref{thm:bound-for-m-in-general-Omega}
and \ref{thm:prim-rank-distr}%
\footnote{The number $2k-2+\frac{2}{2k-3}$ in the last row of the table is
the exponential growth rate of the set of primitives in $\F_{k}$,
namely of $\left|\cpt^{\infty}\left(B_{\frac{d}{2}}\right)\right|$.
(Primitive words have no critical subgroups.) This result is not necessary
for the current work, and is established in a separate paper \cite{PW14},
using completely different techniques. We use it here only to show
that our bounds for $\sum_{w\in\cptm\left(B_{\frac{d}{2}}\right)}\left|\crit\left(w\right)\right|$
are tight - see Section \ref{sec:distr-of-prim-rank}.%
}.

\begin{table}[th]
\begin{centering}
\begin{tabular}{|>{\centering}m{2cm}|>{\centering}m{2.5cm}|c|>{\centering}m{2.9cm}|>{\centering}m{2.3cm}|}
\hline 
$\pi\left(w\right)$ & Description of $w$ & $\mathbb{E}\left[{\cal F}_{w,n}\right]$ & Growth rate\\
 for the bouquet $B_{\frac{d}{2}}$ & Bound on growth rate\\
 for general $\Omega$\tabularnewline
\hline 
\hline 
$0$ & $w=1$ & $n$ & $2\sqrt{2k-1}$ & $\rho$\tabularnewline
\hline 
$1\vphantom{\Big[}$ & a power & $\sim1+|\crit\left(w\right)|$ & $2\sqrt{2k-1}$ & $\rho$\tabularnewline
\hline 
$2\vphantom{\Big[}$ & E.g. $\left[x_{1},x_{2}\right],x_{1}^{\,2}x_{2}^{\,2}$ & $\sim1+\frac{|\crit\left(w\right)|}{n}$ & $2\sqrt{2k-1}$ & $3\rho$\tabularnewline
\hline 
$3\vphantom{\Big[}$ & E.g. $x_{1}^{\,2}x_{2}^{\,2}x_{3}^{\,2}$ & $\sim1+\frac{|\crit\left(w\right)|}{n^{2}}$ & $2\sqrt{2k-1}$ & $5\rho$\tabularnewline
\hline 
$\vdots$ & $\vdots$ & $\vdots$ & $\vdots$ & $\vdots$\tabularnewline
\hline 
$\left\lfloor \frac{\sqrt{2k-1}+1}{2}\right\rfloor \vphantom{\bigg|}$ &  &  & $2\sqrt{2k-1}$ & \tabularnewline
\hline 
$\left\lceil \frac{\sqrt{2k-1}+1}{2}\right\rceil \vphantom{\bigg|}$ &  &  & $2\pi\left(w\right)-1+\frac{2k-1}{2\pi\left(w\right)-1}$$ $ & \tabularnewline
\hline 
$\vdots$ &  &  & $\vdots$ & \tabularnewline
\hline 
$k-1$ & $\vdots$ & $\vdots$ & $2k-2+\frac{2}{2k-3}$ & \tabularnewline
\hline 
$k\vphantom{\Big[}$ & E.g. $x_{1}^{\,2}\ldots x_{k}^{\,2}$ & $\sim1+\frac{|\crit\left(w\right)|}{n^{k-1}}$ & $2k$ & \tabularnewline
\hline 
\hline 
$\infty$ & primitive & $1$ & $2k-2+\frac{2}{2k-3}$ & \tabularnewline
\hline 
\end{tabular}
\par\end{centering}

\centering{}\caption{Primitivity rank, the average number of fixed points, the exponential
growth rate of $\sum_{w\in\cptm\left(B_{\frac{d}{2}}\right)}\left|\crit\left(w\right)\right|$,
and bounds on the exponential growth rate of $\sum_{w\in\cptm\left(\Omega\right)}\left|\crit\left(w\right)\right|$.\label{tab:Primitivity-Rank-and-fixed-points}}
\end{table}

Whereas in the special case of the bouquet we count words in $\F_{k}$
of a given length and a given primitivity rank, the case of a general
graph concerns the equivalent question for words which in addition
belong to some fixed subgroups of $\F_{k}$. (There is one such subgroup
for each vertex $v$ of $\Omega$: the one consisting of the words
which correspond to closed walks at $v$.) The fact that the bounds
in Corollary \ref{cor:non-reduced-words_upper_bound} are better than
those in Theorem \ref{thm:bound-for-m-in-general-Omega} explains
the gap between Theorems \ref{thm:2sqrt(d-1)+1} and \ref{thm:base-d-regular}
which are tight up to a small \emph{additive} constant, and Theorem
\ref{thm:sqrt3-times-rho} which is tight up to a small \emph{multiplicative
}factor.

\subsection*{Step V: Some analysis}

The final step is fairly simple and technical: it consists of analyzing
the upper bounds we obtain from \eqref{upper-bound-with-thm-2} together
with Theorem \ref{thm:bound-for-m-in-general-Omega} and Corollary
\ref{cor:non-reduced-words_upper_bound}. We seek the value of $t$
(as a function of $n$) which yields the best bounds.\\

The paper is arranged as follows. Section \ref{sec:Core-Graphs-and-alg-extensions}
provides some basic facts about the concepts of core graphs and algebraic
extensions which are used throughout this paper. In Section \ref{sec:counting-words-and-algebraic-extensions}
we bound the number of words and critical subgroups and establish
the fourth step of the proof (first for the special case of the bouquet,
in Section \ref{sub:counting-bouquet}, then for the intermediate
case of an arbitrary regular base graph in Section \ref{sub:d-reg-base},
and finally for the most general case in Section \ref{sub:counting-omega}).
The third step of the proof, where the error term from Theorem \ref{thm:avg_fixed_pts}
is dealt with, is carried out in Section \ref{sec:Controlling-the-O},
where we have to recall some more details from \cite{PP15}. Section
\ref{sec:Completing-the-Proof-d-reg} completes the proof of Theorems
\ref{thm:2sqrt(d-1)+1} and \ref{thm:base-d-regular} and addresses
the source of the gap between Theorem \ref{thm:2sqrt(d-1)+1} and
Friedman's result. In Section \ref{sec:Completing-the-Proof-general-Omega}
we complete the proof of Theorem \ref{thm:sqrt3-times-rho}. We end
with results on the accurate exponential growth rate of words with
a given primitivity rank in $\F_{k}$ (Section \ref{sec:distr-of-prim-rank}),
and then list a few open questions. The appendices provide some background
on the relation between different models of random $d$-regular graphs
and between different models of random coverings (Appendix \ref{sec:contiguity}),
and on the theory of spectral expansion of non-regular graphs (Appendix
\ref{sec:Operators-on-Non-Regular}).

\section{Preliminaries: Core Graphs and Algebraic Extensions\label{sec:Core-Graphs-and-alg-extensions}}

This section describes some notions and ideas which are used throughout
the current paper.

\subsection{Algebraic extensions\label{sub:Algebraic-Extensions}}

Let $H\leq J$ be subgroups of $\F_{k}$. We say that $J$ is an \emph{algebraic
extension} \marginpar{\emph{algebraic extension}}of $H$ and denote
$H\leq_{\mathrm{alg}}J$\marginpar{$H\leq_{\mathrm{alg}}J$}, if there
is \emph{no} intermediate subgroup $H\leq L\lneq J$ which is a proper
free factor%
\footnote{If $H\leq J$ are free groups then $H$ is said to be a \emph{free
factor }of $J$ if a (every) basis of $H$ can be extended to a basis
of $J$.%
} of $J$. The name originated in \cite{KM02}, but the notion goes
back at least to \cite{takahasi1951note}, and was formulated independently
by several authors. It is central in the understanding of the lattice
of subgroups of $\mathbf{F}$. For example, it can be shown that every
extension $H\leq J$ of free groups admits a unique intermediate subgroup
$H\leq_{alg}M\ff J$\marginpar{$\ff$} (where $\ff$ denotes a free
factor). Moreover, if $H\leq\mathbf{F}$ is a finitely generated subgroup,
it has only finitely many algebraic extensions in $\mathbf{F}$. Thus,
every group containing $H$ is a free extension of one of finitely
many extensions of $H$, which is a well known theorem of Takahasi
\cite{takahasi1951note}. For more information we refer the interested
reader to \cite{KM02,PP15} and especially to \cite{MVW07}.

The importance of algebraic extensions in the current paper stems
from the following easy observation:
\begin{claim}
\label{claim:critical-is-algebraic}\cite[Cor. 4.4]{Pud14a} Every
$w$-critical subgroup is an algebraic extension of $\left\langle w\right\rangle $
(the subgroup generated by $w$).\\
More precisely, $\crit\left(w\right)$ consists precisely of the algebraic
extensions of $\left\langle w\right\rangle $ of minimal rank besides
$\left\langle w\right\rangle $ itself%
\footnote{Unless $w=1$ in which case $\crit\left(w\right)=\left\{ \left\langle \right\rangle \right\} =\left\{ \left\langle w\right\rangle \right\} $.%
}.
\end{claim}
To see the claim, assume that $H$ is a $w$-critical subgroup of
$\F_{k}$. Obviously, $\left\langle w\right\rangle \lneq H$. If $H$
is not an algebraic extension of $\left\langle w\right\rangle $,
then there is a proper intermediate free factor $\left\langle w\right\rangle \leq L\lneq_{ff}H$.
Since $w$ is not primitive in $H$, it is also not primitive in $L$
(as a basis containing $w$ for $L$ would extend to a basis for $H$),
but $\rk\left(L\right)<\rk\left(M\right)$, which is a contradiction.
Below, we use properties of $w$-critical subgroups which are actually
shared by all proper algebraic extensions of $\left\langle w\right\rangle $.

\subsection{Core graphs\label{sub:Core-Graphs}}

Fix a basis $X=\left\{ x_{1},\ldots,x_{k}\right\} $ of $\F_{k}$.
Associated with every subgroup $H\le\F_{k}$ is a directed, pointed
graph whose edges are labeled by $X$. This graph is called \emph{the
core-graph associated with $H$} and is denoted by $\G_{X}\left(H\right)$\marginpar{$\Gamma_{X}\left(H\right)$}.
We illustrate the notion in Figure \ref{fig:coset_and_core_graphs}.

To understand how $\Gamma_{X}\left(H\right)$ is constructed, recall
first the notion of the Schreier (right) coset graph of $H$ with
respect to the basis $X$, denoted by $\overline{\G}_{X}\left(H\right)$\marginpar{$\overline{\G}_{X}\left(H\right)$}.
This is a directed, pointed and edge-labeled graph. Its vertex set
is the set of all right cosets of $H$ in $\F_{k}$, where the basepoint
corresponds to the trivial coset $H$. For every coset $Hw$ and every
basis-element $x_{j}$ there is a directed $j$-edge (short for $x_{j}$-edge)
going from the vertex $Hw$ to the vertex $Hwx_{j}$.%
\footnote{Alternatively, $\overline{\G}_{X}\left(H\right)$ is the quotient
$H\backslash T$, where $T$ is the Cayley graph of $\F_{k}$ with
respect to the basis $X$, and $F_{k}$ (and thus also H) acts on
this graph from the left. Moreover, this is the covering-space of
$\overline{\G}_{X}\left(F_{k}\right)=\Gamma_{X}\left(F_{k}\right)$,
the bouquet of k loops, corresponding to $H$, via the correspondence
between pointed covering spaces of a space $Y$ and subgroups of its
fundamental group $\pi_{1}\left(Y\right)$.%
}

The core graph $\G_{X}\left(H\right)$ is obtained from $\overline{\G}_{X}\left(H\right)$
by omitting all the vertices and edges of $\overline{\G}_{X}\left(H\right)$
which are not traced by any reduced (i.e., non-backtracking) walk
that starts and ends at the basepoint. Stated informally, we trim
all {}``hanging trees'' from $\overline{\G}_{X}\left(H\right)$.
To illustrate, Figure \figref{coset_and_core_graphs} shows the graphs
$\overline{\G}_{X}\left(H\right)$ and $\G_{X}\left(H\right)$ for
$H=\langle x_{1}x_{2}x_{1}^{-3},x_{1}^{\;2}x_{2}x_{1}^{-2}\rangle\leq\F_{2}$.

\begin{figure}[h]
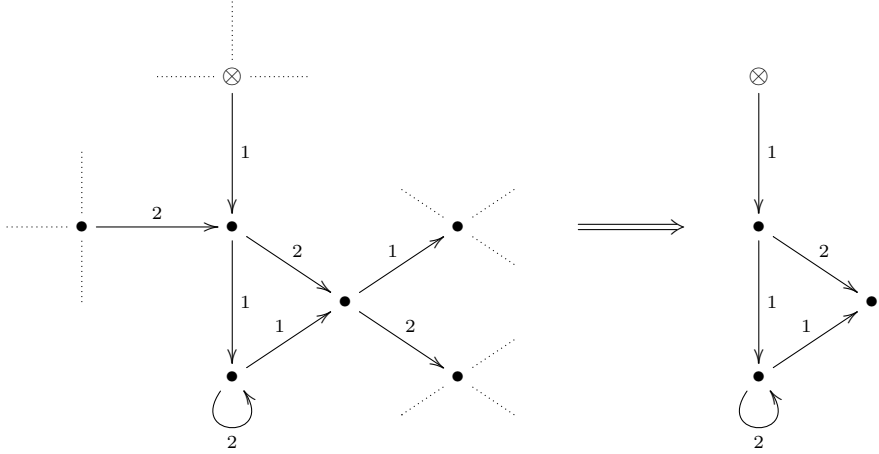

\begin{centering}
\begin{center}
\xy
(30,60)*+{\otimes}="s0";
(30,40)*+{\bullet}="s1";%
(45,30)*+{\bullet}="s2";%
(30,20)*+{\bullet}="s3";%
{\ar^{1} "s0";"s1"};%
{\ar^{2} "s1";"s2"};%
{\ar^{1} "s3";"s2"};%
{\ar^{1} "s1";"s3"};%
{\ar@(dl,dr)_{2} "s3";"s3"};%
(10,40)*+{\bullet}="t0";%
(60,40)*+{\bullet}="t1";%
(60,20)*+{\bullet}="t2";%
{\ar^{2} "t0";"s1"};%
{\ar^{1} "s2";"t1"};%
{\ar^{2} "s2";"t2"};%
{\ar@{..} "s0"; (20,60)*{}};%
{\ar@{..} "s0"; (30,70)*{}};%
{\ar@{..} "s0"; (40,60)*{}};%
{\ar@{..} "t0"; (10,50)*{}};%
{\ar@{..} "t0"; (0,40)*{}};%
{\ar@{..} "t0"; (10,30)*{}};%
{\ar@{..} "t1"; (52.5,45)*{}};%
{\ar@{..} "t1"; (67.5,45)*{}};%
{\ar@{..} "t1"; (67.5,35)*{}};%
{\ar@{..} "t2"; (67.5,25)*{}};%
{\ar@{..} "t2"; (67.5,15)*{}};%
{\ar@{..} "t2"; (52.5,15)*{}};%
{\ar@{=>} (76,40)*{}; (90,40)*{}};
(100,60)*+{\otimes}="m0";%
(100,40)*+{\bullet}="m1";%
(115,30)*+{\bullet}="m2";%
(100,20)*+{\bullet}="m3";%
{\ar^{1} "m0";"m1"};%
{\ar^{2} "m1";"m2"};%
{\ar^{1} "m3";"m2"};%
{\ar^{1} "m1";"m3"};%
{\ar@(dl,dr)_{2} "m3";"m3"};%
\endxy
\par\end{center}
\par\end{centering}

\caption{$\overline{\G}_{X}\left(H\right)$ and $\G_{X}\left(H\right)$ for
$H=\langle x_{1}x_{2}x_{1}^{-3},x_{1}^{\;2}x_{2}x_{1}^{-2}\rangle\leq\F_{2}$.
The Schreier coset graph $\overline{\G}_{X}\left(H\right)$ is the
infinite graph on the left (the dotted lines represent infinite $4$-regular
trees). The basepoint {}``$\otimes$'' corresponds to the trivial
coset $H$, the vertex below it corresponds to the coset $Hx_{1}$,
the one further down corresponds to $Hx_{1}^{\;2}=Hx_{1}x_{2}x_{1}^{-1}$,
etc. The core graph $\G_{X}\left(H\right)$ is the finite graph on
the right, which is obtained from $\overline{\G}_{X}\left(H\right)$
by omitting all vertices and edges that are not traced by reduced
closed walks around the basepoint.}

\label{fig:coset_and_core_graphs} 
\end{figure}

If $\Gamma$ is a directed pointed graph labeled by some set $X$,
walks in $\Gamma$ correspond to words in $\mathbf{F}\left(X\right)$
(the free group generated by $X$). For instance, the walk (from left
to right)
\[
\xymatrix@C=35pt{\bullet\ar[r]^{x_{2}} & \bullet\ar[r]^{x_{2}} & \bullet\ar[r]^{x_{1}} & \bullet & \bullet\ar[l]_{x_{2}}\ar[r]^{x_{3}} & \bullet\ar[r]^{x_{2}} & \bullet & \bullet\ar[l]_{x_{1}}}
\]
corresponds to the word $x_{2}^{\;2}x_{1}x_{2}^{-1}x_{3}x_{2}x_{1}^{-1}$.
The set of all words obtained from closed walks around the basepoint
in $\Gamma$ is a subgroup of $\F\left(X\right)$ which we call the
\emph{labeled fundamental group }of $\Gamma$, and denote by $\pi_{1}^{X}\left(\Gamma\right)$\marginpar{$\pi_{1}^{X}\left(\Gamma\right)$}.
Note that $\pi_{1}^{X}\left(\Gamma\right)$ need not be isomorphic
to $\pi_{1}\left(\Gamma\right)$, the standard fundamental group of
$\Gamma$ viewed as a topological space - for example, take $\vphantom{\Big|}\Gamma=\xymatrix@1{\otimes\ar@(dl,ul)[]^{x_{1}}\ar@(dr,ur)[]_{x_{1}}}$.

However, it is not hard to show that when $\Gamma$ is a core graph,
then $\pi_{1}^{X}\left(\Gamma\right)$ \emph{is }isomorphic to $\pi_{1}\left(\Gamma\right)$
(e.g. \cite{KM02}). In this case the labeling gives a canonical identification
of $\pi_{1}\left(\Gamma\right)$ as a subgroup of $\mathbf{F}\left(X\right)$.
It is an easy observation that 
\begin{equation}
\pi_{1}^{X}\left(\overline{\G}_{X}\left(H\right)\right)=\pi_{1}^{X}\left(\G_{X}\left(H\right)\right)=H\label{eq:canon_iso}
\end{equation}
This gives a one-to-one correspondence between subgroups of $\mathbf{F}\left(X\right)=\F_{k}$
and core graphs labeled by $X$. Namely, $\pi_{1}^{X}$ and $\Gamma_{X}$
are the inverses of each other in a bijection (Galois correspondence)
\begin{equation}
\left\{ {\mathrm{Subgroups}\atop \mathrm{of}\,\mathbf{F}\left(X\right)}\right\} \:{\underrightarrow{\;\Gamma_{X}\;}\atop \overleftarrow{\;\pi_{1}^{X}\;}}\:\left\{ {\mathrm{Core\, graphs}\atop \mathrm{labeled\, by}\, X}\right\} .\label{eq:pi_1_gamma}
\end{equation}
Core graphs were introduced by Stallings \cite{Sta83}. Our definition
is slightly different, and closer to the one in \cite{KM02,MVW07}
in that we allow the basepoint to be of degree one, and in that our
graphs are directed and edge-labeled. 

\medskip{}

We now list some basic properties of core graphs which are used in
the sequel of this paper (proofs can be found in \cite{Sta83,KM02,MVW07,Pud14a}).
\begin{claim}
\label{cla:core-graphs-properties} Let $H$ be a subgroup of $\F_{k}$
with an associated core graph $\G=\G_{X}\left(H\right)$. 
\begin{enumerate}
\item \label{enu:finite-rk-graph}$\rk\left(H\right)<\infty\Longleftrightarrow\G$
is finite.
\item \label{enu:euler}$\rk\left(H\right)=\left|E\left(\Gamma\right)\right|-\left|V\left(\Gamma\right)\right|+1$
for finitely generated (f.g.) subgroup $H$.
\item The correspondence \eqref{pi_1_gamma} restricts to a correspondence
between finitely generated subgroups of $\F_{k}$ and finite core
graphs. 
\end{enumerate}
\end{claim}

A \emph{morphism} between two core-graphs is a map that sends vertices
to vertices and edges to edges, and preserves the structure of the
core graphs. Namely, it preserves the incidence relations, sends the
basepoint to the basepoint, and preserves the directions and labels
of the edges. As in Claim \claref{core-graphs-properties}, each of
the following properties is either proven in (some of) \cite{Sta83,KM02,MVW07,Pud14a}
or is an easy observation:
\begin{claim}
\label{cla:morphism-properties} Let $H,J,L\le\F_{k}$ be subgroups.
Then 
\begin{enumerate}
\item A morphism $\G_{X}\left(H\right)\to\G_{X}\left(J\right)$ exists if
and only if $H\leq J$. 
\item If a morphism $\G_{X}\left(H\right)\to\G_{X}\left(J\right)$ exists,
it is unique. We denote it by $\eta_{H\to J}^{X}$\marginpar{$\eta_{H\to J}^{X}$}.
\item Whenever $H\le L\le J$, $\eta_{H\to J}^{X}=\eta_{L\to J}^{X}\circ\eta_{H\to L}^{X}$.%
\footnote{Points (1)-(3) can be formulated by saying that \eqref{pi_1_gamma}
is in fact an isomorphism of categories, given by the functors $\pi_{1}^{X}$
and $\Gamma_{X}$.%
}
\item If $\Gamma_{X}\left(H\right)$ is a subgraph of $\Gamma_{X}\left(J\right)$,
namely if $\eta_{H\to J}^{X}$ is injective, then $H\ff J$.%
\footnote{But not vice-versa: for example, consider $\left\langle x_{1}x_{2}^{\,2}\right\rangle \ff\F_{2}$.%
}
\item Every morphism in an immersion (locally injective at the vertices). 
\end{enumerate}
\end{claim}

\section{Counting Words and Critical Subgroups\label{sec:counting-words-and-algebraic-extensions}}

In this section we bound the exponential growth rate (as $t\to\infty$)
of
\[
\sum_{w\in\cptm\left(\Omega\right)}\left|\crit\left(w\right)\right|.
\]
For the special case of the bouquet with $k=\frac{d}{2}$ loops, where
$\cpt\left(B_{\frac{d}{2}}\right)=\left(X\cup X^{-1}\right)^{t}$,
we find the accurate exponential growth rate. The bound for a general
graph $\Omega$ is given in terms of the spectral radius $\rho=\rho_{A}\left(\Omega\right)$
of the universal covering tree of $\Omega$. 

We begin with a key lemma to be used in the proofs of all cases (a
bouquet, a $d$-regular base graph and an arbitrary base graph):
\begin{lem}
\label{lem:each-edge-twice}Let $w\in\F_{k}$ and let $N\leq_{f.g.}\F_{k}$
be a proper algebraic extension of $\left\langle w\right\rangle $.\emph{
Then the closed walk in $\Gamma_{X}\left(N\right)$ corresponding
to $w$ traces every edge at least twice.}\end{lem}
\begin{proof}
First, we claim that every edge is traced at least once (in fact,
even more generally, if $H\alg N$ then $\eta_{H\to N}^{X}$ is onto:
see Definition \ref{def:quotient} and e.g.~\cite[Claim 4.2]{PP15}.
We repeat the simple argument here.) Otherwise, let $J$ be the subgroup
of $N$ corresponding to the subgraph $\Delta$ traced by $w$ (so
$\Delta=\mathrm{im}\,\eta_{\left\langle w\right\rangle \to N}^{X}$),
and $J=\pi_{1}^{X}\left(\Delta\right)$, see Section \subref{Core-Graphs}
and in particular Claim \ref{cla:morphism-properties}). Then $w\in J\lneq_{ff}N$
(Claim \claref{morphism-properties}), contradicting the fact that
$N$ is an algebraic extension of $\left\langle w\right\rangle $.

Next, we distinguish between \emph{separating} edges and \emph{non-separating}
edges in $\Gamma=\Gamma_{X}\left(N\right)$. If $e$ is a separating
edge, namely if removing $e$ separates $\Gamma$ into two connected
components, then it is obvious that the walk of $w$ in $\Gamma$
must traverse $e$ an even number of times, and since this number
is $\geq1$, it is in fact $\geq2$.

Finally, assume that $e$ is not separating, and $w$ traverses it
exactly once, so that the walk corresponding to $w$ in $\Gamma_{X}\left(N\right)$
is $w_{1}ew_{2}$ (with $w_{1},w_{2}$ avoiding $e$; we think of
$e$ as oriented according to the direction of $w$). Choose a spanning
tree $T$ of $\Gamma_{X}\left(N\right)$ which avoids $e$ to obtain
a basis for $N$ as follows. There are $r=\rk\left(N\right)$ excessive
edges $e=e_{1},e_{2},\ldots,e_{r}$ outside the tree, and they should
be oriented arbitrarily. For each $1\leq i\leq r$ let $u_{i}$ be
the word corresponding to the walk that goes from $\otimes$ to the
origin of $e_{i}$ via $T$, then traverses $e_{i}$ and returns to
$\otimes$ via $T$. It is easy to see that $\left\{ u_{1},\ldots,u_{r}\right\} $
is a basis of $N$. We claim that so is $\left\{ w,u_{2},\ldots,u_{r}\right\} $,
so that $w$ is primitive in $N$ and therefore $\left\langle w\right\rangle \ff N$,
a contradiction. 

It is enough to show that $u_{1}\in\left\langle w,u_{2},\ldots,u_{r}\right\rangle $
(see footnote on Page \pageref{fn:basis}). Let $p_{1}$ be the walk
through $T$ from $\otimes$ to the origin of $e$, and $p_{2}$ the
walk from the terminus of $e$ back to $\otimes$. Then
\[
u_{1}=p_{1}ep_{2}=p_{1}w_{1}^{-1}w_{1}ew_{2}w_{2}^{-1}p_{2}=\left(p_{1}w_{1}^{-1}\right)w\left(w_{2}^{-1}p_{2}\right)
\]
and we are done because $p_{1}w_{1}^{-1}$ and $w_{2}^{-1}p_{2}$
avoid $e$ and thus belong to $\left\langle u_{2},\ldots,u_{r}\right\rangle $.
\end{proof}
We will also use the following simple properties of the core graph
of a subgroup of rank $m$. A {}``topological edge'' of a graph
is an edge of the graph obtained after ignoring all vertices of degree
2, except for (possibly) the basepoint $\otimes$. 
\begin{claim}
\label{Claim:core-graphs-of-rank-m}Let $\Gamma=\Gamma_{X}\left(J\right)$
be the core graph of a subgroup $J\leq\F_{k}$ of rank $m$. Then,
\begin{enumerate}
\item \label{enu:deg-at-most-2m}After omitting the string to $\otimes$
if the basepoint is a leaf, all vertices of $\Gamma$ are of degree
at most $2m$.
\item \label{enu:at-most-3m-1-edges}$\Gamma$ has at most $3m-1$ topological
edges.
\end{enumerate}
\end{claim}
\begin{proof}
\emph{(1)} After ignoring $\otimes$ and the string leading to $\otimes$
in case it is a leaf, all vertices of $\Gamma$ are of degree $\geq2$
. Thus all summands in the l.h.s.~of 
\[
\sum_{v\in V\left(\Gamma\right)}\left[\deg\left(v\right)-2\right]=2\left|E\left(\Gamma\right)\right|-2\left|V\left(\Gamma\right)\right|=2m-2
\]
are non-negative. So the degree of every vertex is bounded by $2+\left(2m-2\right)=2m$.
In fact, there is a vertex of degree $2m$ if and only if $\Gamma$
is topologically a bouquet of $m$ loops (plus, possibly, a string
to $\otimes$).

\emph{(2) }Consider $\Gamma$ as a {}``topological graph'' as explained
above. Let $e$ and $v$ denote the number of topological edges and
vertices. It is still true that $e-v+1=m$, but now there are no vertices
of degree $\leq2$ except for, possibly, the basepoint. Therefore,
the sum of degrees, which equals $2e$, is at least $3\left(v-1\right)+1$.
So 
\[
2e\geq3\left(v-1\right)+1=3\left(e-m\right)+1
\]
so $e\leq3m-1$.
\end{proof}

\subsection{The special case of the bouquet\label{sub:counting-bouquet}}

For the special case where $\Omega=B_{\frac{d}{2}}$ is the bouquet
of $k=\frac{d}{2}$ loops, our goal is to bound the exponential growth
rate of 
\[
\sum_{w\in\cptm\left(B_{\frac{d}{2}}\right)}\left|\crit\left(w\right)\right|=\sum_{w\in\left(X\cup X^{_{-1}}\right)^{t}:\,\pi\left(w\right)=m}\left|\crit\left(w\right)\right|.
\]
In order to estimate this number we first estimate the exponential
growth rate of the parallel quantity for \emph{reduced} words:
\begin{prop}
\label{prop:reduced-words-upper-bound} Let $k\geq2$ and $m\in\left\{ 1,2,\ldots,k\right\} $.
Then 
\[
\limsup_{t\to\infty}\left[\sum_{\substack{w\in\F_{k}:\\
\left|w\right|=t\,\&\,\pi\left(w\right)=m
}
}\left|\crit\left(w\right)\right|\right]^{1/t}\leq\begin{cases}
\sqrt{2k-1} & 2m-1\leq\sqrt{2k-1}\\
2m-1 & 2m-1\geq\sqrt{2k-1}
\end{cases}.
\]

\end{prop}
Put differently, the $\limsup$ is bounded by $\max\left\{ \sqrt{2k-1},2m-1\right\} $
(we present it in a lengthier way to stress the threshold phenomenon).
In fact, this is not only an upper bound but the actual exponential
growth rate - see Theorem \ref{thm:prim-rank-distr-on-reduced-words}.
\begin{proof}
Note that
\begin{eqnarray}
\sum_{\substack{w\in\F_{k}:\\
\left|w\right|=t\,\&\,\pi\left(w\right)=m
}
}\left|\crit\left(w\right)\right| & = & \sum_{J\leq\F_{k}:\,\rk\left(J\right)=m}\left|\left\{ w\in\F_{k}\,\middle|\,\left|w\right|=t,\, J\in\crit\left(w\right)\right\} \right|\nonumber \\
 & \leq & \sum_{J\leq\F_{k}:\,\rk\left(J\right)=m}\left|\left\{ w\in\F_{k}\,\middle|\,\left|w\right|=t,\,\left\langle w\right\rangle \neqalg J\right\} \right|\label{eq:from-crit-to-tracing-twice}\\
 & \leq & \sum_{J\leq\F_{k}:\,\rk\left(J\right)=m}\left|\left\{ w\in J\,\middle|\,\begin{gathered}\left|w\right|=t,\, w\,\mathrm{traces\, each\, edge}\\
\mathrm{of}\,\Gamma_{X}\left(J\right)\,\mathrm{at\, least\, twice}
\end{gathered}
\right\} \right|,\nonumber 
\end{eqnarray}
where the first inequality stems from Claim \ref{claim:critical-is-algebraic}
and the second from Lemma \ref{lem:each-edge-twice}. We continue
to bound the latter sum. For each $J\leq\F_{k}$ let $\nu_{t}\left(J\right)$
denote the corresponding summand: 

\[
\nu_{t}\left(J\right)=\left|\left\{ w\in J\,\middle|\,\begin{gathered}\left|w\right|=t,\, w\,\mathrm{traces\, each\, edge}\\
\mathrm{of}\,\Gamma_{X}\left(J\right)\,\mathrm{at\, least\, twice}
\end{gathered}
\right\} \right|.
\]

We classify all $J$'s of rank $m$ by the number of edges in $\Gamma_{X}\left(J\right)$.
Consider all $X$-labeled core-graphs $\Gamma$ of total size $\delta t$
and rank $m$ (so that $\delta t$ is an integer, of course). Since
we count words of length $t$ tracing every edge at least twice, $\nu_{t}\left(J\right)=0$
if $\delta>\frac{1}{2}$. So we restrict to the case $\delta\in\left[0,\frac{1}{2}\right]$.
The counting is performed in several steps:
\begin{itemize}
\item First, let us bound the number of unlabeled and unoriented connected
pointed graphs with $\delta t$ edges and rank $m$ (here the rank
of a connected graph is $e-v+1$). As in the proof of Lemma \ref{lem:each-edge-twice},
each such graph has some spanning tree and $m$ excessive edges. The
walks through the tree from $\otimes$ to the origins and termini
of these edges cover \emph{the entire }tree. Denote these walks by
$p_{1,1},p_{1,2},p_{2,1},p_{2,2},\ldots,p_{m,1},p_{m,2}$. We {}``unveil''
the spanning tree step by step: first we unveil $p_{1,1}$. The only
unknown is its length $\in\left\{ 0,1,\ldots,\delta t-1\right\} $.
Then $p_{1,2}$ leaves $p_{1,1}$ at one of $\leq\delta t$ possible
vertices and goes on for some length $<\delta t$. Now, $p_{2,1}$
leaves $p_{1,1}\cup p_{1,2}$ at one of $\leq\delta t$ possible vertices
and goes on for $<\delta t$ new edges. This goes on $2m$ times in
total (afterward, the ends of $p_{i,1}$ and $p_{i,2}$ are connected
by an edge). In total, there are at most $\left(\left(\delta t\right)^{2}\right)^{2m}=\left(\delta t\right)^{4m}$
possible unlabeled pointed graphs of rank $m$ with $\delta t$ edges%
\footnote{A tighter bound of $\left(\delta t\right)^{3m}$ can also be obtained
quite easily. We do not bother to introduce it because this expression
is anyway negligible when exponential growth rate is considered.%
}. 
\item Next, we bound the number of labelings of each such graph $\Gamma$
(here, the labeling includes also the orientation of each edge). Label
some edge (there are $2k$ options) and then gradually label edges
adjacent to at least one edge which is already labeled (at most $2k-1$
possible labels for each edge). Over all the number of possible labelings
of $\Gamma$ is $\leq2k\cdot\left(2k-1\right)^{\delta t-1}$. 
\item For a given labeled core-graph $\Gamma$, let $J=\pi_{1}^{X}\left(\Gamma\right)$
be the corresponding subgroup. We claim that $\nu_{t}\left(J\right)\leq\left(4t^{2}\right)^{3m-1}\cdot\left(2m-1\right)^{\left(1-2\delta\right)t}$.
Indeed, note first that if the basepoint $\otimes$ is a leaf, then
every reduced $w$ must first follow the string from $\otimes$ to
the first {}``topological'' vertex (vertex of degree $\geq3$),
and then return to the string only in its final steps back to $\otimes$.
So we can assume $w$ traces a leaf-free graph of rank $m$ and at
most $\delta t$ edges. A reduced word $w\in J$ which traces every
edge at least twice, also traverses any topological edge at least
twice, each time in one shot (without backtracking). Each time $w$
traces some topological edge $\widetilde{e}$ in $\Gamma$, it begins
in one of $\leq t$ possible positions (in $w$), and from $\leq2$
possible directions of $\widetilde{e}$. So there $\leq4t^{2}$ possible
ways in which $w$ traces $\widetilde{e}$ for the first two times.
By Claim \ref{Claim:core-graphs-of-rank-m}(\ref{enu:at-most-3m-1-edges})
there are at most $3m-1$ topological edges, and so at most $\left(4t^{2}\right)^{3m-1}$
possibilities for how $w$ traces each topological edge of $\Gamma$
for the first two times. The rest of $w$ is of length (at most) $\left(1-2\delta\right)t$,
and in every step there are at most $2m-1$ ways to proceed, by Claim
\ref{Claim:core-graphs-of-rank-m}(\ref{enu:deg-at-most-2m}).
\end{itemize}
Hence,
\begin{eqnarray}
\sum_{\substack{J\leq\F_{k}:\, rk\left(J\right)=m\\
\left|\Gamma_{X}\left(J\right)\right|=\delta t
}
}\nu_{t}\left(J\right) & \leq & \left(\delta t\right)^{4m}\cdot2k\left(2k-1\right)^{\delta t-1}\cdot\left(4t^{2}\right)^{3m-1}\left(2m-1\right)^{\left(1-2\delta\right)t}\nonumber \\
 & \leq & c\cdot t^{10m-2}\cdot\left[\left(2k-1\right)^{\delta}\left(2m-1\right)^{1-2\delta}\right]^{t}\nonumber \\
 & = & c\cdot t^{10m-2}\cdot\left[\left(\frac{2k-1}{\left(2m-1\right)^{2}}\right)^{\delta}\left(2m-1\right)\right]^{t}.\label{eq:bound-for-reduced-words-1}
\end{eqnarray}
Recall that $\delta\in\left[0,\frac{1}{2}\right]$ and $\delta t\in\mathbb{N}$.
We bound $\sum_{J\leq\F_{k}:\, rk\left(J\right)=m}\nu_{t}\left(J\right)$
by $\frac{t}{2}$ times the maximal possible value of the r.h.s.~of
\eqref{bound-for-reduced-words-1} (when going over all possible values
of $\delta$). When $2m-1\leq\sqrt{2k-1}$, the r.h.s.~of \eqref{bound-for-reduced-words-1}
is largest when $\delta=\frac{1}{2}$, so we get overall
\begin{equation}
\sum_{J\leq\F_{k}:\, rk\left(J\right)=m}\nu_{t}\left(J\right)\leq c\cdot t^{10m-1}\cdot\left[\sqrt{2k-1}\right]^{t}.\label{eq:bound-for-reduced-words2-1}
\end{equation}
For $2m-1\geq\sqrt{2k-1}$, the r.h.s.~of \eqref{bound-for-reduced-words-1}
is largest when $\delta=0$, so we get overall
\[
\sum_{J\leq\F_{k}:\, rk\left(J\right)=m}\nu_{t}\left(J\right)\leq c\cdot t^{10m-1}\cdot\left[2m-1\right]^{t}.
\]
The proposition follows.
\end{proof}
The next step is to deduce an analogue result for non-reduced words.
To this goal, we use an extended version of the well known \emph{cogrowth
formula} due to Grigorchuk \cite{Gri80} and Northshield \cite{Nor92}.
Let $\Gamma$ be a connected $d$-regular graph. Let $b_{\Gamma,v}\left(t\right)$
denote the number of \emph{cycles} of length $t$ at some vertex $v$
in $\Gamma$, and let $n_{\Gamma,v}\left(t\right)$ denote the size
of the smaller set of \emph{non-backtracking cycles }of length $t$
at $v$. The spectral radius of $A_{\Gamma}$, denoted $\mathrm{rad}\left(\Gamma\right)$%
\footnote{If $\Gamma$ is finite, $\mathrm{rad\left(\Gamma\right)}=d$. If $\Gamma$
is the $d$-regular tree, $\mathrm{rad}\left(\Gamma\right)=2\sqrt{d-1}$.%
}, is equal to $\limsup_{t\to\infty}b_{\Gamma,v}\left(t\right)^{1/t}$
(in particular, this limit does not depend on $v$). The \emph{cogrowth}
of $\Gamma$ is defined as \marginpar{$\mathrm{cogr}\left(\cdot\right)$}$\mathrm{cogr}\left(\Gamma\right)=\limsup_{t\to\infty}n_{\Gamma,v}\left(t\right)^{1/t}$,
and is also independent of $v$. 

The cogrowth formula expresses $\mathrm{rad}\left(\Gamma\right)$
in terms of $\mathrm{cogr}\left(\Gamma\right)$: it determines that
$\mathrm{rad}\left(\Gamma\right)=g\left(\mathrm{cogr\left(\Gamma\right)}\right)$,
where $g:\left[1,d-1\right]\to\left[2\sqrt{d-1},d\right]$ is defined
by
\begin{equation}
g\left(\alpha\right)=\begin{cases}
2\sqrt{d-1} & \alpha\leq\sqrt{d-1}\\
\frac{d-1}{\alpha}+\alpha & \alpha\geq\sqrt{d-1}
\end{cases}.\label{eq:g}
\end{equation}

Another way to view the parameters $\mathrm{rad}\left(\Gamma\right)$
and $\mathrm{cogr}\left(\Gamma\right)$ is the following: let $T_{d}$
be the $d$-regular tree with basepoint $\otimes$, let $p:T_{d}\to\Gamma$
be a covering map such that $p\left(\otimes\right)=v$, and let $S=p^{-1}\left(v\right)\subseteq V\left(T_{d}\right)$
be the fiber above $v$. Then $b_{\Gamma,v}\left(t\right)$ is the
number of walks of length $t$ in $T_{d}$ emanating from $\otimes$
and terminating inside $S$. Similarly, $n_{\Gamma,v}\left(t\right)$
is the number of non-backtracking walks of length $t$ in $T_{d}$
emanating from $\otimes$ and terminating in $S$. This is also equal
to the number of vertices in the $t$-th sphere%
\footnote{The $t$-th sphere of the pointed $T_{d}$ is the set of vertices
at distance $t$ from $\otimes$.%
} of $T_{d}$ belonging to $S$.

For our needs we introduce (in a separate paper - \cite{Pud15+})%
\footnote{The results in \cite{Pud15+} include a new proof of the original
cogrowth formula.%
} an extended formula applying to other types of subsets $S$ of $V\left(T_{d}\right)$,
which do not necessarily correspond to a fiber of a covering map of
a graph. Even more generally, we extend the formula to a class of
functions on $V\left(T_{d}\right)$ (this extends the previous case
if $S$ is identified with its characteristic function $\mathbbm{1}_{S}$): 

For $f:V\left(T_{d}\right)\to\mathbb{R}$, denote by $\beta_{f}\left(t\right)$\marginpar{$\beta_{f}\left(t\right)$}
the sum 
\[
\beta_{f}\left(t\right)=\sum_{\substack{p:\,\,\mathrm{a\, path\, from}\,\otimes\\
\mathrm{of\, length}\, t
}
}f\left(\mathrm{end}\left(p\right)\right)
\]
over all (possibly backtracking) walks of length $t$ in $T_{d}$
emanating from $\otimes$. Similarly, denote by $\nu_{f}\left(t\right)$\marginpar{$\nu_{f}\left(t\right)$}
the same sum over the smaller set of \emph{non-backtracking} walks
of length $t$ emanating from $\otimes$.
\begin{thm}
\label{thm:cogrowth-formula-extended}[Extended Cogrwoth Formula \cite{Pud15+}]
Let $d\geq3$, $f:V\left(T_{d}\right)\to\mathbb{R}$, $\beta_{f}\left(t\right)$
and $\nu_{f}\left(t\right)$ as above. If $\nu_{f}\left(t\right)\leq c\cdot\alpha^{t}$
for every $t$ (and some $c>0$) then 
\[
\limsup_{t\to\infty}\beta_{f}\left(t\right)^{1/t}\leq g\left(\alpha\right).
\]

\end{thm}
With this theorem at hand, one can obtain the sought-after bound on
the number of non-reduced words from the one on reduced words:
\begin{cor}
\label{cor:non-reduced-words_upper_bound}For every $k\geq2$ and
$m\in\left\{ 1,\ldots,k\right\} $, 
\[
\limsup_{t\to\infty}\left[\sum_{w\in\cptm\left(B_{\frac{d}{2}}\right)}\left|\crit\left(w\right)\right|\right]^{1/t}\leq\begin{cases}
2\sqrt{2k-1} & 2m-1\leq\sqrt{2k-1}\\
\frac{2k-1}{2m-1}+2m-1 & 2m-1\geq\sqrt{2k-1}
\end{cases}.
\]
\end{cor}
\begin{proof}
Consider the the Cayley graph of $\F_{k}$ which is a $2k$-regular
tree. Every vertex corresponds to a word in $\F_{k}$, and we let
$f_{m}\left(w\right)=\mathbbm{1}_{\pi\left(w\right)=m}\left|\crit\left(w\right)\right|$.
The corollary then follows by applying Theorem \ref{thm:cogrowth-formula-extended}
on $f_{m}$, using Proposition \ref{prop:reduced-words-upper-bound}.
\end{proof}
In Section \ref{sec:distr-of-prim-rank} it is shown (Theorem \ref{thm:prim-rank-distr})
that the bound in Corollary \ref{cor:non-reduced-words_upper_bound}
represents the accurate exponential growth rate of the sum, and even
merely of the number of not-necessarily-reduced words with primitivity
rank $m$. This result uses further results from \cite{Pud15+}.
\begin{rem}
Interestingly, the threshold of $\sqrt{2k-1}$ shows up twice, apparently
independently, both in Proposition \ref{prop:reduced-words-upper-bound}
and in the (extended) cogrowth formula.
\end{rem}
Finally, for $m=0$ there is exactly one relevant reduced word: $w=1$,
and this word has exactly one critical subgroup: the trivial subgroup.
Thus, it suffices to bound the number of words in $\left(X\cup X^{-1}\right)^{t}$
reducing to $1$. This is a well-known result:
\begin{claim}
\label{claim:bound-for-m=00003D0} 
\[
\limsup_{t\to\infty}\left|\cpt^{0}\left(B_{\frac{d}{2}}\right)\right|^{1/t}=\limsup_{t\to\infty}\left|\left\{ w\in\left(X\cup X^{-1}\right)^{t}\,\middle|\, w\,\mathrm{reduces\, to}\,1\right\} \right|^{1/t}=2\sqrt{2k-1}.
\]
\end{claim}
\begin{proof}
Denote by $c_{\Gamma}\left(t,u,v\right)$\marginpar{$c_{\Gamma}\left(t,u,v\right)$}
the number of walks of length $t$ from the vertex $u$ to the vertex
$v$ in a connected graph $\Gamma$. If, as above, $A_{\Gamma}$ denotes
the adjacency operator on $l^{2}\left(V\left(\Gamma\right)\right)$,
then $c_{\Gamma}\left(t,u,v\right)=\left\langle A_{\Gamma}^{\,\, t}\delta_{u},\delta_{v}\right\rangle _{1}$
($\left\langle \cdot,\cdot\right\rangle _{1}$ marks the standard
inner product). If $\Gamma$ has bounded degrees, then $A_{\Gamma}$
is a bounded self-adjoint operator, hence
\begin{equation}
\mathrm{rad}\left(\Gamma\right)=\left\Vert A_{\Gamma}\right\Vert =\limsup_{t\to\infty}c_{\Gamma}\left(t,u,v\right)^{1/t}\label{eq:spec-rad=00003Dlim-1}
\end{equation}
 for every $u,v\in V\left(\Gamma\right)$. Moreover, 
\begin{equation}
c_{\Gamma}\left(t,u,v\right)=\left\langle A_{\Gamma}^{\, t}\delta_{u},\delta_{v}\right\rangle _{1}\leq\left\Vert A_{\Gamma}^{\, t}\delta_{u}\right\Vert \cdot\left\Vert \delta_{v}\right\Vert \leq\left\Vert A_{\Gamma}^{\,}\right\Vert ^{t}\cdot\left\Vert \delta_{u}\right\Vert \cdot\left\Vert \delta_{v}\right\Vert =\mathrm{rad}\left(\Gamma\right)^{t}\label{eq:c_Gamma}
\end{equation}
(For these facts and other related ones we refer the reader to \cite[\textsection 6]{LP:book}).

The words of length $t$ reducing to $1$ are exactly the closed walks
of length $t$ at the basepoint of the $2k$-regular tree $T_{2k}$.
So the number we seek is\linebreak{}
 $\limsup_{t\to\infty}c_{T_{2k}}\left(t,v,v\right)^{1/t}$, which
therefore equals $\mathrm{rad}\left(T_{2k}\right)=2\sqrt{2k-1}.$
\end{proof}

\subsection{An arbitrary regular base-graph $\Omega$\label{sub:d-reg-base}}

We proceed with the observation that when $\Omega$ is $d$-regular
(but not necessarily the bouquet), the bounds from Corollary \ref{cor:non-reduced-words_upper_bound}
generally apply. We begin with a few claims that will be useful also
in the next subsection dealing with irregular base graphs. 

Let $\rk\left(\Omega\right)$\marginpar{$\rk\left(\Omega\right)$}
denote the rank of the fundamental group of a finite graph $\Omega$,
so $\rk\left(\Omega\right)=\left|E\left(\Omega\right)\right|-\left|V\left(\Omega\right)\right|+1$.
We claim there are no words in $\cpt\left(\Omega\right)$ admitting
finite primitivity rank which is greater than $\rk\left(\Omega\right)$:
\begin{lem}
\label{lem:prim-rank-at-most-rk-Omega}Let $\Omega$ be a finite,
connected graph. Then $\pi\left(w\right)\in\left\{ 0,1,\ldots,\rk\left(\Omega\right),\infty\right\} $
for every $w\in\cpt\left(\Omega\right)$.\end{lem}
\begin{proof}
Recall from Section \ref{sec:Overview-of-the-proof} that we denote
$k=\left|E\left(\Omega\right)\right|$ and orient each of the $k$
edges arbitrarily and label them by $x_{1},\ldots,x_{k}$. With the
orientation and labeling of its edges, $\Omega$ becomes a non-pointed
$X$-labeled graph, where $X=\left\{ x_{1},\ldots,x_{k}\right\} $.
(This is not a core-graph, for it has no basepoint and may have leaves.)
So every walk in $\Omega$ of length $t$ can be regarded as an element
of $\left(X\cup X^{-1}\right)^{t}$ and (after reduction) of $\F_{k}=\F\left(X\right)$.
If a word $w\in{\cal CW}_{t}\left(\Omega\right)$ begins (and ends)
at $v\in V\left(\Omega\right)$, then $w\in J_{v}$, where $J_{v}=\pi_{1}^{X}\left(\Omega_{v}\right)$\marginpar{$J_{v},\Omega_{v}$}
is the subgroup of $\F_{k}$ corresponding to the $X$-labeled graph
$\Omega$ pointed at $v$. The rank of $J_{v}$ is independent of
$v$ and equals $\rk\left(\Omega\right)$. It is easy to see that
$J_{v}\ff\F_{k}$ (recall that `$\ff$' denotes a free factor): obtain
a basis for $J_{v}$ by choosing an arbitrary spanning tree and orienting
the edges outside the tree, as in the proof of Lemma \ref{lem:each-edge-twice}.
This basis can then be extended to a basis of $\F_{k}$ by the $x_{i}$'s
associated with the edges inside the spanning tree. So if $w$ is
primitive in $J_{v}$, is it also primitive in $\F_{k}$ and $\pi\left(w\right)=\infty$.
Otherwise, $\pi\left(w\right)\leq\rk\left(J_{v}\right)=\rk\left(\Omega\right)$. 
\end{proof}
Moreover, proper algebraic extensions of words in $\cpt\left(\Omega\right)$
are necessarily subgroups of $J_{v}$ for some $v\in V\left(\Omega\right)$:
\begin{claim}
\label{claim: only-subgroups-matter}In $w\in\cpt\left(\Omega\right)$
is a cycle around the vertex $v$ and $\left\langle w\right\rangle \neqalg N$,
then $N\leq J_{v}$.\end{claim}
\begin{proof}
As $J_{v}\ff\F_{k}$, it follows that $J_{v}\cap N\ff N$ (see e.g.~\cite[Claim 3.9]{PP15}).
So if $w$ belongs to $N$, it belongs to the free factor $J_{v}\cap N$
of $N$, which is proper, unless $N\leq J_{v}$.
\end{proof}
If $\Omega$ is $d$-regular, $\left|E\left(\Omega\right)\right|=\frac{d}{2}\left|V\left(\Omega\right)\right|$
so that $\rk\left(\Omega\right)=\left(\frac{d}{2}-1\right)\left|V\left(\Omega\right)\right|+1\ge\frac{d}{2}$
(with equality only for the bouquet). The following Corollary distinguishes
between three classes of primitivity rank: the interval $0,1\ldots,\left\lfloor \frac{\sqrt{d-1}+1}{2}\right\rfloor $,
the interval $\left\lceil \frac{\sqrt{d-1}+1}{2}\right\rceil ,\ldots,\left\lfloor \frac{d}{2}\right\rfloor $
and $\left\lceil \frac{d}{2}\right\rceil ,\ldots,\rk\left(\Omega\right)$.
\begin{cor}
\label{cor:bounds-for-non-red-words-and-reg-base-graph}Let $\Omega$
be a finite, connected $d$-regular graph, and let $m\in\left\{ 0,1,\ldots,\rk\left(\Omega\right)\right\} $.
Then 
\[
\limsup_{t\to\infty}\left[\sum_{w\in\cptm\left(\Omega\right)}\left|\crit\left(w\right)\right|\right]^{1/t}\leq\begin{cases}
2\sqrt{d-1} & 2m-1\in\left[-1,\sqrt{d-1}\right]\\
\frac{d-1}{2m-1}+2m-1 & 2m-1\in\left[\sqrt{d-1},d-1\right]\\
d & 2m-1\in\left[d-1,2\mathrm{rk\left(\Omega\right)}-1\right]
\end{cases}.
\]
\end{cor}
\begin{proof}
First, for words with $\pi\left(w\right)=0$, that is, words reducing
to 1, their number is $\left|V\left(\Omega\right)\right|$ times the
number of cycles of length $t$ at a fixed vertex in the $d$-regular
tree. Thus, as in the proof of Claim \ref{claim:bound-for-m=00003D0},
\[
\limsup_{t\to\infty}\left[\sum_{w\in\cpt^{0}\left(\Omega\right)}\left|\crit\left(w\right)\right|\right]^{1/t}=\limsup_{t\to\infty}\left|\cpt^{0}\left(\Omega\right)\right|^{1/t}=2\sqrt{d-1}\cdot\limsup_{t\to\infty}\left|V\left(\Omega\right)\right|^{1/t}=2\sqrt{d-1}.
\]
For $m\geq1,$ since the extended cogrowth formula (Theorem \ref{thm:cogrowth-formula-extended})
applies here too, it is enough to prove that for \emph{reduced} words
we have: 
\[
\limsup_{t\to\infty}\left[\sum_{\substack{w\in\cptm\left(\Omega\right):\\
w\,\mathrm{is\, reduced}
}
}\left|\crit\left(w\right)\right|\right]^{1/t}\leq\begin{cases}
\sqrt{d-1} & 2m-1\in\left[1,\sqrt{d-1}\right]\\
2m-1 & 2m-1\in\left[\sqrt{d-1},d-1\right]\\
d-1 & 2m-1\in\left[d-1,2\mathrm{rk\left(\Omega\right)-1}\right]
\end{cases}
\]
From Claim \ref{claim: only-subgroups-matter} we deduce that every
critical subgroup is necessarily a subgroup of $J_{v}=\pi_{1}^{X}\left(\Omega_{v}\right)$
for some vertex $v\in V\left(\Omega\right)$. As in the proof of Proposition
\ref{prop:reduced-words-upper-bound}, we denote 
\[
\nu_{t}\left(J\right)=\left|\left\{ w\in\F_{k}\,\middle|\,\begin{gathered}\left|w\right|=t,\, w\,\mathrm{traces\, each\, edge}\\
\mathrm{of}\,\Gamma_{X}\left(J\right)\,\mathrm{at\, least\, twice}
\end{gathered}
\right\} \right|
\]
 for every $J\leq\F_{k}$, and as in \eqref{from-crit-to-tracing-twice},
we obtain the bound:
\[
\sum_{\substack{w\in\cptm\left(\Omega\right):\\
w\,\mathrm{is\, reduced}
}
}\left|\crit\left(w\right)\right|\leq\sum_{v\in V\left(\Omega\right)}\sum_{J\leq J_{v}:\,\mathrm{rk}\left(J\right)=m}\nu_{t}\left(J\right).
\]
We carry the same counting argument as in the proof of Proposition
\ref{prop:reduced-words-upper-bound}:
\begin{itemize}
\item The first stage, where we count unlabeled and unoriented pointed graphs
of a certain size and rank remains unchanged.
\item For the second stage of labeling and orienting the graph, we first
choose $v$ ($\left|V\left(\Omega\right)\right|$ options), and then
we use the fact that whenever $J\leq J_{v}$, there is a core-graph
morphism $\eta:\Gamma_{X}\left(J\right)\to\Omega_{v}$, which is,
as always, an immersion (i.e.~locally injective). So we first label
an arbitrary edge incident to the basepoint $\otimes$, and this one
has to be labeled like one of the $d$ edges incident with $\otimes$
at $\Omega_{v}$. We then label gradually edges adjacent to at least
one already-labeled edge. Thus, the image of one of the endpoints
of the current edge under the core-graph morphism is already known,
and there are at most $d-1$ options to label the current edge. Overall,
the number of possible labelings is bounded by $\left|V\left(\Omega\right)\right|\cdot d\left(d-1\right)^{\delta t-1}$.
\item The third and last stage, where we estimate $\nu_{t}\left(J\right)$
for a particular $J$, is almost identical. The only difference is
that every vertex in $\Gamma_{X}\left(J\right)$ is of degree at most
$\min\left\{ 2m,d\right\} $, so overall we obtain $\nu_{t}\left(J\right)\leq\left(4t^{2}\right)^{3m-1}\cdot\left(\min\left\{ 2m,d\right\} -1\right)^{\left(1-2\delta\right)t}$. 
\end{itemize}

We conclude as in the proof of Proposition \ref{prop:reduced-words-upper-bound}.

\end{proof}

\subsection{An arbitrary base-graph $\Omega$\label{sub:counting-omega}}

We now return to the most general case of an arbitrary connected base
graph $\Omega$. Theorem \ref{thm:bound-for-m-in-general-Omega} below
is needed for proving the bound on the new spectrum of the adjacency
operator on $\Gamma$, the random covering of $\Omega$ in the ${\cal C}_{n,\Omega}$
model (the first part of Theorem \ref{thm:sqrt3-times-rho}). The
small variation needed for the second part of this theorem, dealing
with the Markov operator, is discussed in Section \ref{sub:The-Markov-Operator}. 

Recall that $T$\marginpar{$T$} denotes the universal covering of
$\Omega$ (and of $\Gamma$), and $\rho=\rho_{A}\left(\Omega\right)$
denotes the spectral radius of its adjacency operator. Recall also
that we denote $k=\left|E\left(\Omega\right)\right|$ and orient each
of the $k$ edges arbitrarily and label them by $x_{1},\ldots,x_{k}$.
With the orientation and labeling of its edges, $\Omega$ becomes
a non-pointed $X$-labeled graph, where $X=\left\{ x_{1},\ldots,x_{k}\right\} $.
Every walk in $\Omega$ of length $t$ can be regarded as an element
of $\left(X\cup X^{-1}\right)^{t}$ and (after reduction) of $\F_{k}=\F\left(X\right)$.
We also denoted $\rk\left(\Omega\right)=\left|E\left(\Omega\right)\right|-\left|V\left(\Omega\right)\right|+1$
and showed that $\pi\left(w\right)\in\left\{ 0,1,\ldots,\rk\left(\Omega\right),\infty\right\} $
for every $w\in\cpt\left(\Omega\right)$ (Lemma \ref{lem:prim-rank-at-most-rk-Omega}).
The main theorem of this subsection is the following:
\begin{thm}
\label{thm:bound-for-m-in-general-Omega}Let $\Omega$ be a finite,
connected graph, and let $m\in\left\{ 1,\ldots,\rk\left(\Omega\right)\right\} $.
Then 
\[
\limsup_{t\to\infty}\left[\sum_{w\in\cptm\left(\Omega\right)}\left|\crit\left(w\right)\right|\right]^{1/t}\leq\left(2m-1\right)\cdot\rho.
\]

\end{thm}
Before proceeding to the proof of this theorem, let us refer to the
case $m=0$ which is left out. These are words reducing to $1$, and
the trivial element of $\F_{k}$ has exactly one critical subgroup,
so $\sum_{w\in\cptm\left(\Omega\right)}\left|\crit\left(w\right)\right|$
equals $\left|\cpt^{0}\left(\Omega\right)\right|$.
\begin{claim}
\label{claim:bound-for-m=00003D0-general-Omega}~
\[
\limsup_{t\to\infty}\left|\cpt^{0}\left(\Omega\right)\right|^{1/t}=\rho.
\]
\end{claim}
\begin{proof}
For a given vertex $v\in V\left(\Omega\right)$, each cycle at $v$
of length $t$ reducing to $1$ lifts to a cycle in $T$ at $\widehat{v}$,
where $\widehat{v}\in p^{-1}\left(v\right)$ is some vertex at the
fiber above $v$ of the covering map $p:T\to\Omega$. The number of
cycles of length $t$ reducing to 1 at $v$ is thus $\left[A_{T}^{\,\, t}\delta_{\widehat{v}}\right]_{\widehat{v}}$,
and 
\[
\left[A_{T}^{\,\, t}\delta_{\widehat{v}}\right]_{\widehat{v}}=\left\langle A_{T}^{\,\, t}\delta_{\widehat{v}},\delta_{\widehat{v}}\right\rangle _{1}\leq\left\Vert A_{T}^{\,\, t}\right\Vert \cdot\left\Vert \delta_{\widehat{v}}\right\Vert ^{2}=\left\Vert A_{T}^{\,\, t}\right\Vert =\rho{}^{t}
\]
(the last equality follows from $A_{T}$ being self-adjoint), and
thus 
\[
\limsup_{t\to\infty}\left|\cpt^{0}\left(\Omega\right)\right|^{1/t}\leq\limsup_{t\to\infty}\left[\left|V\left(\Omega\right)\right|\cdot\rho{}^{t}\right]^{1/t}=\rho.
\]
To show there is actual equality, repeat the argument from Claim \ref{claim:bound-for-m=00003D0}. 
\end{proof}
\noindent We return to the proof of Theorem \ref{thm:bound-for-m-in-general-Omega}.
By Claim \ref{claim:critical-is-algebraic},
\begin{eqnarray}
\sum_{w\in\cptm\left(\Omega\right)}\left|\crit\left(w\right)\right| & = & \sum_{\substack{N\leq\F_{k}:\\
\rk\left(N\right)=m
}
}\left|\left\{ w\in{\cal CW}_{t}\left(\Omega\right)\,\middle|\, N\in\crit\left(w\right)\right\} \right|\nonumber \\
 & \leq & \sum_{\substack{N\leq\F_{k}:\\
\rk\left(N\right)=m
}
}\left|\left\{ w\in{\cal CW}_{t}\left(\Omega\right)\,\middle|\,\left\langle w\right\rangle \neqalg N\right\} \right|\label{eq:sum-over-betta-N}
\end{eqnarray}
and we actually bound the latter summation. For every $N\leq\F_{k}$,
we let $\beta_{t}\left(N\right)$\marginpar{$\beta_{t}\left(N\right)$}
denote the corresponding summand, namely 
\[
\beta_{t}\left(N\right)=\left|\left\{ w\in{\cal CW}_{t}\left(\Omega\right)\,\middle|\,\left\langle w\right\rangle \neqalg N\right\} \right|.
\]
Note that while a non-reduced element $w\in\cpt\left(\Omega\right)$
with $w\in N$ might not correspond to a close walk in $\Gamma_{X}\left(N\right)$,
it always does correspond to a close walk at the basepoint of the
Schreier coset graph $\overline{\Gamma}_{X}\left(N\right)$.

If $N\leq\F_{k}$ satisfies that the basepoint $\otimes$ of $\Gamma_{X}\left(N\right)$
is not a leaf, call $N$ and its core-graph CR\marginpar{CR} (cyclically
reduced). The following claim shows it is enough to consider CR subgroups.
\begin{claim}
\label{claim: CR-subgroups} If $N\leq\F_{k}$ is CR then
\[
\sum_{N'\,\mathrm{is\, conjugate\, to\, N}}\beta_{t}\left(N'\right)\leq t\beta_{t}\left(N\right).
\]
\end{claim}
\begin{proof}
The Schreier graphs of $N$ and of any conjugate of it differ only
by the basepoint. If $N'$ is some conjugate of $N$ and $w'\in{\cal CW}_{t}\left(\Omega\right)$
satisfies $\left\langle w'\right\rangle \neqalg N'$, then the walk
corresponding to $w'$ in the Schreier graph $\overline{\Gamma}_{X}\left(N'\right)$
must visit all vertices and edges of the core of $\overline{\Gamma_{X}\left(N'\right)}$,
and in particular the basepoint of $\overline{\Gamma}_{X}\left(N\right)$
(by Lemma \ref{lem:each-edge-twice}). So there is some cyclic rotation
$w$ of $w'$ satisfying $\left\langle w\right\rangle \neqalg N$
(clearly, $w$ also belongs to $\cpt\left(\Omega\right)$). On the
other hand, each such $w$ has at most $t$ possible cyclic rotations,
each of which corresponds to one $w'$ and one $N'$. 
\end{proof}
Next, we classify the subgroups $N\leq\F_{k}$ according to their
{}``topological'' core graph $\Lambda$. As implied in the short
discussion preceding Claim \ref{Claim:core-graphs-of-rank-m}, this
is the homeomorphism class of the pointed $\Gamma_{X}\left(N\right)$.
Namely, this is the graph obtained from $\Gamma_{X}\left(N\right)$
by ignoring vertices of degree two, except for (possibly) the basepoint.
As Claim \ref{claim: CR-subgroups} allows us to restrict to one CR
representative from each conjugacy class of subgroups in $\F_{k}$,
we also restrict attention to one CR representative $\Lambda$ of
each {}``conjugacy class'' of topological core graphs. Ignoring
the basepoints, any $\Lambda'$ in the {}``conjugacy class'' of
$\Lambda$ retracts to this representative. For example, we need exactly
three such representatives in rank 2, as shown in Figure \ref{fig:topo-graphs-rk-2}. 

\begin{figure}
\begin{center}
$\xymatrix@1{
\otimes \ar@{-}@(ul,dl) \ar@{-}@(ur,dr)  &&
\otimes \ar@{-}@(ul,dl) \ar@{-}[r] & \bullet \ar@{-}@(ur,dr) &&
\otimes \ar@{-}@/^2pc/[r] \ar@{-}@/_2pc/[r] \ar@{-}[r]  & \bullet
}$
\par\end{center}\caption{\label{fig:topo-graphs-rk-2}The three CR representatives of topological
graphs of rank 2: Figure-Eight, Barbell and Theta.}
\end{figure}
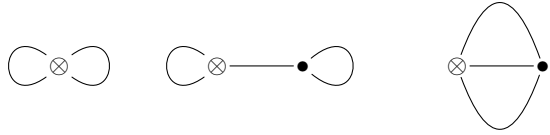
The following proposition is the key step in the proof of Theorem
\ref{thm:bound-for-m-in-general-Omega}.
\begin{prop}
\label{prop:key-prop-for-top-graph}Let $\Lambda$ be a pointed finite
connected graph without vertices of degree 1 or 2 except for possibly
the basepoint, and let $\delta$ denote its maximal degree. Then the
sum of $\beta_{t}\left(N\right)$ over all subgroup $N\leq\F_{k}$
whose core graph is topologically $\Lambda$ is at most 
\[
\left|V\left(\Omega\right)\right|\cdot\left(4t^{4}\right)^{\left|E\left(\Lambda\right)\right|}\cdot\left(\delta-1\right)^{t}\cdot\rho^{t}.
\]
\end{prop}
\begin{proof}
Denote $r=\left|E\left(\Lambda\right)\right|$. Order and orient the
edges of $\Lambda$ $\left\{ e_{1},e_{2},\ldots,e_{r}\right\} $ so
that $e_{1}$ emanates from $\otimes$, and for every $i\geq2$, $e_{i}$
emanates either from $\otimes$ or from a vertex which is the beginning
or endpoint of one of $e_{1},\ldots,e_{i-1}$. (This labeling and
orientation is usually not unique, but we fix one throughout this
proof.) In addition, let $v_{0}$ denote $\otimes$ and $v_{i}$ denote
the endpoint of $e_{i}$ for $1\leq i\leq r$. For example, one can
label the barbell-shaped graph as follows:~~$\xymatrix@1{\otimes \ar@(ul,dl)^{e_3} \ar[r]_{e_1} & \bullet \ar@(ur,dr)_{e_2} }$,
where $v_{0}=v_{3}$ are $\otimes$ and $v_{1}=v_{2}$ are $\bullet$.
Also, denote by $\mathrm{beg}\left(i\right)$ the smallest index $j$
such that $e_{i}$ begins at $v_{j}$, so $e_{i}$ is a directed edge
from $v_{\mathrm{beg}\left(i\right)}$ to $v_{i}$ and $\mathrm{beg\left(i\right)}<i$.
In our example, $\mathrm{beg}\left(1\right)=\mathrm{beg}\left(3\right)=0$
and $\mathrm{beg}\left(2\right)=1$.

Note that each $N$ corresponding to $\Lambda$ is determined by the
walks (words in $\F_{k}$) associated with $e_{1},\ldots,e_{r}$.
From Claim \ref{claim: only-subgroups-matter} it follows one can
restrict to subgroups $N$ which are subgroups of $J_{v}$ for some
$v\in V\left(\Omega\right)$. So fix some $v_{0}\in V\left(\Omega\right)$
and also some $\widehat{v}_{0}\in V\left(T\right)$ which projects
to $v_{0}$. We claim that every subgroup $N\leq J_{v_{0}}$ corresponding
to $\Lambda$ is completely determined by a set of vertices $\widehat{v}_{1},\ldots,\widehat{v}_{r}$
in $T$: the topological edge in $\Gamma_{X}\left(N\right)$ associated
with $e_{i}$ corresponds to the walk in $T$ from $\widehat{v}_{\mathrm{beg\left(i\right)}}$
to $\widehat{v}_{i}$. (There are some constraints on the choices
of the $\widehat{v}_{i}$'s. For example, if $v_{i}=v_{j}$ then $\widehat{v}_{i}$
and $\widehat{v}_{j}$ must belong to the same fiber of the projection
map $p:T\to\Omega$. However, as we only bound from above, we ignore
these constraints.) So instead of summing over all possible $N$'s,
we go through all possible choices of vertices $\widehat{v}_{1},\ldots,\widehat{v}_{r}$
in $T$. 

The counting argument that follows resembles the one in Proposition
\ref{prop:reduced-words-upper-bound}. Fix a particular $N\leq J_{v_{0}}$
corresponding to $\Lambda$ and let $\widehat{v}_{1},\ldots,\widehat{v}_{r}$
be the corresponding vertices in $T$. By Lemma \ref{lem:each-edge-twice},
if $w\in\left(X\cup X^{-1}\right)^{t}$ satisfies $\left\langle w\right\rangle \neqalg N$,
then its reduced form traverses every topological edge of $\Gamma_{X}\left(N\right)$
at least twice. For each $i$, assume that $w$ first traverses the
topological edge associated with $e_{i}$ starting at position $\tau_{i,1}$
(the position is in $w$, namely $0\le\tau_{i,1}\le t-1$), and in
$\ell_{i,1}$ steps, and then from position $\tau_{i,2}$ in $\ell_{i,2}$
steps (recall that $w$ is not reduced so $\ell_{i,2}$ may be different
from $\ell_{i,1}$). The directions of these traverses are $\varepsilon_{i,1},\varepsilon_{i,2}\in\left\{ \pm1\right\} $.
In total, there are less than $t^{2r}$ options for the $\tau_{i,j}$'s,
less than $t^{2r}$ options for the $\ell_{i,j}$'s and less than
$2^{2r}$ options for the $\varepsilon_{i,j}$'s: a total of less
than $\left(4t^{4}\right)^{r}$ options. There are $t-\ell_{1,1}-\ell_{1,2}-\ldots-\ell_{r,1}-\ell_{r,2}$
remaining steps, and these are divided to at most $4r$ segments (we
can always assume one of the $\tau_{i,1}$'s equals 0). Denote the
lengths of these segments by $q_{1},\ldots,q_{4r}$ (some may be 0).
The $i$'th segment reduces to some walk in $\Gamma_{X}\left(N\right)$,
with at most $\left(\delta-1\right)^{q_{i}}$ possibilities (recall
that $\delta$ marks the maximal degree of a vertex in $\Lambda$).
Overall, there are at most $\left(\delta-1\right)^{q_{1}+\ldots+q_{4r}}\leq\left(\delta-1\right)^{t}$
options to choose the reduced walks traced by these $4r$ segments
in $w$. Given such a reduced walk for the $i$'th segment, let $\widehat{x}_{i},\widehat{y}_{i}\in V\left(T\right)$
be suitable vertices in the tree such that the reduced walk lifts
to the unique reduced walk from $\widehat{x}_{i}$ to $\widehat{y}_{i}$.

Now, we sum over all subgroups $N$ corresponding to $\Lambda$ and
all words $w\in\cpt\left(\Omega\right)$ with $\left\langle w\right\rangle \neqalg N$.
By adding a factor of $\left|V\left(\Omega\right)\right|\left(4t^{4}\right)^{r}\cdot\left(\delta-1\right)^{t}$
we assume we already know $v_{0}$ and $\widehat{v}_{0}$, the $\tau_{i,j}$'s,
$\ell_{i,j}$'s, $\varepsilon_{i,j}$'s, the $q_{i}$'s and the reduced
$4r$ walks. Moreover, conditioning on knowing $\widehat{v}_{1},\ldots,\widehat{v}_{r}$,
we also know the $\widehat{x}_{i}$'s and the $\widehat{y}_{i}$'s.
Recall that $c_{\Gamma}\left(t,u,v\right)$ denotes the number of
walks of length $t$ in a graph $\Gamma$ from the vertex $u$ to
the vertex $v$, and that by \eqref{c_Gamma}, $c_{T}\left(t,u,v\right)\leq\rho^{t}$
for every $u,v\in V\left(T\right)$. For each $i=1,\ldots,r$ and
$j=1,2$, there are $c_{T}\left(\ell_{i,j},\widehat{v}_{\mathrm{beg}\left(i\right)},\widehat{v}_{i}\right)$
possible subwords corresponding to the $j$'th traverse of $e_{i}$
(even if $\varepsilon_{i,j}=-1$, because $c_{T}\left(\ell_{i,j},\widehat{v}_{\mathrm{beg}\left(i\right)},\widehat{v}_{i}\right)=c_{T}\left(\ell_{i,j},\widehat{v}_{i},\widehat{v}_{\mathrm{beg}\left(i\right)}\right)$).
Similarly, there are at most $c_{T}\left(q_{i},\widehat{x}_{i},\widehat{y}_{i}\right)$
subwords corresponding to the the $ $$i$'th intermediate segment.
Thus, if $\alpha=\left|V\left(\Omega\right)\right|\cdot\left(4t^{4}\right)^{r}\cdot\left(\delta-1\right)^{t}$
then 
\begin{eqnarray*}
\sum_{\substack{N\leq\F_{k}:\\
\Gamma_{X}\left(N\right)\cong\Lambda
}
}\beta_{t}\left(N\right) & \leq & \alpha\cdot\sum_{\widehat{v}_{1},\ldots,\widehat{v}_{r}\in V\left(T\right)}\left[\prod_{i=1}^{r}\prod_{j=1}^{2}c_{T}\left(\ell_{i,j},\widehat{v}_{\mathrm{beg}\left(i\right)},\widehat{v}_{i}\right)\right]\prod_{i=1}^{4r}c_{T}\left(q_{i},\widehat{x}_{i},\widehat{y}_{i}\right)\\
 & \leq & \alpha\cdot\left[\prod_{i=1}^{4r}\rho^{q_{i}}\right]\sum_{\widehat{v}_{1},\ldots,\widehat{v}_{r}\in V\left(T\right)}\left[\prod_{i=1}^{r}\prod_{j=1}^{2}c_{T}\left(\ell_{i,j},\widehat{v}_{\mathrm{beg}\left(i\right)},\widehat{v}_{i}\right)\right]
\end{eqnarray*}
Note that $\mathrm{beg\left(i\right)<i},$ so $c_{T}\left(\ell_{i,j},\widehat{v}_{\mathrm{beg}\left(i\right)},\widehat{v}_{i}\right)$
only depends on $\ell_{i,j}$ and $\widehat{v}_{0},\ldots,\widehat{v}_{i}$
(and not on $\widehat{v}_{i+1},\ldots,\widehat{v}_{r}$). Therefore,
if we write $f\left(i\right)=\prod_{j=1}^{2}c_{T}\left(\ell_{i,j},\widehat{v}_{\mathrm{beg}\left(i\right)},\widehat{v}_{i}\right)$,
we can split the sum to obtain:

\begin{eqnarray*}
\sum_{\substack{N\leq\F_{k}:\\
\Gamma_{X}\left(N\right)\cong\Lambda
}
}\beta_{t}\left(N\right) & \leq & \alpha\cdot\rho^{\sum q_{i}}\sum_{\widehat{v}_{1}\in V\left(T\right)}f\left(1\right)\left[\sum_{\widehat{v}_{2}\in V\left(T\right)}f\left(2\right)\left[\ldots\right]\right]
\end{eqnarray*}

The following step is the crux of the matter. We use the fact that
each topological edge is traversed twice to get rid of the summation
over vertices in $T$. We begin with the last edge $e_{r}$, where
we replace the expression $\sum_{\widehat{v}_{r}\in V\left(T\right)}f\left(r\right)$
as follows:
\begin{eqnarray*}
\sum_{\widehat{v}_{r}\in V\left(T\right)}f\left(r\right) & = & \sum_{\widehat{v}_{r}\in V\left(T\right)}c_{T}\left(\ell_{r,1},\widehat{v}_{\mathrm{beg}\left(r\right)},\widehat{v}_{r}\right)c_{T}\left(\ell_{r,2},\widehat{v}_{\mathrm{beg}\left(r\right)},\widehat{v}_{r}\right)\\
 & = & \sum_{\widehat{v}_{r}\in V\left(T\right)}c_{T}\left(\ell_{r,1},\widehat{v}_{\mathrm{beg}\left(r\right)},\widehat{v}_{r}\right)c_{T}\left(\ell_{r,2},\widehat{v}_{r},\widehat{v}_{\mathrm{beg}\left(r\right)}\right)\\
 & \overset{\left(*\right)}{=} & c_{T}\left(\ell_{r,1}+\ell_{r,2},\widehat{v}_{\mathrm{beg}\left(r\right)},\widehat{v}_{\mathrm{beg}\left(r\right)}\right)\leq\rho^{\ell_{r,1}+\ell_{r,2}}.
\end{eqnarray*}
The crucial step here is the equality $\overset{\left(*\right)}{=}$.
It follows from the fact that $\widehat{v}_{r}$ can be recovered
as the vertex of $T$ visited by the walk of length $\ell_{r,1}+\ell_{r,2}$
after $\ell_{r,1}$ steps. After {}``peeling'' the expression $\sum_{\widehat{v}_{r}\in V\left(T\right)}f\left(r\right)$,
we can go on and bound $\sum_{\widehat{v}_{r-1}\in V\left(T\right)}f\left(r-1\right)$
by $\rho^{\ell_{r-1,1}+\ell_{r-1,2}}$ and so on. Eventually, we obtain
\begin{eqnarray*}
\sum_{\substack{N\leq\F_{k}:\\
\Gamma_{X}\left(N\right)\cong\Lambda
}
}\beta_{t}\left(N\right) & \leq & \alpha\cdot\rho^{\sum q_{i}}\prod_{i=1}^{r}\rho^{\ell_{i,1}+\ell_{i,2}}=\left|V\left(\Omega\right)\right|\cdot\left(4t^{4}\right)^{r}\cdot\left(\delta-1\right)^{t}\cdot\rho^{t}.
\end{eqnarray*}

\end{proof}
Finally, we are in position to establish the upper bounds stated in
Theorem \thmref{bound-for-m-in-general-Omega}. Fix $m\in\left\{ 1,2,\ldots,\rk\left(\Omega\right)\right\} $.
Then by \eqref{sum-over-betta-N} and Claim \claimref{ CR-subgroups},
\begin{eqnarray}
\sum_{w\in\cptm\left(\Omega\right)}\left|\crit\left(w\right)\right| & \leq & \sum_{\substack{N\leq\F_{k}:\\
\rk\left(N\right)=m
}
}\beta_{t}\left(N\right)\nonumber \\
 & \leq & \sum_{\substack{\left[N\right]\in\mathrm{ConjCls\left(\F_{k},m\right)}\\
N\,\mathrm{is\, CR}
}
}t\beta_{t}\left(N\right)\label{eq:sum_over_conj_classes}
\end{eqnarray}
where the final summation is over all conjugacy classes of subgroups
of rank $m$ in $\F_{k}$, and for each class $N$ is a CR representative.
Moreover, we choose these representatives $N$ so that if $\left[N_{1}\right]$
and $\left[N_{2}\right]$ correspond the same non-pointed topological
graph, the representatives $N_{1}$ and $N_{2}$ correspond to the
same \emph{pointed} topological graph $\Lambda$.

Finally, split the summation of the CR representatives $N$ by their
topological graph $\Lambda$. By Claim \Claimref{core-graphs-of-rank-m},
each such $\Lambda$ has maximal degree at most $2m$ and at most
$3m-1$ edges, so by Proposition \propref{key-prop-for-top-graph},
the $N$'s corresponding to each $\Lambda$ contribute to the summation
in \eqref{sum_over_conj_classes} at most 
\[
t\cdot\left|V\left(\Omega\right)\right|\cdot\left(4t^{4}\right)^{3m-1}\cdot\left(2m-1\right)^{t}\cdot\rho^{t}.
\]
This finishes the proof of Theorem \thmref{bound-for-m-in-general-Omega}
as there is a finite number of topological graphs $\Lambda$ of rank
$m$. $\qed$

\section{Controlling the Error Term of $\mathbb{E}\left[{\cal F}_{w,n}\right]$\label{sec:Controlling-the-O}}

In this section we establish the third step of the proofs of Theorems
\ref{thm:2sqrt(d-1)+1}, \ref{thm:sqrt3-times-rho}, and \ref{thm:base-d-regular},
as introduced in the overview of the proof (Section \ref{sec:Overview-of-the-proof}).
Recall that according to Theorem \ref{thm:avg_fixed_pts}, for every
$w\in\F_{k}$ the following holds:
\[
\mathbb{E}\left[{\cal F}_{w,n}\right]=1+\frac{|\crit\left(w\right)|}{n^{\pi\left(w\right)-1}}+O\left(\frac{1}{n^{\pi\left(w\right)}}\right).
\]
But the $O\left(\cdot\right)$ term depends on $w$. Our goal here
is to obtain a bound on the $O\left(\cdot\right)$ term, which depends
solely on the length of $w$ and $\pi\left(w\right)$, namely a bound
which is uniform on all words of a certain length and primitivity
rank. This is done in the following proposition:
\begin{prop}
\label{prop:controling-the-O}Let $w\in\left(X\cup X^{-1}\right)^{t}$
satisfy $\pi\left(w\right)\ne0$ (so $w$ does not reduce to $1$).
If $n>t^{2}$ then 
\[
\mathbb{E}\left[{\cal F}_{w,n}\right]\leq1+\frac{1}{n^{\pi\left(w\right)-1}}\left(\left|\crit\left(w\right)\right|+\frac{t^{2+2\pi\left(w\right)}}{n-t^{2}}\right).
\]

\end{prop}
Achieving such a bound requires more elaborated details from the proof
of Theorem \ref{thm:avg_fixed_pts}, which appears in \cite{PP15}.
We therefore begin with recalling relevant concepts and results from
\cite{PP15}. We then present the proof of Proposition \propref{controling-the-O}
in Section \subref{A-uniform-bound}. 

Before that, let us mention that the same statement holds for words
in $\left(X\cup X^{-1}\right)^{t}$ that reduce to 1:
\begin{claim}
\label{claim:controlling-the-O-for-m=00003D0}Let $w\in\left(X\cup X^{-1}\right)^{t}$
satisfy $\pi\left(w\right)=0$ (so $w$ reduces to $1$). If $n>t^{2}$
then 
\[
\mathbb{E}\left[{\cal F}_{w,n}\right]\leq1+\frac{1}{n^{\pi\left(w\right)-1}}\left(\left|\crit\left(w\right)\right|+\frac{t^{2+2\pi\left(w\right)}}{n-t^{2}}\right).
\]
\end{claim}
\begin{proof}
Recall that $\pi\left(w\right)=0$ if and only if $w=1$ as an element
of $\F_{k}$. But then the only $w$-critical subgroup is the trivial
one, and so $\mathbb{E}\left[{\cal F}_{w,n}\right]=n=1+\frac{1}{n^{-1}}\left(\left|\crit\left(w\right)\right|-\frac{1}{n}\right)$
which is indeed less than the bound in the statement.
\end{proof}

\subsection{The partial order {}``covers''\label{sub:The-Partial-Order}}

In Section \ref{sub:Core-Graphs} morphisms of core graphs were discussed.
Recall that a morphism $\Gamma_{X}\left(H\right)\to\Gamma_{X}\left(J\right)$
exists (and is unique) if and only if $H\le J$ (Claim \ref{cla:morphism-properties}).
A special role is played by \emph{surjective} morphisms of core graphs:
\begin{defn}
\label{def:quotient}Let $H\le J\le\F_{k}$. Whenever the morphism
$\eta_{H\to J}^{X}:\Gamma_{X}\left(H\right)\to\Gamma_{X}\left(J\right)$
is surjective, we say that \emph{$\G_{X}\left(H\right)$ covers $\G_{X}\left(J\right)$}
or that \emph{$\G_{X}\left(J\right)$ is a quotient of $\G_{X}\left(H\right)$}.
As for the groups, we say that \emph{$H$ $X$-covers $J$} and denote
this by $H\covers J$\emph{}\marginpar{$H\covers J$}\emph{.}
\end{defn}
By {}``surjective'' we mean surjective on both vertices and edges.
Note that we use the term {}``covers'' even though in general this
is \emph{not} a topological covering map (a morphism between core
graphs is always locally injective at the vertices, but it need not
be locally bijective). In contrast, the random graphs in ${\cal C}_{n,H}$
are topological covering maps, and we reserve the term {}``coverings''
for these.

For instance, $H=\langle x_{1}x_{2}x_{1}^{-3},x_{1}^{\;2}x_{2}x_{1}^{-2}\rangle\le\F_{k}$
$X$-covers the group $J=\langle x_{2},x_{1}^{\;2},x_{1}x_{2}x_{1}\rangle$,
the corresponding core graphs of which are the leftmost and rightmost
graphs in Figure \figref{quotient-graph}. As another example, a core
graph $\G$ $X$-covers $\G_{X}\left(\F_{k}\right)$ (which is merely
a wedge of $k$ loops) if and only if it contains edges of all $k$
labels.

As implied by the notation, the relation $H\covers J$ indeed depends
on the given basis\emph{ $X$ }of\emph{ $\F_{k}$}. For example, if
$H=\langle x_{1}x_{2}\rangle$ then $H\covers\F_{2}$. However, for
$Y=\left\{ x_{1}x_{2},x_{2}\right\} $, $H$ does not\emph{ }$Y$-cover
$\F_{2}$, as $\G_{Y}\left(H\right)$ consists of a single vertex
and a single loop and has no quotients apart from itself.

It is easy to see that the relation {}``$\covers$'' indeed constitutes
a partial ordering of the set of subgroups of $\F_{k}$. In fact,
restricted to f.g.~subgroups it becomes a locally-finite partial
order, which means that if $H\covers J$ then the interval of intermediate
subgroups $\XC{H}{J}=\left\{ M\le\F_{k}\,\middle|\, H\covers M\covers J\right\} $
is finite:
\begin{claim}
\label{cla:O_X(H)-is-finite}If $H\le\F_{k}$ is a f.g.~subgroup
then it $X$-covers only a finite number of groups. In particular,
the partial order {}``$\covers$'' restricted to f.g.\ subgroups
of $\F_{k}$ is locally finite.\end{claim}
\begin{proof}
The claim follows from the fact that $\G_{X}\left(H\right)$ is finite
(Claim \claref{core-graphs-properties}\enuref{finite-rk-graph})
and thus has only finitely many quotients. Each quotient corresponds
to a single group, by \eqref{pi_1_gamma}.
\end{proof}

\subsection{Partitions and quotients\label{sub:Partitions-and-Quotients}}

It is easy to see that a quotient $\Gamma_{X}\left(J\right)$ of $\Gamma_{X}\left(H\right)$
is determined by the partition it induces on the vertex set $V\left(\G_{X}\left(H\right)\right)$
(the vertex-fibers of the morphism $\eta_{H\to J}^{X}$). However,
not every partition $P$ of $V\left(\G_{X}\left(H\right)\right)$
corresponds to a quotient core-graph. Indeed, $\Delta$, the graph
we obtain after merging the vertices grouped together in $P$, might
not be a core-graph: two distinct $j$-edges may have the same origin
or the same terminus. (For a combinatorial description of core-graphs
see e.g.~\cite[Claim 2.1]{Pud14a}.) Then again, when a partition
$P$ of $V\left(\G_{X}\left(H\right)\right)$ yields a quotient which
is not a core-graph, we can perform Stallings foldings%
\footnote{A folding means merging two equally-labeled edges with the same origin
or with the same terminus. See also Figure \figref{quotient-graph}.
For a fuller description of Stallings foldings we refer the reader
to \cite{Pud14a,PP15}. %
} until we obtain a core graph. We denote the resulting core-graph
by%
\footnote{In \cite{PP15}, the notation $\nicefrac{\Gamma_{X}\left(H\right)}{P}$
was used to denote something a bit different (the unfolded graph $\Delta$).%
} $\nicefrac{\Gamma_{X}\left(H\right)}{P}$\marginpar{$\nicefrac{\Gamma_{X}\left(H\right)}{P}$}.
Since Stallings foldings do not affect $\pi_{1}^{X}$, this core graph
$\nicefrac{\Gamma_{X}\left(H\right)}{P}$ is $\Gamma_{X}\left(J\right)$,
where $J=\pi_{1}^{X}\left(\Delta\right)$. The resulting partition
$\bar{P}$ of $V\left(\G_{X}\left(H\right)\right)$ (the blocks of
which are the fibers of $\eta_{H\rightarrow J}^{X}$) is the finest
partition of $V\left(\G_{X}\left(H\right)\right)$ which gives a quotient
core-graph and which is still coarser than $P$. We illustrate this
in Figure \figref{quotient-graph}.

\begin{figure}[h]
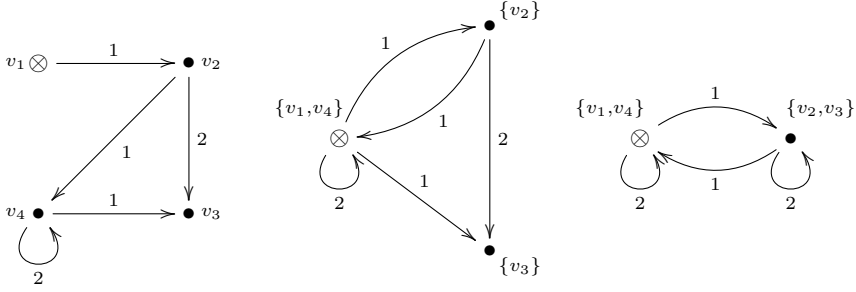

\noindent \begin{centering}
\begin{minipage}[t]{0.9\columnwidth}%
\noindent \begin{center}
\begin{center}
\xy 
(0,35)*+{\otimes}="m0"+(-3,0)*{\scriptstyle v_1};%
(20,35)*+{\bullet}="m1"+(3,0)*{\scriptstyle v_2};%
(20,15)*+{\bullet}="m2"+(3,0)*{\scriptstyle v_3};%
(0,15)*+{\bullet}="m3"+(-3,0)*{\scriptstyle v_4};%
{\ar^{1} "m0";"m1"};%
{\ar^{2} "m1";"m2"};%
{\ar^{1} "m3";"m2"};%
{\ar^{1} "m1";"m3"};%
{\ar@(dl,dr)_{2} "m3";"m3"};%
(40,25)*+{\otimes}="t0"+(-4,4)*{\scriptstyle \{v_1,v_4\}};%
(60,40)*+{\bullet}="t1"+(4,2)*{\scriptstyle \{v_2\}};%
(60,10)*+{\bullet}="t2"+(4,-2)*{\scriptstyle \{v_3\}};%
{\ar@/^1pc/^{1} "t0";"t1"};%
{\ar^{2} "t1";"t2"};%
{\ar^{1} "t0";"t2"};%
{\ar@/^1pc/^{1} "t1";"t0"};%
{\ar@(dl,dr)_{2} "t0";"t0"};%
(80,25)*+{\otimes}="t0"+(-4,4)*{\scriptstyle \{v_1,v_4\}};%
(100,25)*+{\bullet}="t1"+(4,4)*{\scriptstyle \{v_2,v_3\}};%
{\ar@/^1pc/^{1} "t0";"t1"};%
{\ar@/^1pc/^{1} "t1";"t0"};%
{\ar@(dl,dr)_{2} "t0";"t0"};%
{\ar@(dl,dr)_{2} "t1";"t1"};%
\endxy 
\par\end{center}
\par\end{center}%
\end{minipage}
\par\end{centering}

\caption{\label{fig:quotient-graph} The left graph is the core graph $\G_{X}\left(H\right)$
of $H=\left\langle x_{1}x_{2}x_{1}^{-3},x_{1}^{\;2}x_{2}x_{1}^{-2}\right\rangle \leq\F_{2}$.
Its vertices are denoted by $v_{1},\ldots,v_{4}$. The graph in the
middle is the quotient corresponding to the partition $P=\left\{ \left\{ v_{1},v_{4}\right\} ,\left\{ v_{2}\right\} ,\left\{ v_{3}\right\} \right\} $.
This is not a core graph as there are two $1$-edges originating at
$\left\{ v_{1},v_{4}\right\} $. In order to obtain a core quotient-graph,
we use the Stallings folding process and identify these two $1$-edges
and their termini. The resulting core graph, $\nicefrac{\Gamma_{X}\left(H\right)}{P}$,
is shown on the right and corresponds to the partition $\bar{P}=\left\{ \left\{ v_{1},v_{4}\right\} ,\left\{ v_{2},v_{3}\right\} \right\} $.}
\end{figure}

One can think of $\Gamma_{X}\left(J\right)=\nicefrac{\Gamma_{X}\left(H\right)}{P}$
as the core graph {}``generated'' from $\Gamma_{X}\left(H\right)$
by the partition $P$. It is now natural to look for the {}``simplest{}``
partition generating $\Gamma_{X}\left(J\right)$. Formally, we introduce
a measure for the complexity of partitions: if $P\subseteq2^{\mathcal{X}}$
is a partition of some set $\mathcal{X}$, let 
\begin{equation}
\left\Vert P\right\Vert \overset{{\scriptscriptstyle def}}{=}\left|\mathcal{X}\right|-\left|P\right|=\sum_{B\in P}\left(\left|B\right|-1\right).\label{eq:Partition_norm}
\end{equation}
Namely, $\left\Vert P\right\Vert $ is the number of elements in the
set minus the number of blocks in the partition. For example, $\left\Vert P\right\Vert =1$
iff $P$ identifies only a single pair of elements. It is not hard
to see that $\left\Vert P\right\Vert $ is also the minimal number
of identifications one needs to make in $\mathcal{X}$ in order to
obtain the equivalence relation $P$. Restricting to pairs of subgroups
$H,J$ with $H\covers J$, we can define the following distance function:
\begin{defn}
\label{def:distance} Let $H,J\fg\F_{k}$ be subgroups such that $H\covers J$,
and let $\G=\G_{X}\left(H\right)$, $\Delta=\G_{X}\left(J\right)$
be the corresponding core graphs. We define the \emph{$X$-distance}
between $H$ and $J$, denoted $\rho_{X}\left(H,J\right)$\marginpar{$\rho_{X}\left(H,J\right)$}
or $\rho\left(\G,\Delta\right)$ as
\begin{equation}
\rho_{X}\left(H,J\right)=\min\left\{ \left\Vert P\right\Vert \,\middle|\,{P\,\mathrm{is\, a\, partition\, of}\, V\left(\Gamma_{X}\left(H\right)\right)\atop \mathrm{s.t.\,}\,\nicefrac{\Gamma_{X}\left(H\right)}{P}=\Gamma_{X}\left(J\right)}\right\} .\label{eq:rho_from_partition}
\end{equation}

\end{defn}
For example, the rightmost core graph in Figure \figref{quotient-graph}
is a quotient of the leftmost one, and the distance between them is
$1$. For a more geometric description of this distance function,
as well as more details and further examples, we refer the readers
to \cite{Pud14a,PP15}. 

Of course, the distance function $\rho_{X}\left(H,J\right)$ is computable.
It turns out that it can also be used to determine whether $H$ is
a free factor of $J$:
\begin{thm}
\label{thm:distance}[\cite{Pud14a},Theorem 1.1 and Lemma 3.3] Let
$H,J\fg\F_{k}$ such that $H\covers J$. Then 
\[
rk\left(J\right)-rk\left(H\right)~~\le~~\rho_{X}\left(H,J\right)~~\le~~rk\left(J\right).
\]
Most importantly, the minimum is obtained (namely, $\mathrm{rk}\left(J\right)-\mathrm{rk}\left(H\right)=\rho_{X}\left(H,J\right)$)
if and only if $H$ is a free factor of $J$.
\end{thm}
\noindent This theorem is used, in particular, in the proof in \cite{PP15}
of Theorem \thmref{avg_fixed_pts}.

So far the partitions considered here were partitions of the vertex
set\linebreak{}
 $V\left(\Gamma_{X}\left(H\right)\right)$. However, it is also possible
to identify (merge) different \emph{edges} in $\Gamma_{X}\left(H\right)$,
as long as they share the same label, and then, as before, perform
the folding process to obtain a valid core graph. Moreover, it is
possible to consider several partitions $P_{1},\ldots,P_{r}$, each
one \emph{either} of the vertices \emph{or} of the edges of $\Gamma_{X}\left(H\right)$,
identify vertices and edges according to these partitions and then
fold. We denote the resulting core graph by \marginpar{$\nicefrac{\Gamma_{X}\left(H\right)}{\left\langle P_{1},\ldots,P_{r}\right\rangle }$}$\nicefrac{\Gamma_{X}\left(H\right)}{\left\langle P_{1},\ldots,P_{r}\right\rangle }$.
It is easy to see that one can incorporate this more involved definition
into the definition of the distance function $\rho_{X}\left(H,J\right)$,
because, for instance, identifying two edges has the same effect as
identifying their origins (or termini). In fact, the following holds:
\begin{equation}
\rho_{X}\left(H,J\right)=\min\left\{ \left\Vert P_{1}\right\Vert +\ldots+\left\Vert P_{r}\right\Vert \,\middle|\,{P_{i}:\,\,\mathrm{a\, partition\, of}\, V\left(\Gamma_{X}\left(H\right)\right)\,\mathrm{or\, of}\, E\left(\Gamma_{X}\left(H\right)\right)\atop \mathrm{s.t.\,}\,\nicefrac{\Gamma_{X}\left(H\right)}{\left\langle P_{1},\ldots,P_{r}\right\rangle }=\Gamma_{X}\left(J\right)}\right\} .\label{eq:long-distance-definition}
\end{equation}

\subsection{From random elements of $S_{n}$ to random subgroups\label{sub:Phi}}

Recall that Theorem \ref{thm:avg_fixed_pts} estimates $\mathbb{E}\left[{\cal F}_{w,n}\right]$,
the expected number of fixed points of $w\left(\sigma_{1},\ldots,\sigma_{k}\right)$,
where $\sigma_{1},\ldots,\sigma_{k}\in S_{n}$ are chosen independently
at random in uniform distribution. The first step in its proof consists
of a \emph{generalization of the problem to} \emph{subgroups:}

For every f.g.~subgroups $H\leq J\leq\F_{k}$, let $\alpha_{J,S_{n}}:J\to S_{n}$
be a random homomorphism chosen at uniform distribution (there are
exactly $\left|S_{n}\right|^{\mathrm{rk}\left(J\right)}$ such homomorphisms).
Then $\alpha_{J,S_{n}}\left(H\right)$ is a random subgroup of $S_{n}$,
and we count the number of common fixed points of this subgroup, namely
the number of elements in $\left\{ 1,\ldots,n\right\} $ fixed by
all permutations in $\alpha_{J,S_{n}}\left(H\right)$. Formally, we
define\marginpar{$\Phi_{H,J}$}
\[
\Phi_{H,J}\left(n\right)\stackrel{\mathrm{def}}{=}\mathbb{E}\left|_{\mathrm{fixed-points}}^{\mathrm{common}}\left(\alpha_{J,S_{n}}\left(H\right)\right)\right|.
\]
This indeed generalizes $\mathbb{E}\left[{\cal F}_{w,n}\right]$ for
\begin{equation}
\mathbb{E}\left[{\cal F}_{w,n}\right]=\Phi_{\left\langle w\right\rangle ,\F_{k}}\left(n\right).\label{eq:E[f_w]=00003DPhi}
\end{equation}

\subsection{Möbius inversions\label{sub:M=0000F6bius-Inversions}}

The theory of Möbius inversions applies to every poset (partially
ordered set) with a \emph{locally-finite} order (recall that an order
$\preceq$ is locally-finite if for every $x,y$ with $x\preceq y$,
the interval $\left[x,y\right]_{\preceq}\stackrel{def}{=}\left\{ z\,\middle|\, x\preceq z\preceq y\right\} $
is finite). Here we skip the general definition and define these inversions
directly in the special case of interest (for a more general point
of view see \cite{PP15}).

\begin{wrapfigure}{R}{0.4\columnwidth}%
\[
\xymatrix{ & \Phi\ar@{-}[dl]\ar@{-}[dr]\\
L\ar@{-}[dr] &  & R\ar@{-}[dl]\\
 & C
}
\]
\end{wrapfigure}%
 In our case, the poset in consideration is $\mathfrak{sub}_{f\! g}\left(\mathbf{F}_{k}\right)=\left\{ H\leq\F_{k}\,\middle|\, H\,\,\mathrm{is\,\, f.g.}\right\} $,
and the partial order is $\covers$, which is indeed locally-finite
(Claim \claref{O_X(H)-is-finite}). We define three derivations of
the function $\Phi$ defined in Section \subref{Phi}: the left one
($L$), the right one ($R$) and the two-sided one ($C$). These are
usually formally defined by convolution of $\Phi$ with the\emph{
}Möbius function of $\mathfrak{sub}_{f\! g}\left(\mathbf{F}_{k}\right)_{\covers}$
(see \cite{PP15}) but here we define them in an equivalent simpler
way: these are the functions satisfying, for every $H\covers J$,
\begin{equation}
\Phi_{H,J}\left(n\right)=\sum_{M\in\XC{H}{J}}L_{M,J}\left(n\right)=\negthickspace\sum_{M,N:\, H\covers M\covers N\covers J}\negthickspace C_{M,N}\left(n\right)=\negthickspace\sum_{N\in\XC{H}{J}}\negthickspace R_{H,N}\left(n\right).\label{eq:mobius-inversions}
\end{equation}
Note that the summations in \eqref{mobius-inversions} are well defined
because the order is locally finite. To see that \eqref{mobius-inversions}
can indeed serve as the definition for the three new functions, use
induction on $\left|\left[H,J\right]\right|$: for example, for any
$H\covers J$, $L_{H,J}\left(n\right)=\Phi_{H,J}\left(n\right)-\sum_{M\in\XCO{H}{J}}L_{M,J}\left(n\right)$
and all pairs $\left(M,J\right)$ on the r.h.s.~satisfy $\left|\left[M,J\right]\right|<\left|\left[H,J\right]\right|$.

With all this defined, we can state the main propositions along the
proof of the main result in \cite{PP15}.
\begin{prop}
[\cite{PP15}, Proposition 5.1]\label{prop:R_supported-on-algebraic}The
function $R$ is supported on algebraic extensions. 
\end{prop}
Namely, if $J$ is not an algebraic extension of $H$, then $R_{H,J}\left(n\right)=0$
for every $n$. Since, if $H\leq_{\mathrm{alg}}J$ then $H\covers J$
(e.g. \cite[Claim 4.2]{PP15}), we obtain that 
\begin{equation}
\Phi_{H,J}\left(n\right)=\sum_{N:\, H\alg N\leq J}R_{H,N}\left(n\right).\label{eq:phi=00003Dsum-of-alg}
\end{equation}

Next, $\Phi_{H,J}\left(n\right)$ is given a geometric interpretation:
it turns out it equals the expected number of lifts of $\eta_{H\to J}:\Gamma_{X}\left(H\right)\to\Gamma_{X}\left(J\right)$
to a random $n$-covering of $\Gamma_{X}\left(J\right)$ in the model
${\cal C}_{n,\Gamma_{X}\left(J\right)}$ \cite[Lemma 6.2]{PP15}.
Similarly, $L_{H,J}\left(n\right)$ counts the average number of \emph{injective}
lifts \cite[Lemma 6.3]{PP15}. For given $H$ and $J$, it is not
hard to come up with an exact rational expression in $n$ for the
expected number of injective lifts, i.e. of $L_{H,J}\left(n\right)$,
for large enough $n$ (in fact, $n\geq\left|E\left(\Gamma_{X}\left(H\right)\right)\right|$
suffices, see \cite[Lemma 6.4]{PP15}) . As the other three functions
($\Phi$, $R$ and $C$) are obtained via addition and subtraction
of a finite number of $L_{M,J}\left(n\right)$'s, we obtain
\begin{claim}
\label{claim:rational-expressions}Let $H,J\leq\F_{k}$ be f.g.~subgroups
such that $H\covers J$. Then for $n\geq\left|E\left(\Gamma_{X}\left(H\right)\right)\right|$,
the functions $\Phi_{H,J}\left(n\right)$, $L_{H,J}\left(n\right)$,
$R_{H,J}\left(n\right)$ and $C_{H,J}\left(n\right)$ can all be expressed
as rational expressions in $n$.
\end{claim}
After some involved combinatorial arguments, one obtains from this
the following expression for $C_{M,N}\left(n\right)$: Denote by $\mathrm{Sym}\left(S\right)$
the set of permutations of a given set $S$. Every permutation $\sigma\in\mathrm{Sym}\left(S\right)$
defines, in particular, a partition on $S$ whose blocks are the cycles
of $\sigma$. By abuse of notation we denote by $\sigma$ both the
permutation and the corresponding partition. For instance, one can
consider its {}``norm'' $\left\Vert \sigma\right\Vert $ (see \eqref{Partition_norm};
this is also the minimal length of a product of transpositions that
gives the permutation $\sigma$). We also use $V_{M}$ and\marginpar{$V_{M},\, E_{M}$}
$E_{M}$ as short for $V\left(\Gamma_{X}\left(M\right)\right)$ and
$E\left(\Gamma_{X}\left(M\right)\right)$, respectively.
\begin{prop}
[\cite{PP15}, Section 7.1] \label{prop:C_MN}Let $M,N\leq\F_{k}$
be f.g.~subgroups with $M\covers N$. Consider the set
\begin{align*}
\mathcal{T}_{M,N} & =\left\{ \left(\sigma_{0},\sigma_{1},\ldots,\sigma_{r}\right)\,\middle|\,\begin{matrix}r\in\mathbb{N},\:\sigma_{0}\in\mathrm{\Sym}\left(V_{M}\right)\\
\sigma_{1},\ldots,\sigma_{r}\in\Sym\left(E_{M}\right)\backslash\left\{ \mathrm{id}\right\} \vphantom{\Big|}\\
\nicefrac{\Gamma_{X}\left(M\right)}{\left\langle \sigma_{0},\sigma_{1},\ldots,\sigma_{r}\right\rangle }=\Gamma_{X}\left(N\right)
\end{matrix}\right\} .
\end{align*}
Then 
\[
C_{M,N}\left(n\right)=\frac{1}{n^{\rk\left(M\right)-1}}\sum_{\left(\sigma_{0},\sigma_{1},\ldots,\sigma_{r}\right)\in\mathcal{T}_{M,N}}\left(-1\right)^{r}\cdot\left(\frac{-1}{n}\right)^{\sum\limits _{i=0}^{r}\left\Vert \sigma_{i}\right\Vert }.
\]

\end{prop}
The derivation of the main result of \cite{PP15} (Theorem \ref{thm:avg_fixed_pts})
from Theorem \ref{thm:distance} and Propositions \ref{prop:R_supported-on-algebraic}
and \ref{prop:C_MN} is short: see the beginning of Section 7 in \cite{PP15}.

\subsection{Proving the uniform bound for the error term\label{sub:A-uniform-bound}}

We now have all the tools required for proving Proposition \propref{controling-the-O}.
Namely, we now prove that every $1\ne w\in\F_{k}$ of length $t$
and every $n>t^{2}$, 
\[
\mathbb{E}\left[{\cal F}_{w,n}\right]\leq1+\frac{1}{n^{\pi\left(w\right)-1}}\left(\left|\crit\left(w\right)\right|+\frac{t^{2+2\pi\left(w\right)}}{n-t^{2}}\right).
\]
(Note that we pass here to reduced words. Reducing an element of $\left(X\cup X^{-1}\right)^{t}$
does not affect $\mathbb{E}\left[{\cal F}_{w,n}\right]$, and only
tightens the upper bound.)
\begin{proof}
{[}of Proposition \propref{controling-the-O}{]} Recall (Section \ref{sub:Phi})
that $\mathbb{E}\left[{\cal F}_{w,n}\right]=\Phi_{\left\langle w\right\rangle ,\F_{k}}\left(n\right)$
and this quantity is given by some rational expression in $n$ (for
large enough $n$, say $n\geq\left|w\right|$, see Claim \claimref{rational-expressions}).
This expression can be expressed as a Taylor series in $\frac{1}{n}$,
so write
\[
\mathbb{E}\left[{\cal F}_{w,n}\right]=\sum_{s=0}^{\infty}\frac{a_{s}\left(w\right)}{n^{s}}
\]
where $a_{s}\left(w\right)\in\mathbb{R}$ (in fact these are integers:
see \cite[Claim 5.1]{Pud14a} and also the sequel of the current proof).
By Theorem \ref{thm:avg_fixed_pts}, $ $$a_{0}=1$, $a_{1}=a_{2}=\ldots=a_{\pi\left(w\right)-2}=0$
and $\alpha_{\pi\left(w\right)-1}=\left|\crit\left(w\right)\right|$
(unless $\pi\left(w\right)=1$ in which case $a_{0}=1+\left|\crit\left(w\right)\right|$).
So our goal here is to bound the remaining coefficients $a_{s}\left(w\right)$
for $s\geq\pi\left(w\right)$.

The discussion in Section \ref{sub:M=0000F6bius-Inversions} yields
the following equalities:
\begin{eqnarray*}
\mathbb{E}\left[{\cal F}_{w,n}\right] & = & \Phi_{\left\langle w\right\rangle ,\F_{k}}\left(n\right)=\sum_{N:\,\left\langle w\right\rangle \alg N\leq\F_{k}}R_{\left\langle w\right\rangle ,N}\left(n\right)=\\
 & = & \sum_{M,N:\,\left\langle w\right\rangle \covers M\covers N}C_{M,N}\left(n\right)=\sum_{M:\,\left\langle w\right\rangle \covers M}\sum_{N:\, M\covers N}C_{M,N}\left(n\right)
\end{eqnarray*}
From Proposition \ref{prop:C_MN} we obtain that for a fixed $M$,
\begin{eqnarray*}
\sum_{N:\, M\covers N}C_{M,N}\left(n\right) & = & \frac{1}{n^{\mathrm{rk}\left(M\right)-1}}\sum_{r\in\mathbb{N}}\left(-1\right)^{r}\sum_{\substack{\sigma_{0}\in\mathrm{Sym}\left(V_{M}\right)\\
\sigma_{1},\ldots,\sigma_{r}\in\mathrm{Sym}\left(E_{M}\right)\setminus\left\{ id\right\} 
}
}\left(\frac{-1}{n}\right)^{\left\Vert \sigma_{0}\right\Vert +\ldots+\left\Vert \sigma_{r}\right\Vert }.
\end{eqnarray*}
For every $q\geq0$ define the following set:
\begin{equation}
\mathcal{P}_{M,q}=\left\{ \left(\sigma_{0},\ldots,\sigma_{r}\right)\,\middle|\,\begin{matrix}r\in\mathbb{N},\:\sigma_{0}\in\mathrm{\Sym}\left(V_{M}\right)\\
\sigma_{1},\ldots,\sigma_{r}\in\Sym\left(E_{M}\right)\backslash\left\{ \mathrm{id}\right\} \vphantom{\Big|}\\
\left\Vert \sigma_{0}\right\Vert +\ldots+\left\Vert \sigma_{r}\right\Vert =q
\end{matrix}\right\} ,\label{eq:P_Mq}
\end{equation}
so that
\[
\sum_{N:\, M\covers N}C_{M,N}\left(n\right)=\frac{1}{n^{\mathrm{rk}\left(M\right)-1}}\sum_{q=0}^{\infty}\frac{\left(-1\right)^{q}}{n^{q}}\sum_{\left(\sigma_{0},\ldots,\sigma_{r}\right)\in{\cal P}_{M,q}}\left(-1\right)^{r}.
\]
Hence, 
\begin{equation}
a_{s}\left(w\right)=\sum_{i=1}^{s+1}\sum_{\substack{M:\,\left\langle w\right\rangle \covers M\\
\rk\left(M\right)=i
}
}\left(-1\right)^{s-\left(i-1\right)}\sum_{\substack{\left(\sigma_{0},\ldots,\sigma_{r}\right)\in{\cal P}_{M,s-\left(i-1\right)}}
}\left(-1\right)^{r}.\label{eq:a_s-equality}
\end{equation}
In what follows we ignore the alternating signs of the summands in
(\ref{eq:a_s-equality}) and bound $\left|a_{s}\left(w\right)\right|$
by 
\begin{equation}
\left|a_{s}\left(w\right)\right|\leq\sum_{i=1}^{s+1}\sum_{\substack{M:\,\left\langle w\right\rangle \covers M\\
\rk\left(M\right)=i
}
}\left|{\cal P}_{M,s-\left(i-1\right)}\right|.\label{eq:a_s_ineq}
\end{equation}
\textbf{Claim: }For every $M\fg\F_{k}$ with $\left\langle w\right\rangle \covers M$,
we have $\left|{\cal P}_{M,q}\right|\leq t^{2q}$.\\
\textbf{Proof of Claim: }Fix $M$ and denote $b_{q}=\left|{\cal P}_{M,q}\right|$.
Clearly, $b_{0}=1$, and we proceed by induction on $q$. Let $q\geq1$.
We split the set ${\cal P}_{M,q}$ by the value of $\sigma_{r}$.
For $r=0$ there are at most
\[
\left|\left\{ \sigma\in\mathrm{Sym}\left(V_{M}\right)\,\middle|\,\left\Vert \sigma\right\Vert =q\right\} \right|\leq\binom{\left|V_{m}\right|}{2}^{q}\leq\binom{t}{2}^{q}\leq\frac{t^{2q}}{2^{q}}
\]
elements with $r=0$. (For the middle inequality note that $\left|V_{M}\right|\leq\left|V_{\left\langle w\right\rangle }\right|\leq t$;
this is also the case with the edges: $\left|E_{M}\right|\leq\left|E_{\left\langle w\right\rangle }\right|\leq t$.)
For $r\geq1$, $\sigma_{r}$ is a permutation of the set of edges
$E_{M}$ and given $\sigma_{r}$, the number of options for $\sigma_{0},\ldots,\sigma_{r-1}$
is exactly $b_{q-\left\Vert \sigma_{r}\right\Vert }$. By the induction
hypothesis we obtain:
\begin{eqnarray*}
b_{q} & \leq & \frac{t^{2q}}{2^{q}}+\sum_{\sigma_{r}\in\mathrm{Sym}\left(E_{M}\right)\setminus\left\{ id\right\} }b_{q-\left\Vert \sigma_{r}\right\Vert }=\frac{t^{2q}}{2^{q}}+\sum_{\alpha=1}^{q}b_{q-\alpha}\left|\left\{ \sigma\in\mathrm{Sym}\left(E_{M}\right)\,\middle|\,\left\Vert \sigma\right\Vert =\alpha\right\} \right|\\
 & \leq & \frac{t^{2q}}{2^{q}}+\sum_{\alpha=1}^{q}t^{2q-2\alpha}\frac{t^{2\alpha}}{2^{\alpha}}=t^{2q}.\qed
\end{eqnarray*}

We proceed with the proof of the proposition. For a given $w\in\left(X\cup X^{-1}\right)^{t}$
there are at most $\binom{\left|V_{\left\langle w\right\rangle }\right|}{2}^{\beta}\leq\binom{t}{2}^{\beta}$
partitions of norm $\beta$ of $V_{\left\langle w\right\rangle }$,
and so at most $\binom{t}{2}^{\beta}$ subgroups $M$ of rank $\beta$
with $\left\langle w\right\rangle \covers M$ (see Theorem \thmref{distance}).
Hence from \eqref{a_s_ineq} we obtain,
\begin{eqnarray*}
\left|a_{s}\left(w\right)\right| & \leq & \sum_{i=1}^{s+1}\binom{t}{2}^{i}t^{2\left(s-\left(i-1\right)\right)}\leq\sum_{i=1}^{s+1}\frac{t^{2i}}{2^{i}}\cdot t^{2\left(s-i+1\right)}\leq t^{2s+2}.
\end{eqnarray*}
Finally,
\begin{eqnarray*}
\left|\mathbb{E}\left[{\cal F}_{w,n}\right]-1-\frac{\left|\crit\left(w\right)\right|}{n^{\pi\left(w\right)-1}}\right| & = & \left|\sum_{s=\pi\left(w\right)}^{\infty}\frac{a_{s}\left(w\right)}{n^{s}}\right|\leq\sum_{s=\pi\left(w\right)}^{\infty}\frac{\left|a_{s}\left(w\right)\right|}{n^{s}}\\
 & \leq & \sum_{s=\pi\left(w\right)}^{\infty}\frac{t^{2s+2}}{n^{s}}=t^{2}\cdot\left(\frac{t^{2}}{n}\right)^{\pi\left(w\right)}\cdot\frac{n}{n-t^{2}}.
\end{eqnarray*}
This finishes the proof.
\end{proof}

\section{Completing the Proof for Regular Graphs\label{sec:Completing-the-Proof-d-reg}}

In this section we complete the proofs of Theorems \ref{thm:2sqrt(d-1)+1}
and \thmref{base-d-regular}. In addition, we explain (in Section
\ref{sub:the-gap}) the source of the gap between these results on
the one hand and Friedman's result and Conjecture \ref{conj:friedman}
on the other.

\subsection{Proof of Theorem \ref{thm:2sqrt(d-1)+1} for $d$ even}

We begin with the case of even $d$ in Theorem \ref{thm:2sqrt(d-1)+1}.
We show that a random $d$-regular graph $\Gamma$ on $n$ vertices
in the permutation model (a random $n$-covering of the bouquet with
$\frac{d}{2}$ loops) satisfies a.a.s.~$\lambda\left(\Gamma\right)<2\sqrt{d-1}+0.84$,
where $\lambda\left(\Gamma\right)$ is the largest non-trivial eigenvalue
of $A_{\Gamma}$. As explained in more details in Appendix \ref{sec:contiguity},
this yields the same result for a uniformly random $d$-regular simple
graph.

So let $d=2k$ and $n,t=t\left(n\right)$ be such that $n>t^{2}$
and $t$ is even. The base graph $\Omega$ is the bouquet with $k$
loops, so $\cpt\left(\Omega\right)=\left(X\cup X^{-1}\right)^{t}$.
By \eqref{bounding-lambda-with-thm}, Proposition \propref{controling-the-O}
and Claim \ref{claim:controlling-the-O-for-m=00003D0},
\begin{eqnarray*}
\mathbb{E}\left[\lambda\left(\Gamma\right)^{t}\right] & \leq & \sum_{w\in\left(X\cup X^{-1}\right)^{t}}\left(\mathbb{E}\left[{\cal F}_{w,n}\right]-1\right)=\\
 & = & \sum_{m=0}^{k}\sum_{\substack{w\in\left(X\cup X^{-1}\right)^{t}:\\
\pi\left(w\right)=m
}
}\left(\frac{\left|\crit\left(w\right)\right|}{n^{m-1}}+O\left(\frac{1}{n^{m}}\right)\right)\\
 & \leq & \sum_{m=0}^{k}\frac{1}{n^{m-1}}\sum_{\substack{w\in\left(X\cup X^{-1}\right)^{t}:\\
\pi\left(w\right)=m
}
}\left(\left|\crit\left(w\right)\right|+\frac{t^{2+2m}}{n-t^{2}}\right)\\
 & \leq & \left(1+\frac{t^{2+2k}}{n-t^{2}}\right)\sum_{m=0}^{k}\frac{1}{n^{m-1}}\sum_{\substack{w\in\left(X\cup X^{-1}\right)^{t}:\\
\pi\left(w\right)=m
}
}\left|\crit\left(w\right)\right|
\end{eqnarray*}
Let $\varepsilon>0$. For $m\in\left\{ 0,1,\ldots,k\right\} $, Corollary
\corref{non-reduced-words_upper_bound} (for $m\geq1$) and Claim
\claimref{bound-for-m=00003D0} (for $m=0$) yield that for large
enough t, 
\[
\sum_{\substack{w\in\left(X\cup X^{-1}\right)^{t}:\\
\pi\left(w\right)=m
}
}\left|\crit\left(w\right)\right|\leq\left[g\left(2m-1\right)+\varepsilon\right]^{t},
\]
where $g\left(\cdot\right)$ is defined as in \eqref{g} with an extended
domain:
\[
g\left(2m-1\right)=\begin{cases}
2\sqrt{d-1} & 2m-1\in\left[-1,\sqrt{d-1}\right]\\
2m-1+\frac{d-1}{2m-1} & 2m-1\in\left[\sqrt{d-1},d-1\right]
\end{cases}.
\]
Thus
\begin{eqnarray}
\mathbb{E}\left[\lambda\left(\Gamma\right)^{t}\right] & \leq & \left(1+\frac{t^{2+2k}}{n-t^{2}}\right)\sum_{m=0}^{k}\frac{\left[g\left(2m-1\right)+\varepsilon\right]^{t}}{n^{m-1}}\nonumber \\
 & \leq & \left(1+\frac{t^{2+2k}}{n-t^{2}}\right)\cdot\left(k+1\right)\cdot\nonumber \\
 &  & \cdot\left[\max\left\{ \begin{array}{c}
n^{1/t}\left[g\left(-1\right)+\varepsilon\right],g\left(1\right)+\varepsilon,\frac{g\left(3\right)+\varepsilon}{n^{1/t}}\ldots\\
\ldots,\frac{g\left(2k-3\right)+\varepsilon}{\left(n^{1/t}\right)^{k-2}},\frac{2k+\varepsilon}{\left(n^{1/t}\right)^{k-1}}
\end{array}\right\} \right]^{t}\label{eq:final1}
\end{eqnarray}
Recall that $\Gamma$ is a random graph on $n$ vertices. In order
to obtain the best bound, $t$ needs to be chosen to minimize the
maximal summand in the r.h.s.~of \eqref{final1}. This requires $t=\theta\left(\log n\right)$:
if $t$ is larger than that, the last elements are unbounded, and
if $t$ is smaller than that, the first element is unbounded. Thus,
in particular, $\left(1+\frac{t^{2+2k}}{n-t^{2}}\right)=1+o_{n}\left(1\right)$.
We show that for every $d$ there is some constant $c=c\left(d\right)$,
such that if $t$ is chosen so that $n^{1/t}\thickapprox c$, then
all $k+1$ elements in the set in the r.h.s.~of \eqref{final1} are
strictly less than $2\sqrt{d-1}+0.835$ (for small enough $\varepsilon$).
Thus, for large enough $t$, $\mathbb{E}\left[\lambda\left(\Gamma\right)^{t}\right]\leq\left[2\sqrt{d-1}+0.835\right]^{t}$.
A standard application of Markov's inequality then shows that $\mathrm{Prob}\left[\lambda\left(\Gamma\right)<2\sqrt{d-1}+0.84\right]\underset{n\to\infty}{\to}1$.

Indeed, for $d\geq26$, one can set $n^{1/t}=e^{\frac{2}{5\sqrt{d-1}}}$.
Simple analysis shows that for $d\geq26$, $e^{\frac{2}{5\sqrt{d-1}}}<1+\frac{5}{12\sqrt{d-1}}$,
so the element corresponding to $m=0$ is at most $2\sqrt{d-1}\cdot e^{\frac{2}{5\sqrt{d-1}}}<2\sqrt{d-1}\left(1+\frac{5}{12\sqrt{d-1}}\right)=2\sqrt{d-1}+\frac{5}{6}<2\sqrt{d-1}+0.835$.
This first element is clearly larger than all other elements $ $corresponding
to $m$ such that $2m-1\leq\sqrt{d-1}$. Among all other values of
$m$, the maximal element is obtained when $2m-1\approx4.55\sqrt{d-1}$,
but its value is bounded from above by $1.94\sqrt{d-1}+0.4$ (again,
by simple analysis). For all remaining $d's$ ($4,6,\ldots,24$),
it can be checked case by case that choosing $n^{1/t}$ so that $n^{1/t}\cdot2\sqrt{d-1}=2\sqrt{d-1}+0.8$
works (and see the table in Section \subref{From-even-to-odd}).~$\qed$

\subsection{From even $d$ to odd $d$\label{sub:From-even-to-odd}}

In this subsection we derive the statement of Theorem \thmref{2sqrt(d-1)+1}
for $d$ odd from the now established statement for $d$ even. We
showed that for $d$ even we have a.a.s.~$\lambda\left(\Gamma\right)<2\sqrt{d-1}+0.84$.
The idea is that every upper bound applying to some value of $d$
also applies to $d-1$. 

As explained in Appendix \secref{contiguity}, by contiguity results
from \cite{GJKW02}, it is enough to show the $2\sqrt{d-1}+1$ upper
bound for random graphs $\Gamma$ in a random model denoted ${\cal G}_{n,d}^{*}$
(the result for random simple graphs then follows immediately). 
\begin{claim}
Let $d\ge3$ be odd. Assume that a random $\left(d+1\right)$-regular
graph $\Gamma$ in the permutation model satisfies a.a.s.~$\lambda\left(\Gamma\right)<C$.
Then a random $d$-regular graph $\Gamma$ in ${\cal G}_{n,d}^{*}$
also satisfies a.a.s.~$\lambda\left(\Gamma\right)<C$.\end{claim}
\begin{proof}
Let $\Gamma$ be a random $d$-regular graph in ${\cal G}_{n,d}^{*}$.
By (\cite[Theorem 1.3]{GJKW02}, the permutation model ${\cal P}_{n,d+1}$
is contiguous to the distribution on $\left(d+1\right)$-regular graphs
obtained by considering $\Gamma$ and adding a uniformly random perfect
matching $m$. (As $d$ is odd, the number of vertices $n$ in $\Gamma$
is necessarily even.) Denote by $\hat{\Gamma}$ the random graph obtained
this way. It is enough to show that $\lambda\left(\hat{\Gamma}\right)\ge\lambda\left(\Gamma\right)-o_{n}\left(1\right)$
with probability tending to $1$ as $n\to\infty$.

Indeed, let $\mu$ be the eigenvalue of $\Gamma$ whose absolute value
is largest (so $\lambda\left(\Gamma\right)=\left|\mu\right|$), and
let $f\in\ell^{2}\left(V\left(\Gamma\right)\right)$ be a corresponding
real eigenfunction with $\left\Vert f\right\Vert =1$. In particular,
$\sum_{v\in V\left(\Gamma\right)}f\left(v\right)=0$ and $\sum_{v\in V\left(\Gamma\right)}f\left(v\right)^{2}=1$.
We have 
\[
\lambda\left(\hat{\Gamma}\right)\ge\left\langle A_{\hat{\Gamma}}f,f\right\rangle =\left\langle A_{\Gamma}f,f\right\rangle +2\sum_{e\in m}f\left(e^{+}\right)f\left(e^{-}\right)=\mu+2\sum_{e\in m}f\left(e^{+}\right)f\left(e^{-}\right),
\]
where the summation is over all edges $e$ in the random perfect matching
$m$, and $e^{+}$ and $e^{-}$ mark the two endpoints of $e$. Let
$R$ denote the random summation $2\sum_{e\in m}f\left(e^{+}\right)f\left(e^{-}\right)$.
We finish by showing that $R$ is generally very small. 

To accomplish that we use standard identities involving symmetric
polynomials over $f\left(v_{1}\right),\ldots,f\left(v_{n}\right)$.
Let $p_{k}=\sum_{v}f\left(v\right)^{k}$ be the $k$'th symmetric
Newton polynomial, so $p_{1}=0$ and $p_{2}=1$. Moreover, since $\left|f\left(v\right)\right|<1$
for every $v$, $\left|p_{k}\right|<p_{2}=1$. We use the fact that
every symmetric polynomial is a polynomial in the $p_{k}$'s and is
thus bounded. 

To begin with, 
\[
\mathbb{E}\left[R\right]=n\cdot\frac{1}{\binom{n}{2}}\sum_{\left\{ u,v\right\} \in\binom{V}{2}}f\left(v\right)f\left(u\right)=\frac{2}{n-1}s_{2}\left(f\left(v_{1}\right),\ldots,f\left(v_{n}\right)\right),
\]
where $s_{2}$ is the second elementary symmetric function: $s_{2}\left(x_{1},\ldots,x_{n}\right)=\sum_{i<j}x_{i}x_{j}$.
Since $s_{2}=\frac{1}{2}\left(p_{1}^{2}-p_{2}\right)=-\frac{1}{2}$,
we conclude that $\mathbb{E}\left[R\right]=-\frac{1}{n-1}=o_{n}\left(1\right)$.

Similarly,
\begin{eqnarray*}
\mathbb{E}\left[R^{2}\right] & = & 4\cdot\frac{n}{2}\cdot\frac{1}{\binom{n}{2}}\sum_{\left\{ u,v\right\} \in\binom{V}{2}}f\left(v\right)^{2}f\left(u\right)^{2}+8\cdot\binom{n/2}{2}\cdot\frac{1}{\binom{n}{4}}\sum_{\left\{ u,v,w,x\right\} \in\binom{V}{4}}f\left(u\right)f\left(v\right)f\left(w\right)f\left(x\right),\\
 & = & \frac{4}{n-1}\sum_{\left\{ u,v\right\} \in\binom{V}{2}}f\left(v\right)^{2}f\left(u\right)^{2}+\frac{48}{\left(n-1\right)\left(n-3\right)}\sum_{\left\{ u,v,w,x\right\} \in\binom{V}{4}}f\left(u\right)f\left(v\right)f\left(w\right)f\left(x\right).
\end{eqnarray*}
Since the two summations here are symmetric polynomials, they are
bounded, and thus $\mathbb{E}\left[R^{2}\right]=o_{n}\left(1\right)$
and so is the variance of $R$. Thus $R=o_{n}\left(1\right)$ with
probability tending to $1$ as $n\to\infty$. 
\end{proof}
If $d\ge3$ is odd, we can thus use our bound for $d+1$ to obtain
that a.a.s.
\[
\lambda\left(\Gamma\right)<2\sqrt{\left(d+1\right)-1}+0.84=2\sqrt{d}+0.84\approx2\sqrt{d-1}+\frac{1}{\sqrt{d}}+0.84.
\]
This proves our result for large enough $d$. Indeed, for $d\ge41$,
$2\sqrt{d}+0.84<2\sqrt{d-1}+1$. 

For smaller values of odd $d$ we use tighter results for $d+1$.
For example, we seek the smallest constant $c$ for which a bound
of $2\sqrt{4-1}+c$ can be obtained for $4$-regular graphs in our
methods. In order to minimize $\max\left\{ n^{1/t}\cdot2\sqrt{d-1},2\sqrt{d-1},\frac{4}{n^{1/t}}\right\} $
(see \eqref{final1}), we choose $n^{1/t}=\sqrt{\frac{4}{2\sqrt{d-1}}}$
to get an upper bound of $3.723$ (compared with $2\sqrt{d-1}=3.464$,
so here $c\thickapprox0.259$). For $d=3$ this bound is useless (it
is larger than the trivial bound of $3$).

The following table summarizes the bounds we obtain for $d\le20$
in the scenario of Theorem $1$. This can be carried on to establish
Theorem \thmref{2sqrt(d-1)+1} for $d\le40$.\\

\begin{tabular}{|c|c|c|c|c|c|c|}
\hline 
$d$ & Upper Bound & $c$ in $2\sqrt{d-1}+c$ & $n^{1/t}$ & $d$ & Uppder Bound & $c$ in $2\sqrt{d-1}+c$\tabularnewline
\hline 
\hline 
$4$ & 3.723 & $0.259$ & $1.075$ & $3$ & 3 & $0.172$\tabularnewline
\hline 
$6$ & $4.933$ & $0.460$ & $1.103$ & $\negthickspace\negthickspace\negthickspace\negthickspace\Longrightarrow~~~5$ & $4.933$ & $0.933$\tabularnewline
\hline 
$8$ & $5.868$ & $0.576$ & $1.109$ & $\negthickspace\negthickspace\negthickspace\negthickspace\Longrightarrow~~~7$ & $5.868$ & $0.969$\tabularnewline
\hline 
$10$ & $6.646$ & $0.646$ & 1.108 & $\negthickspace\negthickspace\negthickspace\negthickspace\Longrightarrow~~~9$ & 6.646 & 0.989\tabularnewline
\hline 
$12$ & 7.323 & $0.689$ & 1.104 & $\negthickspace\negthickspace\negthickspace\negthickspace\Longrightarrow~~~11$ & 7.323 & 0.998\tabularnewline
\hline 
$14$ & 7.928 & $0.7169$ & 1.099 & $\negthickspace\negthickspace\negthickspace\negthickspace\Longrightarrow~~~13$ & 7.928 & 0.9998\tabularnewline
\hline 
$16$ & 8.482 & $0.7352$ & 1.095 & $\negthickspace\negthickspace\negthickspace\negthickspace\Longrightarrow~~~15$ & 8.482 & 0.999\tabularnewline
\hline 
$18$ & 8.994 & $0.747$ & 1.091 & $\negthickspace\negthickspace\negthickspace\negthickspace\Longrightarrow~~~17$ & 8.994 & 0.994\tabularnewline
\hline 
$20$ & 9.473 & $0.755$ & 1.087 & $\negthickspace\negthickspace\negthickspace\negthickspace\Longrightarrow~~~19$ & 9.473 & 0.988\tabularnewline
\hline 
\end{tabular}\medskip{}

\begin{rem}
\textbf{\label{remark:Odd-ds}} Of course, the method presented here
to derive the statement of Theorem \thmref{2sqrt(d-1)+1} for odd
$d$'s from the statement for even $d$'s works only because of the
small additive constant we have in the result. To obtain a tight result
(Friedman's Theorem) in our approach, we will need another method
to work with odd $d$'s. 

One plausible direction is as follows. We may construct a random $d$-regular
graph with $d$ odd using $k=\frac{d-1}{2}$ random permutations plus
one random perfect matching. If we label the edges corresponding to
the perfect matching by $b$, and orient the edges corresponding to
the permutations and label them by $a_{1},\ldots,a_{k}$, the graphs
become Schreier graphs of subgroups of $\F_{k}*\nicefrac{\mathbb{Z}}{2\mathbb{Z}}=\left\langle a_{1},\ldots,a_{k},b\,\middle|\, b^{2}=1\right\rangle $.
It is conceivable that the machinery we developed for the free group
(and especially, Theorem \thmref{avg_fixed_pts}) can be also developed
for this kind of free products.
\end{rem}

\subsection{Proof of Theorem \thmref{base-d-regular}\label{sub:Proof-of-Theorem-base-regular}}

\noindent The only change upon the previous case (Theorem \thmref{2sqrt(d-1)+1}
with $d$ even) is that the summation in \eqref{final1} over the
primitivity rank $m$ does not stop at $k=\frac{d}{2}$ but continues
until $\rk\left(\Omega\right)=\left|V\left(\Omega\right)\right|\left(\frac{d}{2}-1\right)+1$.
However, when $m>k$, it follows from Corollary \corref{bounds-for-non-red-words-and-reg-base-graph}
that the corresponding term inside the $\max$ operator is $\frac{d}{\left(n^{1/t}\right)^{m-1}}$
which is strictly less than $\frac{d}{\left(n^{1/t}\right)^{d/2-1}}$
(for every choice of $t$ and $n$), but this latter term is already
there in \eqref{final1}. Thus, the maximal term is remained unchanged,
and we obtain the same bound overall as in the even case of Theorem
\thmref{2sqrt(d-1)+1}, namely $2\sqrt{d-1}+0.84$.

Let us stress that in this case the proof as is works for all $d\ge3$
(odd and even alike). As before, for small $d$'s we can obtain better
bounds, even if $d$ is odd. For example, for $d=3$ one can obtain
an upper bound of $\sqrt{3\cdot2\sqrt{d-1}}\approx2.913$.

\subsection{The source of the gap\label{sub:the-gap}}

It could be desirable to use the approach presented in this paper
and replace the constant $1$ in Theorem \ref{thm:2sqrt(d-1)+1} with
an arbitrary $\varepsilon>0$, to obtain Friedman's tight result.
Unfortunately, this is still beyond our reach. It is possible, however,
to point out the source of the gap and how it may be potentially overcome.

In the first inequality in our proof (as outlined in Section \ref{sec:Overview-of-the-proof}),
we bound $\left\{ \mathbb{E}\left[\lambda\left(\Gamma\right)^{t}\right]\right\} ^{1/t}$
by $\left\{ \mathbb{E}\left[\sum_{\mu\in\mathrm{Spec}\left(A_{\Gamma}\right)\setminus\left\{ d\right\} }\mu^{t}\right]\right\} ^{1/t}$.
Since 
\[
\lambda\left(\Gamma\right)^{t}\le\sum_{\mu\in\mathrm{Spec}\left(A_{\Gamma}\right)\setminus\left\{ d\right\} }\mu^{t}\le n\cdot\lambda\left(\Gamma\right)^{t}=\left[n^{1/t}\cdot\lambda\left(\Gamma\right)\right]^{t},
\]
as long as $t=\theta\left(\log n\right)$ the loss here is bounded,
and if $t\gg\log n$ we lose nothing. 

On the other hand, if $t\ll\log n$, one cannot obtain anything: It
is known (e.g.~\cite[Corollary 1]{GZ99}) that for every $\delta>0$
there exists $0<\varepsilon<1$ such that at least $\varepsilon\cdot n$
of the eigenvalues of $\Gamma$ satisfy $\left|\mu\right|\geq\rho-\delta$
(here $\rho=2\sqrt{d-1}$). If $t\in o\left(\log n\right)$ then $n^{1/t}$
tends to infinity, and thus 
\[
\left\{ \sum_{\mu\in\mathrm{Spec}\left(A_{\Gamma}\right)\setminus\left\{ d\right\} }\mu^{t}\right\} ^{1/t}>\left\{ \varepsilon n\left(\rho-\delta\right)^{t}\right\} ^{1/t}\underset{n\to\infty}{\to}\infty.
\]

Our proof proceeds by bounding this $t$-th moment of the non-trivial
spectrum. Let us stress that as long as $t=t\left(n\right)$ is small
enough in terms of $n$ so that the error term in Proposition \ref{prop:controling-the-O}
is negligible ($t=o\left(n^{1/\left(2+2k\right)}\right)$ suffices),
the upper bound our technique yields for $\mathbb{E}\left[\sum_{\mu\in\mathrm{Spec}\left(A_{\Gamma}\right)\setminus\left\{ d\right\} }\mu^{t}\right]$
is tight. In particular, for large enough $d$, and $t\approx c\log n$
with a suitable constant $c=c\left(d\right)$, 
\[
\left\{ \mathbb{E}\left[\sum_{\mu\in\mathrm{Spec}\left(A_{\Gamma}\right)\setminus\left\{ d\right\} }\mu^{t}\right]\right\} ^{1/t}\approx2\sqrt{d-1}+0.84.
\]
To see why, note that all relevant steps of the proof yield equalities
or tight bounds: the second step, which relies on Theorem \thmref{avg_fixed_pts},
has only equalities so it is surely tight. In the third step, we prove
that the error term is $o_{n}\left(1\right)$ for every $w$ of length
$t$ (note the proof bounds the absolute value of the error term).
As mentioned above, the bound we have in the fourth step for the exponential
growth rate of $\sum_{w\in\left(X\cup X^{-1}\right)^{t}:\,\pi\left(w\right)=m}\left|\crit\left(w\right)\right|$
~is, in fact, the correct value (see Theorem \thmref{prim-rank-distr}).
In the final, fifth step we may tighten our calculation in order to
come closer to the real constant (slightly smaller than $0.84$),
but we cannot improve it considerably.\\

What is, then, the source of this gap? It seems, therefore, that the
reason the bound we get for $\lambda\left(\Gamma\right)$ is not tight
lies in rare events that enlarge $\mathbb{E}\left[\lambda\left(\Gamma\right)^{t}\right]$
substaintially. For example, in the permutation model every vertex
of $\Gamma$ is isolated with probability $\frac{1}{n^{k}}$, so overall
there are on average $\frac{1}{n^{k-1}}$ isolated vertices. Each
such vertex is responsible to an additional eigenvalue $d$, alongside
the trivial one. These rare events alone contribute $\frac{1}{n^{k-1}}\cdot d^{t}$
to $\mathbb{E}\left[\lambda\left(\Gamma\right)^{t}\right]$. For example,
for $d=4$ ($k=2$) and $n^{1/t}\approx1.075$ as in the table in
Section \subref{From-even-to-odd}, isolated vertices contribute about
$\left[\frac{4}{\left(1.075\right)}\right]^{t}\approx3.721^{t}$ to
$\mathbb{E}\left[\lambda\left(\Gamma\right)^{t}\right]$, which is
roughly the bound we obtain in this case.

There are other, slightly more complicated, rare events that contribute
much to $\mathbb{E}\left[\lambda\left(\Gamma\right)^{t}\right]$.
Consider, for instance, the event that when $d=4$ the random graph
$\Gamma$ contains the subgraph ~~~~~~$\xymatrix@1@C=15pt{\bullet\ar@{-}@(dl,ul)&\bullet\ar@{-}@(dr,ur)\ar@{-}[l]}$~~~~~~.
If this subgraph is completed to a $4$-regular graph by attaching
a tree to each vertex, its spectral radius becomes $3.5$. Since this
resulting graph topologically covers (the connected component of the
subgraph in) $\Gamma$, we get a non-trivial eigenvalue which is at
least $3.5$ (but normally very close to $3.5$). On average, there
are $\frac{2}{n}$ such subgraphs in $\Gamma$, so they contribute
about $\frac{2}{n}\cdot3.5^{t}$ to $\mathbb{E}\left[\lambda\left(\Gamma\right)^{t}\right]$.
When $n^{1/t}$ is small enough, this is strictly larger than $\left[2\sqrt{d-1}\right]^{t}\approx3.464^{t}$.

Each such small graph corresponds to a few particular subgroups of
$\F_{k}$. For example, the subgraph ~~~~~~~~$\xymatrix@1@C=15pt{\bullet\ar@{-}@(dl,ul)&\bullet\ar@{-}@(dr,ur)\ar@{-}[l]}$~~~~~~~~
corresponds to one of four subgroups, one of which is $\xymatrix@1@C=15pt{\bullet\ar@(dl,ul)[]^{x_1}&\otimes\ar@(dr,ur)[]_{x_1}\ar[l]_{x_2}}$.One
therefore needs to realize which all these {}``bad'' subgroups are,
show their overall {}``probability'' is small (the average number
of appearances of $H\le\F_{k}$ in $\Gamma$ is exactly $L_{H,\F_{k}}$),
and somehow omit their contribution to $\mathbb{E}\left[\lambda\left(\Gamma\right)^{t}\right]$.
This would be relatively easy were our analysis of $\mathbb{E}\left[{\cal F}_{w}\right]$
based on $\mathbb{E}\left[{\cal F}_{w}\right]=\sum_{M\in\XC{\left\langle w\right\rangle }{\F_{k}}}L_{M,\F_{k}}$.
However, it is based, instead, on $\mathbb{E}\left[{\cal F}_{w}\right]=\sum_{N\in\XC{\left\langle w\right\rangle }{\F_{k}}}R_{\left\langle w\right\rangle ,N}$
(see Section \secref{Controlling-the-O}). It seems that overcoming
this difficulty requires a better control over the error term: this
might enable us to omit the contribution of these {}``bad'' subgroups
from our bounds.
\begin{rem}
These {}``bad'', rare events are somewhat parallel to the notion
of \emph{tangles} in \cite{Fri08}.
\end{rem}

\section{Completing the Proof for Arbitrary Graphs\label{sec:Completing-the-Proof-general-Omega}}

The completion of the proof of Theorem \ref{thm:sqrt3-times-rho}
is presented in this Section. We begin with the proof of the first
statement of the theorem which concerns the spectrum of the adjacency
operator of $\Gamma$, the random $n$-covering of the fixed base
graph $\Omega$. The variations needed in order to establish the statement
about the Markov operator are described in Section \ref{sub:The-Markov-Operator}.

Recall that $\rho=\rho_{A}\left(\Omega\right)$ denotes the spectral
radius of the adjacency operator of the covering tree. Our goal now
is to prove that for every $\varepsilon>0$, $\lambda_{A}\left(\Gamma\right)$,
the largest absolute value of a non-trivial eigenvalue of the adjacency
operator $A_{\Gamma}$, satisfies asymptotically almost surely 
\begin{equation}
\lambda_{A}\left(\Gamma\right)<\sqrt{3}\cdot\rho+\varepsilon.\label{eq:adj-upp-bnd}
\end{equation}
As in the proof of Theorem \ref{thm:2sqrt(d-1)+1} (the beginning
of Section \ref{sec:Completing-the-Proof-d-reg}), let $n,t=t\left(n\right)$
be so that $n>t^{2}$ and $t$ is even. Using \eqref{bounding-lambda-with-thm},
Proposition \propref{controling-the-O}, Claim \ref{claim:controlling-the-O-for-m=00003D0}
and Lemma \ref{lem:prim-rank-at-most-rk-Omega}, one obtains
\begin{eqnarray*}
\mathbb{E}\left[\lambda_{A}\left(\Gamma\right)^{t}\right] & \leq & \sum_{w\in\cpt\left(\Omega\right)}\left(\mathbb{E}\left[{\cal F}_{w}\right]-1\right)=\\
 & \leq & \left(1+\frac{t^{2+2\rk\left(\Omega\right)}}{n-t^{2}}\right)\sum_{m=0}^{\rk\left(\Omega\right)}\frac{1}{n^{m-1}}\sum_{w\in\cptm\left(\Omega\right)}\left|\crit\left(w\right)\right|
\end{eqnarray*}
Let $\varepsilon>0$. From Theorem \ref{thm:bound-for-m-in-general-Omega}
and Lemma \ref{claim:bound-for-m=00003D0-general-Omega} it follows
now that for $t$ even and large enough, 
\begin{eqnarray}
\mathbb{E}\left[\lambda_{A}\left(\Gamma\right)^{t}\right] & \leq & \left(1+\frac{t^{2+2\rk\left(\Omega\right)}}{n-t^{2}}\right)\left[n\cdot\left[\rho+\varepsilon\right]^{t}+\sum_{m=1}^{\rk\left(\Omega\right)}\frac{\left[\left(2m-1\right)\cdot\rho+\varepsilon\right]^{t}}{n^{m-1}}\right].\nonumber \\
 & \leq & \left(1+\frac{t^{2+2\rk\left(\Omega\right)}}{n-t^{2}}\right)\left(1+\rk\left(\Omega\right)\right)\cdot\nonumber \\
 &  & \cdot\left[\max\left\{ n^{1/t}\left[\rho+\varepsilon\right],\rho+\varepsilon,\frac{3\rho+\varepsilon}{n^{1/t}},\frac{5\rho+\varepsilon}{\left(n^{1/t}\right)^{2}},\ldots,\frac{\left(2\rk\left(\Omega\right)-1\right)\rho+\varepsilon}{\left(n^{1/t}\right)^{\rk\left(\Omega\right)-1}}\right\} \right]^{t}\label{eq:choose-the-max}
\end{eqnarray}
Again, to obtain a bound we must have $t\in\theta\left(\log n\right)$,
and the best bound we can obtain in this general case is obtained
by choosing $n^{1/t}\approx\sqrt{3}$ , so $\left(1+\frac{t^{2+2\rk\left(\Omega\right)}}{n-t^{2}}\right)^{1/t}\underset{n\to\infty}{\to}1$,
and the maximal value inside the set in \eqref{choose-the-max} is
then $\sqrt{3}\left(\rho+\epsilon\right)$. Again, a standard application
of Markov inequality finishes the proof. $\qed$

\subsection{The spectrum of the Markov operator \label{sub:The-Markov-Operator}}

After establishing the first statement of Theorem \thmref{sqrt3-times-rho},
we want to explain how the proof should be modified to apply to $\lambda_{M}\left(\Gamma\right)$,
the maximal absolute value of a non-trivial eigenvalue of the \emph{Markov}
operator on $\Gamma$. The goal is to show that for every $\varepsilon>0$

\begin{equation}
\lambda_{M}\left(\Gamma\right)<\sqrt{3}\cdot\rho_{M}\left(\Omega\right)+\varepsilon\label{eq:mrkv-upp-bnd}
\end{equation}
asymptotically almost surely. 

As we note in Appendix \secref{Operators-on-Non-Regular}, the Markov
operator is given by $B_{\Gamma}D_{\Gamma}^{-1}$, where $B_{\Gamma}$
is the adjacency matrix and $D_{\Gamma}$ the diagonal matrix with
the degrees of vertices in the diagonal. This is conjugate to and
thus share the same spectrum with $Q_{\Gamma}=D_{\Gamma}^{-1/2}B_{\Gamma}D_{\Gamma}^{-1/2}$,
but the latter has the advantage of being symmetric, so we work with
it.

The $\left(u,v\right)$ entry of $Q_{\Gamma}$ equals $\frac{1}{\sqrt{\deg\left(u\right)\deg\left(v\right)}}$
times the number of edges between $u$ and $v$. For every walk $w$
in $\Gamma$ we assign a weight function $f\left(w\right)$ as follows:
if $w$ starts at $v_{0}$, then visits $v_{1},v_{2},\ldots,v_{t-1}$
and ends at $v_{t}$, then $f\left(w\right)=\frac{1}{\sqrt{\deg v_{0}}\cdot\deg v_{1}\cdot\ldots\cdot\deg v_{t-1}\cdot\sqrt{\deg v_{t}}}$.
It is easy to see that $\left[Q_{\Gamma}^{\,\, t}\right]_{u,v}$ equals
the sum of $f\left(w\right)$ over all walks $w$ of length $t$ from
$u$ to $v$, and thus 
\[
\sum_{\lambda\in\mathrm{Spec}\left(M_{\Gamma}\right)}\lambda^{t}=\mathrm{tr}M_{\Gamma}=\sum_{w\in\cpt\left(\Gamma\right)}f\left(w\right).
\]
Moreover, note that when a walk from the covering $\Gamma$ projects
to the base graph $\Omega$, its weight does not change. Using this
fact, we can imitate step I from Section \secref{Overview-of-the-proof}
to obtain, for $t$ even,
\begin{eqnarray*}
\lambda_{M}\left(\Gamma\right)^{t} & \leq & \sum_{\mu\in\mathrm{Spec}\left(M_{\Gamma}\right)}\mu^{t}-\sum_{\mu\in\mathrm{Spec}\left(M_{\Omega}\right)}\mu^{t}=\sum_{w\in\cpt\left(\Gamma\right)}f\left(w\right)-\sum_{w\in\cpt\left(\Omega\right)}f\left(w\right)=\\
 & = & \sum_{w\in\cpt\left(\Omega\right)}f\left(w\right)\left[{\cal F}_{w,n}\left(\sigma_{1},\ldots,\sigma_{k}\right)-1\right].
\end{eqnarray*}
The second and third steps remain the same, obtaining 
\[
\mathbb{E}\left[\lambda_{M}\left(\Gamma\right)^{t}\right]\leq\left(1+\frac{t^{2+2\rk\left(\Omega\right)}}{n-t^{2}}\right)\sum_{m=0}^{\rk\left(\Omega\right)}\frac{1}{n^{m-1}}\sum_{w\in\cptm\left(\Omega\right)}f\left(w\right)\left|\crit\left(w\right)\right|.
\]
The next modification needs take place in the fourth step, where instead
of bounding $\sum_{w\in\cpt\left(\Omega\right):\,\pi\left(w\right)=m}\left|\crit\left(w\right)\right|$,
one needs to bound $\sum_{w\in\cpt\left(\Omega\right):\,\pi\left(w\right)=m}f\left(w\right)\left|\crit\left(w\right)\right|$.
But the exact same proofs work if we merely replace $\rho_{A}\left(\Omega\right)$
with $\rho_{M}\left(\Omega\right)$. Theorem \thmref{bound-for-m-in-general-Omega}
becomes
\begin{equation}
\limsup_{t\to\infty}\left[\sum_{w\in\cptm\left(\Omega\right)}f\left(w\right)\left|\crit\left(w\right)\right|\right]^{1/t}\leq\left(2m-1\right)\cdot\rho_{M}\left(\Omega\right).\label{eq:bound-for-words-markov-oper}
\end{equation}
and likewise, Lemma \claimref{bound-for-m=00003D0-general-Omega}
becomes 
\[
\limsup_{t\to\infty}\left[\sum_{w\in\cpt^{0}\left(\Omega\right)}f\left(w\right)\right]^{1/t}=\rho_{M}\left(\Omega\right).
\]
Similarly, the definition of $\beta_{t}\left(N\right)$ (preceding
Claim \claimref{ CR-subgroups}) should be modified to 
\[
\beta_{t}\left(N\right)=\sum_{w\in\cpt\left(\Omega\right):\,\left\langle w\right\rangle \neqalg N}f\left(w\right)
\]
and in the proof of Claim \claimref{ CR-subgroups} one should use
the fact that $f\left(w\right)$ does not change when the closed walk
$w$ is being cyclically rotated. Finally, in the proof of Proposition
\propref{key-prop-for-top-graph} we sometimes replace a walk with
its inverse and use the symmetry of the operator. This is the reason
for working with $Q_{\Gamma}$ rather than with $M_{\Gamma}$. Also,
the coefficient $\left|V\left(\Omega\right)\right|$ from the statement
of the proposition needs be replaced with some constant function of
the degrees of all vertices.

Because the bounds in \eqref{bound-for-words-markov-oper} are exactly
those in Theorem \thmref{bound-for-m-in-general-Omega} only with
$\rho_{M}\left(\Omega\right)$ instead of $\rho_{A}\left(\Omega\right)$,
the final step of the proof (which appears in Section \secref{Completing-the-Proof-general-Omega})
also remains unchanged.

\section{The Distribution of Primitivity Ranks\label{sec:distr-of-prim-rank}}

In this subsection we show that the upper bounds from Proposition
\ref{prop:reduced-words-upper-bound} and Corollary \ref{cor:non-reduced-words_upper_bound}
are the accurate exponential growth rates of the number of words (reduced
or not) and critical subgroups with a given primitivity rank. This
is not needed for the proof of the main results of this paper. However,
it does show that in the proof of Theorem \ref{thm:2sqrt(d-1)+1},
the fourth step of the proof, where words and critical subgroups are
counted, yields a tight bound. Thus, the origin of the gap between
our result and Friedman's lies elsewhere (see Section \ref{sub:the-gap}).

First, let us recall a theorem due to the author and Wu which counts
primitive words in $\F_{k}$.
\begin{thm}
\label{thm:prim-words-growth}\cite{PW14} For every $k\geq3$, let
$p_{k}\left(t\right)$ denote the number of primitive words of length
$t$ in $\F_{k}$. Then, 
\[
\lim_{t\to\infty}p_{k}\left(t\right)^{1/t}=2k-3.
\]

\end{thm}
For $\F_{2}$ it is known that this exponential growth rate equals
$\sqrt{3}$ (\cite{rivin2004remark}). These results show that the
portion of primitive words among all words of length $t$ decays exponentially
fast%
\footnote{That primitive words in $\F_{k}$ are negligible in this sense follows
also from the earlier results \cite{BV02counting}, \cite[Thm 10.4]{BMS02}
and \cite{shpilrain2005counting}, where the exponential growth rate
from Theorem \ref{thm:prim-words-growth} is shown to be $\leq2k-2-o_{k}\left(1\right)$.%
}. They are used in the following theorem, which states that the upper
bounds from Proposition \ref{prop:reduced-words-upper-bound} are
accurate. 
\begin{thm}
\label{thm:prim-rank-distr-on-reduced-words}Let $k\geq2$ and $m\in\left\{ 1,2,\ldots,k\right\} $.
Let\marginpar{$c_{k,m}\left(t\right)$} 
\[
c_{k,m}\left(t\right)=\left|\left\{ w\in\F_{k}\,\middle|\,\left|w\right|=t,\,\pi\left(w\right)=m\right\} \right|.
\]
Then, 

\begin{equation}
\limsup_{t\to\infty}c_{k,m}\left(t\right)^{1/t}=\begin{cases}
\sqrt{2k-1} & 2m-1\le\sqrt{2k-1}\\
2m-1 & 2m-1\ge\sqrt{2k-1}
\end{cases}.\label{eq:reduced-words-limit}
\end{equation}

\end{thm}
In fact, as the proof shows, for $m\ge2$, we can replace the $\limsup$
with regular $\lim$, and for $m=1$ we can replace $\limsup_{t\to\infty}c_{k,1}\left(t\right)^{1/t}$
with $\lim_{t\to\infty}c_{k,1}\left(2t\right)^{1/2t}$.
\begin{cor}
A generic word in $\F_{k}$ has primitivity rank $k$. \end{cor}
\begin{proof}
{[}of Theorem \ref{thm:prim-rank-distr-on-reduced-words}{]} The r.h.s.~of
(\ref{eq:reduced-words-limit}) is an upper bound for the $\limsup$
by Proposition \ref{prop:reduced-words-upper-bound}. It remains to
show that for every $m\in\left\{ 1,\ldots,k\right\} $, there is some
subset of words with primitivity rank $m$ and growth rate $\max\left(\sqrt{2k-1},2m-1\right)$.

Consider first the case $2m-1>\sqrt{2k-1}$. Take any subset of the
generators $S\subseteq X$ of size $m$ and consider the subgroup
$H=\F\left(S\right)$. Its core graph is a bouquet of $m$ loops.
The number of words of length $t$ in $H$ is $2m\cdot\left(2m-1\right)^{t-1}$.
By Theorem \ref{thm:prim-words-growth}, a random word in $H$ of
length $t$ is a.a.s.~non-primitive in $H$, so its primitivity rank
is at most $m$. On the other hand, the exponential growth rate of
all words with $\pi\left(w\right)<m$ combined is smaller than $\left(2m-1\right)$
(by Proposition \ref{prop:reduced-words-upper-bound}). Thus, a word
$w\in H$ of length $t$ satisfies $\pi\left(w\right)=m$ a.a.s.,
and we are done. In particular, we proved that for such values of
$m$, 
\[
\limsup_{t\to\infty}c_{k,m}\left(t\right)^{1/t}=\lim_{t\to\infty}c_{k,m}\left(t\right)^{1/t}=2m-1.
\]

Now assume that $2m-1\le\sqrt{2k-1}$. Consider subgroups of the form
$H=\left\langle x_{1},\ldots,x_{m-1},u\right\rangle $ where $u$
is a cyclically reduced word of length $\sim\frac{t}{2}$ such that
its first and last letters are \emph{not} one of $\left\{ x_{1}^{\pm1},\ldots,x_{m-1}^{\pm1}\right\} $.
Then, $\Gamma_{X}\left(H\right)$ has the form of a bouquet of $m-1$
small loops of size 1 and one large loop of size $\sim\frac{t}{2}$.
Now consider the word $w=w\left(u\right)=x_{1}^{\,2}x_{2}^{\,2}\ldots x_{m-1}^{\,\,\,\,\,\,\,\,\,\,2}u^{2}$.
Obviously, the growth rate of the number of possible $u$'s (as a
function of $t$) is $\sqrt{2k-1}$, hence also the growth rate of
the number of different $w$'s. It can be shown that $w$ is \emph{not
}primitive in $H$, using the primitivity criterion from Theorem \ref{thm:distance}
(\cite[Thm 1.1]{Pud14a}). (In fact, it follows from \cite[Lemma 6.8]{Pud14a}
that as an element of the free group $H$, $w$ has primitivity rank
$m$ with $H$ being the sole $w$-critical subgroup.) Thus, $\pi\left(w\right)\le m$.
In general, the primitivity rank might be strictly smaller. For example,
for $m=3$ and $u=x_{3}x_{1}^{2}x_{2}^{2}x_{3}$, we have $\pi\left(w\right)=2$
because $w$ is not a proper power yet is not primitive in $\left\langle x_{3},x_{1}^{2}x_{2}^{2}\right\rangle $.
However, we claim that for a generic $u$, the primitivity rank of
$w$ is exactly $m$.

Indeed, if this is not the case, then there is some $\tilde{m}<m$
such that the growth rate of words $w=w\left(u\right)$ as above with
$\pi\left(w\right)=\tilde{m}$ is $\sqrt{2k-1}$. By the proof of
Proposition \ref{prop:reduced-words-upper-bound} and especially (\ref{eq:bound-for-reduced-words-1})
, it follows that most of these words ($w=w\left(u\right)$ with $\pi\left(w\right)=\tilde{m}$)
have an algebraic extension N of rank $\tilde{m}$ such that the number
of edges in $\Gamma_{X}\left(N\right)$ is close to $\frac{t}{2}$.
(By (\ref{eq:bound-for-reduced-words-1}), the total number of words
of length $t$ with an algebraic extension $N$ of rank $\tilde{m}$
and $\delta t$ edges in $\Gamma_{X}\left(N\right)$, for some $\delta<\frac{1}{2}$,
grows strictly slower than $\sqrt{2k-1}$.) So almost all these words
$w=w\left(u\right)$ trace twice every edge of some $\Gamma_{X}\left(N\right)$
of rank $\tilde{m}$ with roughly $\frac{t}{2}$ edges. In particular,
each such $w=w\left(u\right)$ traces twice some topological edge
in $\Gamma_{X}\left(N\right)$ of length at least $\frac{1}{2\left(3m-1\right)}t$.
This implies that there is some linear-size two overlapping subwords
of $u$ or of $u^{-1}$. But for a generic $u$, the longest subword
appearing twice in $u$ or in $u^{-1}$ has length of order $\log t$. 

Since the $w$'s we obtained are of arbitrary even length, this shows
that if $2m-1\le\sqrt{2k-1}$, then 
\[
\limsup_{t\to\infty}c_{k,m}\left(t\right)^{1/t}=\lim_{t\to\infty}c_{k,m}\left(2t\right)^{1/2t}=\sqrt{2k-1}.
\]
If, in addition, $m\ge2$, the same argument as above works also for
$w=w\left(u\right)=x_{1}^{\,3}x_{2}^{\,2}\ldots x_{m-1}^{\,\,\,\,\,\,\,\,\,\,2}u^{2}$
which is of arbitrary odd length. Thus, $\lim_{t\to\infty}c_{k,m}\left(t\right)^{1/t}=\sqrt{2k-1}.$\end{proof}
\begin{rem}
It follows from the proofs of Proposition \ref{prop:reduced-words-upper-bound}
and Theorem \ref{thm:prim-rank-distr-on-reduced-words} that while
for $2m-1>\sqrt{2k-1}$ the main source for words with $\pi\left(w\right)=m$
is in subgroups with core graphs of minimal size (and their conjugates),
the main source for $2m-1<\sqrt{2k-1}$ is in subgroups with core
graphs of maximal size, namely of size roughly $\frac{t}{2}$. 
\end{rem}
Recall that in the proof of Theorem \ref{thm:2sqrt(d-1)+1} we used
bounds on the number of \emph{not-necessarily-reduced} words (and
their critical subgroups). Here, too, the bounds from Corollary \ref{cor:non-reduced-words_upper_bound}
are accurate for every value of $m$:
\begin{thm}
\label{thm:prim-rank-distr}Let $k\geq2$ and $m\in\left\{ 0,1,2,\ldots,k,\infty\right\} $.
Let\marginpar{$b_{k,m}\left(t\right)$} 
\[
b_{k,m}\left(t\right)=\left|\left\{ w\in\left(X\cup X^{-1}\right)^{t}\,\middle|\,\pi\left(w\right)=m\right\} \right|.
\]
Then for $m=0$ we have 
\[
\lim_{\substack{t\to\infty\\
t\,\mathrm{even}
}
}b_{k,0}\left(t\right)^{1/t}=2\sqrt{2k-1}.
\]
For $m\in\left\{ 1,\ldots,k\right\} $, 
\[
\lim_{t\to\infty}b_{k,m}\left(t\right)^{1/t}=\begin{cases}
2\sqrt{2k-1} & 2m-1\leq\sqrt{2k-1}\\
2m-1+\frac{2k-1}{2m-1} & 2m-1\geq\sqrt{2k-1}
\end{cases}.
\]
Finally, for $m=\infty$ we have
\[
\lim_{t\to\infty}b_{k,\infty}\left(t\right)^{1/t}=2k-2+\frac{2}{2k-3}.
\]

\end{thm}
This shows, in particular, that as in the case of reduced words, a
generic word in $\left(X\cup X^{-1}\right)^{t}$ is of primitivity
rank $k$, namely, the share of words with this property tends to
$1$ as $t\to\infty$. It also shows that for every $m$, the growth
rate of the number of words with primitivity rank $m$ is equal to
the growth rate of the larger quantity of $\sum_{w\in\left(X\cup X^{-1}\right)^{t}:\,\pi\left(w\right)=m}\left|\crit\left(w\right)\right|$.
\begin{proof}
For $m=0$ this is (the proof of) Claim \ref{claim:bound-for-m=00003D0}
(evidently, there are no odd-length words reducing to 1). For $1\leq m$
with $2m-1\leq\sqrt{2k-1}$ the same proof (as in Claim \ref{claim:bound-for-m=00003D0})
can be followed as long as we present at least one even-length and
one odd-length words with primitivity rank $m$. And indeed, as mentioned
above (and see \cite[Lemma 6.8]{Pud14a}), $\pi\left(x_{1}^{\,2}x_{2}^{\,2}\ldots x_{m}^{\,2}\right)=\pi\left(x_{1}^{\,3}x_{2}^{\,2}\ldots x_{m}^{\,2}\right)=m$.
If $2m-1>\sqrt{2k-1}$, the statement follows from the statements
on reduced words (Theorems \ref{thm:prim-rank-distr-on-reduced-words}
and \ref{thm:prim-words-growth}) and an application of the extended
cogrowth formula \cite{Pud15+} (here a bit more elaborated results
from \cite{Pud15+}, not mentioned in Theorem \ref{thm:cogrowth-formula-extended},
are required).
\end{proof}
The statements of the last theorem are summarized in Table \ref{tab:Primitivity-Rank-and-fixed-points}.

\section{Open Questions}

We end with some open problems that suggest themselves from this paper:
\begin{itemize}
\item Can one obtain a better control over the error term in Theorem \ref{thm:avg_fixed_pts}?
This would probably require not ignoring the alternating signs in
\eqref{a_s-equality}. As explained in Section \subref{the-gap},
this may be the seed to closing the gap in the result of Theorem \thmref{2sqrt(d-1)+1}.
\item Is it possible to generalize the techniques in this paper (and even
more so the ones from \cite{PP15}) to odd values of $d$? (See Remark
\ref{remark:Odd-ds}).
\item Can one obtain the accurate exponential growth rate of the number
of not-necessarily-reduced words with a given primitivity rank in
a general base graph $\Omega$, thus improving the statements of Theorems
\ref{thm:bound-for-m-in-general-Omega} and \thmref{sqrt3-times-rho}?
This may require some sort of clever extension of the cogrowth formula
that applies to non-regular graphs (there have been a few attempts
in this aim, see e.g.~\cite{Bar99,Nor04,AFH07}, but see limitations
in \cite{Pud15+}).
\item Several classic results from the theory of expansion in graphs were
generalized lately to simplicial complexes of dimension greater than
one (see e.g.~\cite{GW12,PRT12,Lub13}). In particular, a parallel
of Alon-Boppana Theorem is presented in \cite{PR12}. Is there a parallel
to Alon's conjecture in this case? Can the methods of the current
paper be extended to higher dimensions?
\end{itemize}

\section*{Acknowledgments}

We would like to thank Nati Linial and Ori Parzanchevski for their
valuable suggestions and useful comments. We would also like to thank
Miklós Abért, Noga Alon, Itai Benjamini, Ron Rosenthal and Nick Wormald
for their beneficial comments.

\section*{Late Remark\label{sec:Late-Remark}}

Slightly over a year after this manuscript was written and submitted,
Friedman and Kohler wrote \cite{friedman2014relativized}, where they
prove an asymptotic probabilistic upper bound of $2\sqrt{d-1}+\varepsilon$
for $\lambda\left(\Gamma\right)$, where $\Gamma$ is a random covering
of an arbitrary \emph{$d$-regular} base graph $\Omega$. They improve
Friedman's former techniques from \cite{Fri08} to apply to this more
general case. This bound is tight and improves on the statements from
Theorem \ref{thm:base-d-regular} and Corollary \ref{cor:bipartite}
in the current paper. It is claimed in \cite{friedman2014relativized}
that at present, they are unable to make their techniques apply to
the most general case of an arbitrary (not necessarily regular) base
graph $\Omega$.

\section*{Appendices}

\begin{appendices}

\section{Contiguity and Related Models of Random Graphs\label{sec:contiguity}}

\subsection*{Random $d$-regular graphs }

In this paper, the statement of Theorem \thmref{2sqrt(d-1)+1} is
first proved for the permutation model of random $d$-regular graphs
with $d$ even. We then derive Theorem \thmref{2sqrt(d-1)+1}, stated
for the uniform distribution on all $d$-regular \emph{simple }graphs
on $n$ vertices with $d$ even or odd, using results of Wormald \cite{Wor99}
and Greenhill et al.~\cite{GJKW02}. These works show the \emph{contiguity
}(see footnote \vpageref{fn:contiguity}) of different models of random
regular graphs.

In particular, they describe the following model: consider $dn$ labeled
points, with $d$ points in each of $n$ buckets, and take a random
perfect matching of the points. Letting the buckets be vertices and
each pair represent an edge, one obtains a random $d$-regular graph.
This model is denoted ${\cal G}_{n,d}^{*}$. It is shown \cite[Theorem 1.3]{GJKW02}
that ${\cal G}_{n,d}^{*}$ is contiguous to the permutation model
${\cal P}_{n,d}$ (for $d$ even). If $\Gamma$ is a random $d$-regular
graph in ${\cal G}_{n,d}^{*}$, the event that $\Gamma$ is a simple
graph (with no loops nor multiple edges) has positive probability,
bounded away from 0. Moreover, within this event, simple graphs are
distributed uniformly%
\footnote{To be precise, \emph{vertex-labeled} simple graphs are distributed
uniformly in this event. Unlabeled simple graphs have probability
proportional to the order of their automorphism group. Then again,
for $d\geq3$, this group is a.a.s.~trivial, so the result of Theorem
\ref{thm:2sqrt(d-1)+1} applies both to the uniform model of labeled
graph and to the uniform model of unlabeled graphs.%
}. Thus, for even values of $d$, Theorem \thmref{2sqrt(d-1)+1} follows
from the corresponding result for the permutation model. The derivation
of the odd case also uses contiguity results, as explained in Section
\subref{From-even-to-odd}.

\subsection*{Random $d$-regular bipartite graphs}

As an immediate corollary from Theorem \ref{thm:base-d-regular} we
deduced that a random $d$-regular bipartite graph is {}``nearly
Ramanujan'' in the sense that besides its two trivial eigenvalues
$\pm d$, all other eigenvalues are at most $2\sqrt{d-1}+0.84$ in
absolute value a.a.s.~(Corollary \ref{cor:bipartite}). Our proof
works in the model $C_{n,\Omega}$ (here $\Omega$ is the graph with
$2$ vertices and $d$ parallel edges connecting them). However, by
the results of \cite{Ben74}, the probability that our graph has no
multiple edges is bounded away from zero (asymptotically it is $e^{-\binom{d}{2}}$).
Thus, our result applies also to the model of $d$ random \emph{disjoint
}perfect matchings between two sets of $n$ vertices. This model,
in turn, is contiguous to the uniform model of bipartite (vertex-labeled)
$d$-regular simple graphs (for $d\geq3$: see \cite[Section 4]{MRRW97}%
\footnote{In fact, there is an explicit proof there only for $d=3$. To derive
the general case, one can show that a random $\left(d+1\right)$-regular
graph is contiguous to a random $d$-regular bipartite graph plus
one edge-disjoint random matching (following, e.g., the computations
in \cite{BM86}). We would like to thank Nick Wormald for helpful
private communications surrounding this point.%
}), so our result applies in the latter model as well.

\subsection*{Random coverings of a fixed graph}

In Theorem \ref{thm:sqrt3-times-rho} we consider random $n$-coverings
of a fixed graph $\Omega$ in the model ${\cal C}_{n,\Omega}$, where
a uniform random permutation is generated for every edge of $\Omega$.
An equivalent model is attained if we cover some spanning tree of
$\Omega$ by $n$ disjoint copies and then choose a random permutation
for every edge outside the tree (that is, the same automorphism-types
of non-labeled graphs are obtained with the same distribution). In
fact, picking a basepoint $\otimes\in V\left(\Omega\right)$, there
is yet another description for this model: The classification of $n$-sheeted
coverings of $\Omega$ by the action of $\pi_{1}\left(\Omega,\otimes\right)$
on the fiber $\left\{ \otimes\right\} \times\left[n\right]$ above
$\otimes$ shows that ${\cal C}_{n,\Omega}$ is equivalent to choosing
uniformly at random an action of the free group $\pi_{1}\left(\Omega,\otimes\right)$
on $\left\{ \otimes\right\} \times\left[n\right]$. 

A different but related model uses the classification of \emph{connected},
pointed coverings of $\left(\Omega,\otimes\right)$ by the corresponding
subgroups of $\pi_{1}\left(\Omega,\otimes\right)$. A random $n$-covering
is thus generated by choosing a random subgroup of index $n$. However,
it seems that this model is contiguous to ${\cal C}_{n,\Omega}$ if
$\rk\left(\Omega\right)\ge2$ (note that the random covering $\Gamma$
in ${\cal C}_{n,\Omega}$ is a.a.s.~connected provided that $\mathrm{rk\left(\Omega\right)\geq2}$).
Indeed, the only difference is that in the new model, the probability
of every connected graph $\Gamma$ from ${\cal C}_{n,\Omega}$ is
proportional to $\frac{1}{\left|\mathrm{Aut}\left(\Gamma\right)\right|}$.
When $\mathrm{rk\left(\Omega\right)\geq2}$, it seems that a.a.s.~$\left|\mathrm{Aut}\left(\Gamma\right)\right|=1$,
which would show that our result applies to this model as well.

Finally, there is another natural model that comes to mind: given
a periodic infinite tree, namely a tree that covers some finite graph,
one can consider a random (simple) graph $\Gamma$ with $n$ vertices
covered by this tree (with uniform distribution among all such graphs
with $n$ vertices, for suitable $n$'s only). One can then analyze
$\lambda\left(\Gamma\right)$, the largest absolute value of an eigenvalue
besides%
\footnote{Leighton showed that two finite graphs with a common covering share
also some common finite covering \cite{leighton1982finite}. It follows
that all finite quotients of the same tree share the same Perron-Frobenius
eigenvalue.%
} $\pf\left(\Gamma\right)$. (This generalizes the uniform model on
$d$-regular graphs.) Occasionally, all the quotients of some given
periodic tree $ $$T$ cover the same finite {}``minimal'' graph
$\Omega$. Interestingly, Lubotzky and Nagnibeda \cite{LN98} showed
that there exist such $T$'s with a minimal quotient $\Omega$ which
is \emph{not }Ramanujan (in the sense that $\lambda\left(\Omega\right)$
is strictly larger than $\rho\left(T\right)$, the spectral radius
of $T$). Since all the quotients of $T$ inherit the eigenvalues
of $\Omega$, their $\lambda\left(\cdot\right)$ is also bounded away
from $\rho\left(T\right)$ (from above). Hence, the corresponding
version of Conjecture \ref{conj:friedman} is false in this general
setting.

\section{Spectral Expansion of Non-Regular Graphs\label{sec:Operators-on-Non-Regular}}

In this section we provide some background on the theory of expansion
of irregular graphs, describing how spectral expansion is related
to other measurements of expansion (combinatorial expansion, random
walks and mixing). This further motivates the claim that Theorem \ref{thm:sqrt3-times-rho}
shows that if the base graph $\Omega$ is a good (nearly optimal)
expander, then a.a.s.~so are its random coverings. We would like
to thank Ori Parzanchevski for his valuable assistance in writing
this appendix.

The spectral expansion of a (non-regular) graph $\Gamma$ on $m$
vertices is measured by some function on its spectrum, and most commonly
by the \emph{spectral gap}: the difference\emph{ }between the largest
eigenvalue and the second largest. As mentioned above, it is not apriori
clear which operator best describes in spectral terms the properties
of the graph. There are three main candidates (see, e.g.~\cite{GW12}),
all of which are bounded%
\footnote{All operators considered here are bounded provided that the degree
of vertices in $\Gamma$ is bounded. This is the case in all the graphs
considered in this paper.%
}, self-adjoint operators and so have real spectrum:
\begin{enumerate}
\item The\emph{ adjacency operator} $A_{\Gamma}$ on $\left(\ell^{2}\left(V\left(\Gamma\right)\right),1\right)$%
\footnote{Here, $\left(\ell^{2}\left(V\left(\Gamma\right)\right),1\right)$
stands for $\ell^{2}$-functions on the set of vertices $V\left(\Gamma\right)$
with the standard inner product: $\left\langle f,g\right\rangle =\sum_{v}f\left(v\right)\overline{g\left(v\right)}$;
In the summation $\sum_{w\sim v}$, each vertex $w$ is repeated with
multiplicity equal to the number of edges between $v$ and $w$.%
}: 
\[
(A_{\Gamma}f)(v)=\sum_{w\sim v}f\left(w\right)
\]
If $\Gamma$ is finite this operator is represented in the standard
basis by the adjacency matrix, and its spectral radius is the \emph{Perron-Frobenius}
eigenvalue $\pf\left(\Gamma\right)$. The spectrum in this case is
\[
\pf\left(\Gamma\right)=\lambda_{1}\geq\lambda_{2}\geq\ldots\geq\lambda_{m}\geq-\pf\left(\Gamma\right),
\]
and the spectral gap is $\pf\left(\Gamma\right)-\lambda\left(\Gamma\right)$,
where $\lambda\left(\Gamma\right)=\max\left\{ \lambda_{2},-\lambda_{n}\right\} $%
\footnote{Occasionally, the spectral gap is taken to be $\pf\left(\Gamma\right)-\lambda_{2}\left(\Gamma\right)$.%
}. The spectrum of $A_{\Gamma}$ was studied in various works, for
instance \cite{Gre95,LN98,Fri03,LP10}.
\item The \emph{averaging Markov operator }$M_{\Gamma}$ on $\left(\ell^{2}\left(V\left(\Gamma\right)\right),\deg\left(\cdot\right)\right)$%
\footnote{Here, $\left(\ell^{2}\left(V\left(\Gamma\right)\right),\deg\left(\cdot\right)\right)$
stands for $l^{2}$-functions on the set of vertices $V\left(\Gamma\right)$
with the inner product: $\left\langle f,g\right\rangle =\sum_{v}f\left(v\right)\overline{g\left(v\right)}\deg\left(v\right)$.%
}:
\[
(M_{\Gamma}f)(v)=\frac{1}{\deg\left(v\right)}\sum_{w\sim v}f\left(w\right)
\]
This operator is given by $D_{\Gamma}^{-1}A_{\Gamma}$, and its spectrum
is contained in $\left[-1,1\right]$. The eigenvalue $1$ corresponds
to locally-constant functions when $\Gamma$ is finite, and in this
case the spectrum is
\[
1=\mu_{1}\geq\mu_{2}\geq\ldots\geq\mu_{m}\geq-1.
\]
The spectral gap is then $1-\mu\left(\Gamma\right)$ here $\mu\left(\Gamma\right)=\max\left\{ \mu_{2},-\mu_{m}\right\} $.
Up to a possible affine transformation, the spectrum of $M_{\Gamma}$
is the same as the spectrum of the simple \emph{random walk operator}
($A_{\Gamma}D_{\Gamma}^{-1}$) or of one of the \emph{normalized Laplacian
}operators\emph{ }($I-A_{\Gamma}D_{\Gamma}^{-1}$ or $I-D_{\Gamma}^{-1/2}A_{\Gamma}D_{\Gamma}^{-1/2}$).
This spectrum is considered for example in \cite{Sin93,Chu97,GZ99}.
\item The \emph{Laplacian operator }$\Delta_{\Gamma}^{+}$ on $\left(\ell^{2}\left(V\left(\Gamma\right)\right),1\right)$:
\[
\left(\Delta_{\Gamma}^{+}f\right)\left(v\right)=\deg\left(v\right)f\left(v\right)-\sum_{w\sim v}f\left(w\right)
\]
The Laplacian equals $D_{\Gamma}-A_{\Gamma}$, where $D_{\Gamma}$
is the diagonal operator $\left(D_{\Gamma}f\right)\left(v\right)=\deg\left(v\right)\cdot f\left(v\right)$.
The entire spectrum is non-negative, with $0$ corresponding to locally-constant
functions when $\Gamma$ is finite. In the finite case, the spectrum
is 
\[
0=\nu_{1}\leq\nu_{2}\leq\ldots\leq\nu_{m},
\]
the spectral gap being $\nu_{2}-\nu_{1}=\nu_{2}$. The Laplacian operator
is studied e.g. in \cite{AM85}.
\end{enumerate}
For a regular graph $\Gamma$, all different operators are identical
up to an affine shift. However, in the general case there is no direct
connection between the three different spectra. In this paper we consider
the spectra of $A_{\Gamma}$ and of $M_{\Gamma}$. At this point we
do not know how to extend our results to the Laplacian operator $\Delta_{\Gamma}^{+}$. 

The spectrum of all three operators is closely related to different
notions of expansion in graphs. The adjacency operator, for example,
has the following version of \emph{the expander} \emph{mixing lemma}:
for every two subsets $S,T\subseteq V\left(\Gamma\right)$ (not necessarily
disjoint), one has
\[
\left|E\left(S,T\right)-\pf\left(\Gamma\right)\mathrm{vol}_{\pf}\left(S\right)\mathrm{vol}_{\pf}\left(T\right)\right|\leq\lambda\left(\Gamma\right)\frac{\sqrt{\left|S\right|\cdot\left|T\right|}}{m},
\]
where $\mathrm{vol}_{\pf}\left(S\right)=\left\langle \mathfrak{1}_{S},f_{\pf}\left(\Gamma\right)\right\rangle $
and $f_{\pf}\left(\Gamma\right)$ is the (normalized) Perron-Frobenius
eigenfunction. This is particularly useful in the ${\cal C}_{n,\Omega}$
model since the $f_{\pf}\left(\Gamma\right)$ is easily obtained from
the Perron-Frobenius eigenfunction of $\Omega$ by 
\[
f_{\pf}\left(\Gamma\right)=\frac{1}{\sqrt{n}}f_{\pf}\left(\Omega\right)\circ\pi.
\]
In the $d$-regular case, this amounts to the usual mixing lemma:
$\left|E\left(S,T\right)-d\frac{\left|S\right|\cdot\left|T\right|}{m}\right|\leq\lambda\left(\Gamma\right)\sqrt{\left|S\right|\cdot\left|T\right|}$.
If one takes $T=V\setminus S$, one can attain a bound on the Cheeger
constant of $\Gamma$ (see (\ref{eq:cheeger-def})).

As for the averaging Markov operator, it is standard that $\mu\left(\Gamma\right)$
controls the speed in which a random walk converges to the stationary
distribution. In addition, if one defines $\mathrm{deg}\left(S\right)$
to denote the sum of degrees of the vertices in $S$, then 
\[
\left|E\left(S,T\right)-\frac{\mathrm{deg}\left(S\right)\mathrm{deg}\left(T\right)}{2\left|E\left(\Gamma\right)\right|}\right|\leq\mu\left(\Gamma\right)\sqrt{\mathrm{deg}\left(S\right)\mathrm{deg}\left(T\right)}.
\]
Moreover, consider the \emph{conductance} of $\Gamma$
\[
\phi\left(\Gamma\right)=\min_{{\emptyset\ne S\subseteq V\atop \mathrm{deg}\left(S\right)\leq\frac{\mathrm{deg}\left(V\right)}{2}}}\frac{\left|E\left(S,V\setminus S\right)\right|}{\mathrm{deg}\left(S\right)}.
\]
Then the following version of the Cheeger inequality holds \cite[Lemmas 2.4, 2.6]{Sin93}:
\[
\frac{\phi^{2}\left(\Gamma\right)}{2}\leq1-\mu_{2}\leq2\phi\left(\Gamma\right).
\]

Finally, the spectrum of the Laplacian operator is related to the
standard \emph{Cheeger Constant} of $\Gamma$, defined as 
\begin{equation}
h\left(\Gamma\right)=\min_{{\emptyset\ne S\subseteq V\atop \left|S\right|\leq\frac{\left|V\right|}{2}}}\frac{\left|E\left(S,V\setminus S\right)\right|}{\left|S\right|}.\label{eq:cheeger-def}
\end{equation}
By the so-called {}``discrete Cheeger inequality'' \cite{AM85}:
\[
\frac{h^{2}\left(\Gamma\right)}{2k}\leq\nu_{2}\leq2h\left(\Gamma\right)
\]
with $k$ being the largest degree of a vertex. In addition, one has
a variation on the mixing lemma for $\Delta_{\Gamma}^{+}$ as well
\cite[Thm 1.4]{PRT12}.

\end{appendices}

\bibliographystyle{amsalpha} \bibliographystyle{amsalpha}
\bibliography{united,self}

\end{document}